\documentclass{amsart}
\usepackage{shortsalch, xy, eufrak}
\usepackage{cancel}
\usepackage{hyperref}
\usepackage{cleveref}
\xyoption{all}
\title{$\ell$-adic topological Jacquet-Langlands duality.}
\date{October 2023.}
\makeatletter
\let\@wraptoccontribs\wraptoccontribs
\makeatother
\author{Andrew Salch}
\contrib[with an appendix by]{Andrew Salch and Matthias Strauch}
\begin{document}
\begin{abstract}
We embed the Lubin-Tate tower into a larger tower of formal schemes, the {\em degenerating Lubin-Tate tower.} We construct a topological realization of the degenerating Lubin-Tate tower, i.e., a compatible family of presheaves of $E_{\infty}$-ring spectra on the \'{e}tale site of each formal scheme in the degenerating Lubin-Tate tower, which agrees on the base of the tower with the Goerss--Hopkins presheaf on Lubin-Tate space. 

We define and prove basic properties of nearby cycle and vanishing cycle presheaves of spectra on formal schemes. We apply these constructions to our spectrally-enriched degenerating Lubin-Tate tower to produce, for every spectrum $X$, an ``$\ell$-adic topological Jacquet-Langlands (TJL) dual'' of $X$. We prove that there is a correspondence between: 1. certain irreducible representations of $\Aut(\mathbb{G})$ occurring in $(\mathcal{K}(E(\mathbb{G}))^{\wedge}_{\ell})_*(X)$, and 2. certain supercuspidal irreducible representations of $GL_n$ occuring in the rational homotopy groups of the TJL dual of $X$. Here $E(\mathbb{G})$ is the Morava $E$-theory spectrum of a height $n$ formal group $\mathbb{G}$, and $\mathcal{K}(E(\mathbb{G}))^{\wedge}_{\ell}$ is its algebraic $K$-theory spectrum completed away from the characteristic of the ground field of $\mathbb{G}$. 

Finally, we prove that at height $1$, TJL duality {\em preserves the $L$-factors}. This means that the automorphic $L$-factor of the $GL_1$-representation associated to $(\mathcal{K}(E(\mathbb{G}))^{\wedge}_{\ell})_*(X)$ by TJL duality is precisely the $p$-local Euler factor in a meromorphic $L$-function whose special values in the left half-plane recover the orders of the $KU$-local stable homotopy groups of $X$. 
\end{abstract}
\maketitle

\tableofcontents

\section{Introduction.}

In this paper, all formal groups and formal modules will be assumed to be one-dimensional. The symbols $\ell,p$ shall always denote {\em distinct} primes.

\subsection{Background on the LT tower and the Jacquet-Langlands correspondence.}
\label{Background on the Lubin-Tate tower...}

Given a formal group $\mathbb{G}$ of finite height $n$ over a perfect field $k$ of characteristic $p$, the {\em Lubin-Tate tower} is a sequence of morphisms of formal schemes
\begin{equation*}\label{lt tower 0} \dots \rightarrow \Def(\mathbb{G})_{\lvl p^2} \rightarrow \Def(\mathbb{G})_{\lvl p} \rightarrow \Def(\mathbb{G}),\end{equation*}
where $\Def(\mathbb{G})$ is the classical Lubin-Tate space \cite{MR0238854}, i.e., the moduli space of deformations of $\mathbb{G}$. The formal scheme $\Def(\mathbb{G})_{\lvl p^m}$ is the moduli space of deformations of $\mathbb{G}$ equipped with Drinfeld level $p^m$-structures \cite{MR0384707}. We recall the full definitions below, in \cref{Review of the Lubin-Tate tower.}.

Consider the $\ell$-adic cohomology\footnote{We are being purposefully vague about what ``$\ell$-adic cohomology'' means, here, since it is a bit subtle: one has to use either the $\ell$-adic vanishing cycle cohomology, or pass to the Raynaud generic fibre and take the {\em compactly supported} \'{e}tale cohomology of the resulting rigid analytic space. These two constructions yield the same cohomology, by results of Berkovich \cite{MR1395723}. We give a brief review of relevant ideas in \cref{Review of the role...}, and further discussion in \cref{Nearby cycles and vanishing cycles...}, where we work out analogous constructions for sheaves of spectra rather than sheaves of abelian groups.} in dimension $n-1$ of the formal scheme $\Def(\mathbb{G})_{\lvl p^m}$, take the colimit as $m\rightarrow \infty$, and tensor up to the algebraic closure $\overline{\mathbb{Q}}_{\ell}$. Write $H^{n-1}$ as shorthand for the resulting $\overline{\mathbb{Q}}_{\ell}$-vector space. Then $H^{n-1}$ has a natural action of $GL_n(\hat{\mathbb{Z}}_p)$, given by permuting the Drinfeld level structures, as well as a natural action of $\Aut(\mathbb{G})$. It was conjectured by Carayol \cite{MR1044827} that this $\overline{\mathbb{Q}}_{\ell}$-linear representation of $GL_n(\hat{\mathbb{Z}}_p)\times \Aut(\mathbb{G})$ ``realizes'' the supercuspidal part of the Jacquet-Langlands correspondence\footnote{The $\ell$-adic Jacquet-Langlands correspondence is a correspondence between certain $\ell$-adic representations of $GL_n(\mathbb{Q}_p)$ and certain $\ell$-adic representations of $D_{1/n,\mathbb{Q}_p}^{\times}$. There is also the $\ell$-adic Langlands correpondence, which is a correspondence between certain $\ell$-adic representations of $GL_n(\mathbb{Q}_p)$ and certain $\ell$-adic representations of the Weil group of $\mathbb{Q}_p$, a certain large subgroup of the absolute Galois group $\Gal(\overline{\mathbb{Q}}_p/\mathbb{Q}_p)$. The $\ell$-adic Langlands correspondence is also realized in the cohomology of the Lubin-Tate tower, if one base-changes the tower up to $\overline{\mathbb{Q}}_p$; see \cref{Review of the role...} below for a brief account, or \cite{MR1876802} for a fuller account.}, in a precise sense, which we now sketch. The Jacquet-Langlands correspondence of Deligne-Kazhdan-Vigneras \cite{MR0771672} and of Rogawski \cite{MR0700135} pairs each irreducible essentially $L^2$-integrable complex representation $\pi$ of $GL_n(\mathbb{Q}_p)$ with an irreducible complex representation $\mathcal{JL}(\pi)$ of the unit group $D_{1/n,\mathbb{Q}_p}^{\times}$ of the Hasse invariant $1/n$ central division algebra over $\mathbb{Q}_p$. 
The endomorphism ring of $\mathbb{G}$ agrees with the maximal compact subring $\mathcal{O}_{D_{1/n,\mathbb{Q}_p}}$ of $D_{1/n,\mathbb{Q}_p}$; see for example Proposition 1.7 of \cite{MR0384707}.
The action of $\Aut(\mathbb{G}) \cong \mathcal{O}_{D_{1/n,\mathbb{Q}_p}}^{\times}\subseteq D_{1/n,\mathbb{Q}_p}^{\times}$ on $H^{n-1}$ extends to an action of $D_{1/n,\mathbb{Q}_p}^{\times}$, and the action of $GL_n(\hat{\mathbb{Z}}_p)$ on $H^{n-1}$ also extends\footnote{There is a very close relationship between smooth representations of $GL_n(\hat{\mathbb{Z}}_p)$ and smooth representations of $GL_n(\mathbb{Q}_p)$, since the former is the maximal compact open subgroup of the latter, and since the Iwasawa decomposition expresses $GL_n(\mathbb{Q}_p)$ as the product of $GL_n(\hat{\mathbb{Z}}_p)$ with the group of upper-triangular matrices. See section 1.5 of \cite{MR0546593} for the decomposition of every smooth representation of $GL_n(\mathbb{Q}_p)$ into the direct sum of its isotypic components, which are indexed by the compact open subgroups of $GL_n(\mathbb{Q}_p)$. It is also true that the restriction functor from smooth representations of $GL_n(\mathbb{Q}_p)$ to smooth representations of $GL_n(\hat{\mathbb{Z}}_p)$ has an exact {\em right} adjoint, by a dual version of Frobenius reciprocity, as in section 1.8 of \cite{MR0546593}. In this paper, we will switch between discussing representations of $GL_n(\mathbb{Q}_p)$ and representations of $GL_n(\hat{\mathbb{Z}}_p)$ whenever it is convenient for the exposition.

This is also an opportune moment to recall the definition of smoothness for representations. A representation $V$ of a $p$-adic group like $GL_n(\mathbb{Q}_p)$ or $GL_n(\hat{\mathbb{Z}}_p)$ is called {\em smooth} if, for every element $v\in V$, the subgroup stabilizing $v$ is open. This definition, and some others given in this paper, are standard in the representation theory of $p$-adic groups. However, we expect that the audience for this paper is likely to include topologists who know very little about representation theory of $p$-adic groups (a description which also fits the author). So we think it may be useful to provide these definitions.} to an action of $GL_n(\mathbb{Q}_p)$. Carayol's conjecture was essentially that, for supercuspidal\footnote{An irreducible smooth representation of $GL_n(\mathbb{Q}_p)$ is called {\em supercuspidal} if it does not occur as a subrepresentation of an induced representation from the Levi subgroup of a proper parabolic subgroup of $GL_n(\mathbb{Q}_p)$. For the sake of this paper, what matters is that the supercuspidal representations are the smooth irreducible representations which are generally hardest to construct and classify, and furthermore, by a theorem of Jacquet \cite{MR0291360} (cf. Theorem 2.5 of \cite{MR0546593}), {\em all} of the irreducible admissible representations of $GL_n(\mathbb{Q}_p)$ occur as subrepresentations of supercuspidals induced from various parabolic subgroups. The case for the difficulty and importance of constructing the supercuspidal representations is made strongly in \cite{MR0546593} and in \cite{MR4649682}, which are each also good introductions to the topic. The former reference \cite{MR0546593} uses an older term, ``absolutely cuspidal,'' for what is today called ``supercuspidal.''} $\overline{\mathbb{Q}}_{\ell}$-linear $\pi$, there is an isomorphism of $\overline{\mathbb{Q}}_{\ell}$-linear representations of $D_{1/n,\mathbb{Q}_p}^{\times}$:
\begin{align*}
 \mathcal{JL}(\pi)^{\oplus n} 
  &\cong \hom_{\overline{\mathbb{Q}}_{\ell}[GL_n(\mathbb{Q}_p)]}(H^{n-1},\pi),\end{align*} and furthermore every summand of $H^{n-1}$ isomorphic to $\mathcal{JL}(\pi)$ arises in this way. (This is a somewhat unfaithful retelling of Carayol's conjecture, which was really phrased in terms of Drinfeld's $p$-adic hyperplane complement spaces rather than in terms of the Lubin-Tate tower, although the two were later \cite{MR2441311} shown to yield the same representations.) 
See Theorem 2.5.2 of \cite{MR2383890} for an especially clear statement in this form. Versions of this conjecture were proven in \cite{MR1876802}, \cite{MR2383890}, and \cite{MR2680204}. See \cref{Review of the role...}, below, for an expository treatment of some of this material, and the first chapter of \cite{MR1876802} for a much fuller discussion. 

\subsection{Background on the role of LT space in stable homotopy theory.}

By a classical result of \cite{MR0238854}, the Lubin-Tate space $\Def(\mathbb{G})$ is isomorphic to the formal spectrum $\Spf W(k)[[u_1, \dots ,u_{n-1}]]$, where $W(k)$ is the Witt ring of $k$. In algebraic topology, this same ring is also known to arise as the coefficient ring of {\em Morava $E$-theory} \cite[0.3.1]{MR782555}, a generalized cohomology theory, written $E(\mathbb{G})^*$. The coefficient ring $E(\mathbb{G})^*(\pt)$ of Morava $E$-theory is a graded ring isomorphic to $W(k)[[u_1, \dots ,u_{n-1}]][u^{\pm 1}]$, with $u$ in degree $2$ and $u_1, \dots ,u_{n-1}$ all in degree $0$. Furthermore, the generalized cohomology theory $E(\mathbb{G})^*$ is complex oriented, i.e., $E(\mathbb{G})^*$ has a theory of Chern classes for complex vector bundles. This yields a formal group law over $E(\mathbb{G})^0(\pt) \cong W(k)[[u_1, \dots ,u_{n-1}]]$ which describes the effect of the first Chern class on a tensor product of line bundles \cite{MR1324104}. This formal group law coincides with the universal formal group law over the ring of functions on Lubin-Tate space $\Def(\mathbb{G})$. 

Morava $E$-theory plays an important role in computational stable homotopy theory, since in the case $k = \mathbb{F}_{p^n}$, there exist spectral sequences \cite{MR2030586}
\begin{align*}
 E_2^{s,t} &\cong H^s\left(\Aut(\mathbb{G});E(\mathbb{G})_t(X)\right) 
  \Rightarrow \pi_{t-s}\left( (E(\mathbb{G})\wedge X )^{h\Aut(\mathbb{G})}\right)\mbox{\ \ \ \ \ \ and} \\
 E_2^{s,t} &\cong H^s\left( \Gal(\mathbb{F}_{p^n}/\mathbb{F}_p); \pi_t\left( (E(\mathbb{G})\wedge X )^{h\Aut(\mathbb{G})}\right)\right)
  \Rightarrow \pi_{t-s}\left(L_{K(n)}X\right)
\end{align*}
for any finite CW-complex $X$. Here and elsewhere in this paper, the superscript $h\Aut(\mathbb{G})$ denotes the continuous homotopy fixed-points of an action of $\Aut(\mathbb{G})$ on a spectrum. The notation $L_{K(n)}X$ denotes the Bousfield localization (\cite{MR551009},\cite{MR737778}) of $X$ at the height $n$ Morava $K$-theory $K(n)_*$. 

Using the long exact sequence induced in homotopy groups by the homotopy pullback square
\[\xymatrix{
 L_{E(n)}X \ar[r]\ar[d] & L_{K(n)}X \ar[d]\\
 L_{E(n-1)}X \ar[r] & L_{E(n-1)}L_{K(n)}X,}\]
one can, in principle, inductively calculate the homotopy groups of $L_{E(n)}X$ for $n=1, 2, 3, \dots $ starting from $\pi_*(L_{E(0)}X)\cong \mathbb{Q}\otimes_{\mathbb{Z}}\pi_*(X)$ and the homotopy groups of the localizations $\pi_*(L_{K(1)}X),\pi_*(L_{K(2)}X),\dots$. Here $L_{E(n)}$ is Bousfield localization at the Johnson-Wilson $E$-theory $E(n)_*$. These localizations are standard in computational stable homotopy theory. For the sake of this paper, rather than discussing precisely what these localizations are, and what the Morava $K$-theories and the Johnson-Wilson $E$-theories are, we prefer to simply refer to the chromatic convergence theorem \cite{MR1192553}, which states that the homotopy groups of the homotopy limit $\holim_{n\rightarrow \infty} L_{E(n)}X$ agrees with the $p$-local stable homotopy groups of the finite CW-complex $X$. 

The moral is that the cohomology of the action of $\Aut(\mathbb{G})$ acting on the Morava $E$-theory $E(\mathbb{G})_*(X)$ is the input for a widely-studied computational method---a sequence of spectral sequences and long exact sequences---for calculating the stable homotopy groups of a finite CW-complex $X$. 

\subsection{Topological nonrealizability of the classical Lubin-Tate tower.}

So far in this introduction, we have explained that:
\begin{itemize}
\item The Lubin-Tate tower plays an important role in the construction of the $\ell$-adic Jacquet-Langlands correspondence.
\item The {\em base} $\Def(\mathbb{G})$ of the Lubin-Tate tower plays an important role in the calculation of stable homotopy groups of finite CW-complexes.
\end{itemize}
Stable homotopy theorists have long discussed the question of whether there is some deeper connection between these two observations, perhaps some kind of ``$\ell$-adic topological Jacquet-Langlands correspondence'' which might potentially aid in either understanding or calculating stable homotopy groups. The purpose of this paper is to construct such a correspondence and to show that it does indeed aid in understanding stable homotopy groups. 

Before stating the main theorems, we explain the point at which earlier investigations have gotten stuck. From the work of Goerss--Hopkins \cite{MR2125040}, it is known that there exists a particularly natural presheaf $\mathcal{O}^{\topsym}$ of complex-oriented $E_{\infty}$-ring spectra on the small \'{e}tale site of the formal scheme $\Def(\mathbb{G})$, whose $E_{\infty}$-ring spectrum of global sections is the Morava $E$-theory spectrum $E(\mathbb{G})$. In spectral algebraic geometry, 
the pair $(\Def(\mathbb{G}),\mathcal{O}^{\topsym})$ would be called a ``nonconnective spectral formal scheme.'' It is natural to try to lift $\mathcal{O}^{\topsym}$ to a presheaf of $E_{\infty}$-ring spectra on $\Def(\mathbb{G})_{\lvl p}$, then to $\Def(\mathbb{G})_{\lvl p^2}$, and so on. However, since $\Def(\mathbb{G})_{\lvl p^{m+1}} \rightarrow \Def(\mathbb{G})_{\lvl p^m}$ is not \'{e}tale, the existence of such a lift is not automatic. Indeed, it was established by \cite{MR1718085} (cf. \cite{MR4099803} for further work and discussion in that direction) that no such lift to $\Def(\mathbb{G})_{\lvl p}$ can exist when $\mathbb{G}$ has $p$-height $1$. This is what is meant by saying that ``there is no topological realization of the Lubin-Tate tower.''

\subsection{Topological realizability of the degenerating Lubin-Tate tower.}

In this paper, we overcome that obstacle by an extremely simple trick: rather than considering the moduli space $\Def(\mathbb{G})_{\lvl p^m}$ of deformations of $\mathbb{G}$ equipped with a Drinfeld level $p^m$ structure, we allow level structures which are ``degenerate'' in various ways. In \cref{The degenerating Lubin-Tate tower.} we (jointly with M. Strauch) define ``degenerating level structures.'' There is a formal scheme $\Def(\mathbb{G})_{\lvl p^m}^{\degen}$ parameterizing deformations of $\mathbb{G}$ equipped with {\em degenerating} level structures, and these formal schemes fit into a tower, the {\em degenerating Lubin-Tate tower,} into which the classical Lubin-Tate tower embeds $GL_n(\hat{\mathbb{Z}}_p)\times \Aut(\mathbb{G})$-equivariantly\footnote{As a sub-tower of the degenerating Lubin-Tate tower, the classical Lubin-Tate tower is closed under the action of $GL_n(\mathbb{Z}/p^m\mathbb{Z})\times \Aut(\mathbb{G})$ for each $m$, but it is {\em not} closed under the action by power operations on the degenerating Lubin-Tate tower that arise from the $E_{\infty}$-ring realization of $\Def(\mathbb{G})_{\lvl p^m}^{\degen}$ for each $m$.}:
\begin{equation*}\label{lt tower 1} \xymatrix{
 \dots \ar[r] & \Def(\mathbb{G})_{\lvl p^2}\ar[r]\ar@{^{(}->}[d] & \Def(\mathbb{G})_{\lvl p} \ar[r]\ar@{^{(}->}[d] & \Def(\mathbb{G}) \ar[d]^{=} \\
 \dots \ar[r] & \Def(\mathbb{G})_{\lvl p^2}^{\degen}\ar[r] & \Def(\mathbb{G})_{\lvl p}^{\degen} \ar[r] & \Def(\mathbb{G})^{\degen}. }\end{equation*}
In \cref{Topological realizability...}, we find that, as a straightforward consequence of the work \cite{MR1758754} of Hopkins--Kuhn--Ravenel, the Goerss--Hopkins presheaf $\mathcal{O}^{\topsym}$ of $E_{\infty}$-ring spectra on $(\Def(\mathbb{G})^{\degen}_{\lvl p^0})_{\et} =  (\Def(\mathbb{G}))_{\et}$ {\em does} lift to a presheaf of $E_{\infty}$-ring spectra on $(\Def(\mathbb{G})_{\lvl p}^{\degen})_{\et}$, and indeed to $(\Def(\mathbb{G})_{\lvl p^m}^{\degen})_{\et}$ for all $m\geq 1$. To summarize: the classical Lubin-Tate tower does not admit a topological realization, but the classical Lubin-Tate tower embeds into the degenerating Lubin-Tate tower, which {\em does} admit a topological realization.

Degenerating level structures are more elementary than Drinfeld level structures. While we prove some new things about them (and give them a name) in this paper, some types of degenerating level structures have certainly been studied before this paper, e.g. in \cite{MR1344767}, \cite{MR1758754}, and \cite{MR1473889}.

The material in this paper on the cohomology and the topological realizability of the degenerating Lubin-Tate tower is joint work with Matthias Strauch. We have elected to present that material as a self-contained appendix in this paper. That appendix does not logically rely on the material which comes earlier in the paper. To be clear, much of \cref{l-adic JL cohomology theory} and \cref{l-adic TJL duality} relies on the results in the appendix. Two subsections of the appendix, \cref{Review of the Lubin-Tate tower.} and \cref{Review of the role...}, are purely expository sections, which provide a review of known results that we think some readers may find useful. {\em We recommend that readers browse the appendix before beginning to read \cref{l-adic JL cohomology theory} of this paper.}

\subsection{Spectral vanishing cycles, and $\ell$-adic Jacquet-Langlands cohomology.}

Having overcome the obstacle of non-realizability, the next step is to try to mimic, in the topological setting, the construction of the cohomology group $H^{n-1}$ which realizes the $\ell$-adic Jacquet-Langlands correspondence. Clearly we have to use the degenerating Lubin-Tate tower in place of the classical Lubin-Tate tower. However, at this stage we also have to be honest about what kind of cohomology group $H^{n-1}$ is: it is the colimit, over $m$, of the $\ell$-adic compactly supported cohomology of the rigid analytic (Raynaud) generic fibre of $\Def(\mathbb{G})_{\lvl p^m}$, or equivalently (by \cite{MR1395723}), the $\ell$-adic vanishing cycle cohomology of the formal scheme $\Def(\mathbb{G})_{\lvl p^m}$. Facing the decision to either develop some theory of spectral rigid analytic spaces, or some theory of vanishing cycles for spectral formal schemes, we have elected for the latter option. In \cref{The stalk spectral...} and in \cref{Nearby cycles and vanishing cycles...}, we define vanishing cycles for presheaves of spectra on the small \'{e}tale site of a formal scheme over $\Spec$ of a discrete valuation ring, along with a spectral sequence \eqref{vanishing cycles stalk ss 1} for calculating the homotopy groups of the stalk of the vanishing cycles. The constructions are nontrivial, but the papers \cite{MR1262943},\cite{MR1395723} of Berkovich have already done the work of setting up a theory of vanishing cycles for formal schemes which, we find, generalizes from abelian presheaves to spectral presheaves without great difficulty. Since vanishing cycles are not at all a standard tool used in stable homotopy theory, we provide in \cref{Nearby cycles and vanishing cycles...} an introduction to the basic ideas, geared toward easy generalization to the spectral setting.

In \cref{l-adic JL cohomology theory}, we construct the {\em $\ell$-adic Jacquet-Langlands cohomology theory} $JL_{\ell}(\mathbb{G})^*$ associated to the formal group $\mathbb{G}$. The generalized cohomology theory $JL_{\ell}(\mathbb{G})^*$ is represented by a spectrum $JL_{\ell}(\mathbb{G})$ constructed in several steps, as follows: 
\begin{itemize} 
\item For each integer $m\geq 0$, due to the topological realizability of the degenerating Lubin-Tate tower, we have the spectral formal scheme $(\Def(\mathbb{G})_{\lvl p^m}^{\degen},\mathcal{O}^{\topsym})$. 
\item For each integer $j\geq 1$, take the mod $\ell^j$ algebraic $K$-theory of $\mathcal{O}^{\topsym}$ ``sectionswise,'' i.e., form the presheaf of spectra on the \'{e}tale site of $\Def(\mathbb{G})_{\lvl p^m}^{\degen}$ whose value on an \'{e}tale open $U$ is $S/\ell^j \wedge \mathcal{K}(\mathcal{O}^{\topsym}(U))$. Here $S/\ell^j$ denotes the mod $\ell^j$ Moore spectrum.
\item Fibrantly replace (``sheafify'') that presheaf, then take the vanishing cycles presheaf of spectra, then take the stalk at the closed point. 
\item Finally, take the homotopy limit as $j\rightarrow\infty$, then the homotopy colimit as $m\rightarrow\infty$. 
\end{itemize}
The result is the spectrum $JL_{\ell}(\mathbb{G})$. A more leisurely step-by-step account, together with the motivations for the steps, is given in \cref{l-adic JL cohomology theory}. 

The spectrum $JL_{\ell}(\mathbb{G})$ is equipped with a natural action of $GL_n(\hat{\mathbb{Z}}_p)\times \Aut(\mathbb{G})$, where again $n$ is the $p$-height of the formal group $\mathbb{G}$. The spectrum $JL_{\ell}(\mathbb{G})$ is also nonconnective: in the ``vanishing cycles stalk spectral sequence'' \eqref{vanishing cycles stalk ss 1} which we use to analyze the homotopy groups of $JL_{\ell}(\mathbb{G})$, the $\ell$-adic cohomology groups of $\Def(\mathbb{G})^{\degen}_{\lvl p^m}$ in {\em positive} degrees contribute to the homotopy groups of $JL_{\ell}(\mathbb{G})$ in {\em negative} degrees. In particular, in Theorem \ref{main thm}, we show that the $\ell$-adic Jacquet-Langlands correspondence for $\mathbb{Q}_p$ is realized in the rational stable homotopy group $\overline{\mathbb{Q}}_{\ell}\otimes_{\hat{\mathbb{Z}}_{\ell}} \pi_{1-n}\left( JL_{\ell}(\mathbb{G})\right)$. 
This is accomplished by showing that the cohomology group $H^{n-1}$ studied by Carayol and others is in fact a $GL_n(\hat{\mathbb{Z}}_p)\times \Aut(\mathbb{G})$-equivariant summand in $\overline{\mathbb{Q}}_{\ell}\otimes_{\hat{\mathbb{Z}}_{\ell}} \pi_{1-n}\left( JL_{\ell}(\mathbb{G})\right)$. Some work is required to demonstrate this: we must make a calculation in algebraic $K$-theory, Theorem \ref{splitting of k-thy spectrum}, as well as a splitting of the $\ell$-adic compactly supported cohomology of the Raynaud generic fibre of $\Def(\mathbb{G})^{\degen}_{\lvl p^m}$ for each $m$, by showing in Theorem \ref{coh splitting thm} that the inclusion $\Def(\mathbb{G})_{\lvl p^m} \hookrightarrow \Def(\mathbb{G})^{\degen}_{\lvl p^m}$ induces a split $GL_n(\hat{\mathbb{Z}}_p)\times\Aut(\mathbb{G})$-equivariant monomorphism in $\ell$-adic compactly supported cohomology on the generic fibres. In the appendix, we have included several illustrative examples, Examples \ref{ht 1 example 1}, \ref{ht 1 example 2}, and \ref{ht 2 example}, to demonstrate how this cohomological splitting works, and to describe the corresponding geometry in the degenerating Lubin-Tate tower.

\subsection{The topological Jacquet-Langlands dual of a spectrum.}

We continue to let $\mathbb{G}$ be a formal group of finite height over a perfect field of characteristic $p$. Consider the generalized homology theory $(\mathcal{K}(E(\mathbb{G}))^{\wedge}_{\ell})_*$ represented by the $\ell$-adically complete algebraic $K$-theory spectrum of the Morava $E$-theory spectrum $E(\mathbb{G})$. This generalized homology theory takes its values in $\hat{\mathbb{Z}}_{\ell}$-linear representations of $\Aut(\mathbb{G})$. We also have a generalized cohomology theory, valued in $\overline{\mathbb{Q}}_{\ell}$-linear representations of $GL_n(\hat{\mathbb{Z}}_p)$, which sends a spectrum $X$ to the homotopy groups of the homotopy fixed-point spectrum
\begin{align}\label{tjl dual of X}  F(\mathcal{K}(E(\mathbb{G}))^{\wedge}_{\ell}\wedge X,H\overline{\mathbb{Q}}_{\ell}\wedge_{\hat{S}_{\ell}} JL_{\ell}(\mathbb{G}))^{h\Aut(\mathbb{G})}.\end{align}
Here $F(\mathcal{K}(E(\mathbb{G}))^{\wedge}_{\ell}\wedge X,H\overline{\mathbb{Q}}_{\ell}\wedge_{\hat{S}_{\ell}} JL_{\ell}(\mathbb{G}))$ denotes the function spectrum of maps from $\mathcal{K}(E(\mathbb{G}))^{\wedge}_{\ell}\wedge X$ to $H\overline{\mathbb{Q}}_{\ell}\wedge_{\hat{S}_{\ell}} JL_{\ell}(\mathbb{G})$. 
We will refer to the spectrum \eqref{tjl dual of X} as the {\em $\ell$-adic topological Jacquet-Langlands dual of $X$ with respect to $\mathbb{G}$}, or {\em TJL dual of $X$} for short.

As a consequence of Theorem \ref{main thm}, we show in \cref{l-adic TJL duality} that, for each integer $i$ and each irreducible subrepresentation $\rho$ of the $\overline{\mathbb{Q}}_{\ell}$-linear representation $\overline{\mathbb{Q}}_{\ell}\otimes_{\hat{\mathbb{Z}}_{\ell}} \pi_i(\mathcal{K}(E(\mathbb{G}))^{\wedge}_{\ell}\wedge X)$ of $\Aut(\mathbb{G})$ whose ``ordinary'' (i.e., non-topological) Jacquet-Langlands dual $\mathcal{JL}(\rho)$ is supercuspidal, $\mathcal{JL}(\rho)$ occurs as a factor in the $(1-n-i)$th homotopy group of the TJL dual \eqref{tjl dual of X} of $X$. 
Here $\mathcal{JL}(\rho)$ denotes the representation of $GL_n(\hat{\mathbb{Z}}_p)$ associated to $\rho$ by the Jacquet-Langlands correspondence studied in \cite{MR0771672}, \cite{MR1876802}, \cite{MR0700135}, and many other references. 

Speaking a bit more loosely, but perhaps getting the point across better: {\em taking the TJL dual of a spectrum $X$ recovers, on the rational stable homotopy groups, the supercuspidal representations of $GL_n$ which are Jacquet-Langlands duals of the irreducible summands of the action of the Morava stabilizer group $\Aut(\mathbb{G})$ on the $\ell$-complete $\mathcal{K}(E(\mathbb{G}))$-homology of $X$.}


However, there are many other $\ell$-adic representations of $\Aut(\mathbb{G})$ that occur in the homotopy group $\pi_{1-n-i}\left(\left(  F(X,H\overline{\mathbb{Q}}_{\ell}\wedge JL_{\ell}(\mathbb{G}))\right)^{h\Aut(\mathbb{G})}\right)$ of \eqref{tjl dual of X}, aside from those with supercuspidal dual representations. The degenerating Lubin-Tate tower is much larger than the classical Lubin-Tate tower, with many redundant copies of ``shifts'' of the Lubin-Tate tower alongside the canonical copy of the classical Lubin-Tate tower. These extra components of the degenerating Lubin-Tate tower contribute to its \'{e}tale cohomology and consequently, through the modified vanishing cycles stalk spectral sequence constructed in the proof of Theorem \ref{main thm}, also to the homotopy groups of \eqref{tjl dual of X}. As a consequence, various Tate twists of principal series and non-supercuspidal discrete series representations of $GL_n$ also occur in the rational homotopy groups of \eqref{tjl dual of X}. Unlike what occurs in the \'{e}tale cohomology of the classical Lubin-Tate tower, the supercuspidal representations are not isolated (by cohomological degree) from the non-supercuspidal representations in the \'{e}tale cohomology groups of the degenerating Lubin-Tate tower, or in the homotopy groups of \eqref{tjl dual of X}. These ``noisy'' extra summands are the price we have paid in order to have a topological realization. For $n=1$, we describe the resulting decomposition of the homotopy groups of \eqref{tjl dual of X} as a $GL_1(\hat{\mathbb{Z}}_p)$-representation in \cref{l-adic TJL duality}.


\subsection{At height $1$, TJL duality preserves $L$-factors.}

The $\ell$-adic Langlands correspondence, between certain $\ell$-adic representations of Weil groups and certain $\ell$-adic representations of $GL_n(\mathbb{Q}_p)$, is known to {\em preserve the $L$-factors}: the Artin $L$-factor of a Galois representation agrees with the automorphic $L$-factor of the corresponding automorphic representation\footnote{See \cite{MR2808915} for a textbook treatment for automorphic $L$-factors---i.e., local $L$-functions---of representations of $GL_n(\mathbb{Q}_p)$. However, in this paper, the only theorem we prove involving $L$-factors is Theorem \ref{main automorphic thm}, and that theorem involves only very simple $L$-factors, namely, products of Tate twists of $L$-factors of {\em trivial} representations. Put another way: they are products of factors of the form $(1-p^{w-s})^{-1}$ for various integers $w$. Consequently the reader does not need to be knowledgeable about automorphic representations to read this paper. 

The simple form of these $L$-factors that occur in Theorem \ref{main automorphic thm} is partly due to a topological phenomenon, the rational diagonalizability of the Adams operations on complex $K$-theory. But it is also partly due to the fact that Theorem \ref{main automorphic thm} addresses only the case $n=1$. Any future attempts to prove analogues of Theorem \ref{main automorphic thm} for $n>1$---i.e., attempts to prove that TJL duality preserves $L$-factors at higher heights---will surely need to make a more substantial use of the theory of automorphic representations than the extremely modest engagement with that theory that the reader can find in this paper.}. One would like to know whether there is any sense in which TJL duality also preserves $L$-factors. On the automorphic side (i.e., on the side \eqref{tjl dual of X} of TJL duality on which there is a representation of $GL_n(\hat{\mathbb{Z}}_p)$), one can associate an automorphic $L$-factor to the representation of $GL_n$. It is not clear what $L$-factor one ought associate to the CW-complex $X$ on the homological side $(\mathcal{K}(E(\mathbb{G}))^{\wedge}_{\ell})_*(X)$ of the TJL duality. 

However, if $\mathbb{G}$ has $p$-height $1$, there is a good candidate for such an $L$-factor: the preprint \cite{salch2023kulocal} proposes a theory of ``$KU$-local zeta-functions'' for finite CW-complexes. Suppose that on the homological side of TJL duality, for $\mathbb{G}$ of $p$-height $1$, we take the appropriate $L$-factor to be the $p$-local Euler factor in the ``provisional $KU$-local zeta-function'' of the finite CW-complex $X$, as defined in \cite{salch2023kulocal}. In Theorem \ref{main automorphic thm}, we prove that, if the cohomology $H^*(X;\mathbb{Z})$ is torsion-free and concentrated in even dimensions, then {\em TJL duality preserves the $L$-factors}.

Here is a striking consequence, which combines the results in this paper with one of the main results of the preprint \cite{salch2023kulocal}. Let $\mathbb{G}$ denote the formal group of complex topological $K$-theory $KU^*$, i.e., $\mathbb{G}$ is the multiplicative formal group over $\mathbb{Z}$. For each prime $p$, let $\mathbb{G}\otimes_{\mathbb{Z}} \overline{\mathbb{F}}_p$ be the base-change of the formal group $\mathbb{G}$ to the field $\overline{\mathbb{F}}_p$. Given a finite CW-complex $X$, consider the product, over all primes $p$, of the $p$-local automorphic $L$-factor of the representation of $GL_1(\mathbb{Q}_p)$ which is the TJL dual\footnote{This statement is slightly oversimplified: see \cref{l-adic TJL duality} for a precise statement. If $n=1$, the TJL dual $\pi_*\left( F(\mathcal{K}(E)^{\wedge}_{\ell}\wedge X,H\overline{\mathbb{Q}}_{\ell}\wedge_{\hat{S}_{\ell}} JL_{\ell}(\mathbb{G}))^{h\Aut(\mathbb{G})}\right)$ of $(\mathcal{K}(E(\mathbb{G})^{\wedge}_{\ell})_*(X)$ is really an enormous direct sum of many redundant copies of the same representation of $GL_1(\hat{\mathbb{Z}}_p)$. This redundancy is a consequence of the fact that we have used the degenerating Lubin-Tate tower in place of the classical Lubin-Tate tower, and the cohomology of the degenerating Lubin-Tate tower includes the cohomology of the Lubin-Tate tower as a distinguished summand, along with other summands. In the case $n=1$ these summands are all isomorphic to each other; see Example \ref{ht 1 example 2} for details. When we speak of ``the automorphic $L$-factor associated to $X$'' we are simply taking the automorphic $L$-factor of one copy among the countably infinitely many copies of the same representation of $GL_1(\hat{\mathbb{Z}}_p)$ in the TJL dual of $X$.} of $X$. Then this product analytically continues to a meromorphic function $L_{\mathbb{G}}(s,X)$ on the complex plane. Furthermore, special values of that meromorphic function in a suitable left-hand half plane describe the orders of the $KU$-local stable homotopy groups of the Spanier-Whitehead dual $DX$ of $X$:
\begin{align*}
 \left| \pi_{-2k-1}(L_{KU}DX)\right|
  &= \denom L_{\mathbb{G}}(1-k,X)
\end{align*}
up to powers of $2$ and powers of irregular primes, for all integers $k \geq b+1$. Here $b$ is the greatest integer $i$ such that $H^{2i}(X;\mathbb{Q})$ is nontrivial. This is the content of Theorem \ref{main automorphic cor}.

The point is that TJL duality actually says something of interest to a topologist who doesn't already care about Langlands correspondences or vanishing cycles or anything else in this paper except homotopy groups: localized at a complex-oriented theory of height $1$, the orders of stable homotopy groups of a large class of finite CW-complexes are describable in terms of values of $L$-functions that you can produce using TJL duality. We hope that generalizations of Theorems \ref{main automorphic thm} and \ref{main automorphic cor} also hold for formal groups of height $>1$, but that is a question for a later paper.

We remark that, for expository purposes, we have written this paper almost exclusively about Jacquet-Langlands correspondences and not Langlands correspondences. But an ``$\ell$-adic topological Langlands duality'' is constructible using the same methods as described in this paper, together with standard ideas about Galois actions on the cohomology of Lubin-Tate tower. The essential ideas are as follows. 
By replacing the (rigid analytic generic fibre of) degenerating Lubin-Tate space $\Def(\mathbb{G})^{\degen}_{\lvl p^m}$ with the base-changed version $\Def(\mathbb{G})_{\lvl p^m}^{\degen}\times_{\mathbb{Q}_p^{nr}}\overline{\mathbb{Q}}_p$, the action of $GL_n(\mathbb{Q}_p)\times \Aut(\mathbb{G})$ on the cohomology of the degenerating Lubin-Tate tower can be extended to an action of $GL_n(\mathbb{Q}_p)\times \Aut(\mathbb{G})\times W_{\mathbb{Q}_p}$, where $W_{\mathbb{Q}_p}$ is the Weil group of $\mathbb{Q}_p$. This yields $\ell$-adic Galois representations and corresponding Artin $L$-factors, realizing the $\ell$-adic Langlands (as opposed to Jacquet-Langlands) correspondence for $\mathbb{Q}_p$ in the homotopy groups of the resulting extended Jacquet-Langlands cohomology spectrum. Some relevant background discussion is given in \cref{Background on the Lubin-Tate tower...}, but in this paper, for simplicity, we mostly avoid discussing the Galois representations, focusing our discussion instead on the $\Aut(\mathbb{G})$-representations and the $GL_n(\mathbb{Z}_p)$-representations.


In this paper, we have been fortunate to be able to take advantage of the good ideas and hard work of many others who have developed various powerful methods and results which are used in this paper largely as a ``black box.'' For example:
\begin{itemize}
\item We take advantage of many results on Carayol's program, in particular those proven by Harris--Taylor \cite{MR1876802}, Strauch \cite{MR2383890}, and Mieda \cite{MR2680204}.
\item Once you have the idea to consider the degenerating Lubin-Tate tower, its topological realizability is an easy consequence of the work \cite{MR1758754} of Hopkins--Kuhn--Ravenel.
\item We use various tools for working with presheaves of spectra on Grothendieck sites, developed by Thomason \cite{MR826102} and developed further by Jardine \cite{MR3309296}. 
We are saved from from having to enter much technically deeper waters by never having to consider any version of quasicoherent modules over rigid analytic spaces: abelian sheaves and spectral presheaves on the relevant rigid analytic spaces are enough for our purposes in this paper.
\item We use the work of Berkovich \cite{MR1262943}, \cite{MR1395723} on vanishing cycle sheaves for formal schemes.
\item Various results in algebraic $K$-theory: Suslin rigidity \cite{MR934225}, Gabber rigidity \cite{MR1156502}, and Blumberg--Mandell's version \cite{MR2413133} of the localization sequence \cite{MR802796} for Waldhausen $K$-theory.
\end{itemize}

\begin{acknowledgements}
This paper took a very long time to come together. In 2009, the author and Jack Morava ran a ``Topological Langlands'' seminar at Johns Hopkins University, in which we and the other participants (Abhishek Banerjee, Romie Banerjee, Chenghao Chu, Katia Consani, Snigdhayan Mahanta, and Rekha Santhanam---we apologize if we have forgotten anyone!) tried to understand some relationships between homotopy-theoretical aspects and representation-theoretic aspects of Morava $E$-theories. The seminar petered out as we got stuck on the problem of topological nonrealizability of the Lubin-Tate tower. During a workshop on $p$-divisible groups hosted at Urbana-Champaign around the same time, the author learned from Matt Ando and Charles Rezk that they were quite familiar with this problem. The author benefitted from their insightful comments and generous explanations. By 2011 the author had an early, clumsy version of the topological realizability result in this paper, using the Drinfeld tower of \cite{MR0422290}, but felt it was not really the ``right'' approach, as there seemed to be no noteworthy homotopy-theoretic consequences. In 2017, the author met Matthias Strauch at a conference hosted at Indiana University, and we began discussing this problem. These discussions continued into 2018 when the author spent part of a sabbatical semester at IU, and Strauch and the author worked out the theory of degenerating level structures, including the topological realizability and cohomological splitting results, documented in the appendix to this paper. Finally, Gabe Angelini-Knoll helped the author with references for facts about topological cyclic homology and localization.

The author is grateful to all of the above people, as well as to everyone who has patiently listened to the author prattle on about this subject over the years, and everyone who has waited a very, very long time to see anything writen about it. More than anyone else, Jack Morava and Matthias Strauch have played critical roles in this project, and the author is deeply grateful to them. 
\end{acknowledgements}

\begin{convention}\label{main conventions}\leavevmode
\begin{itemize}
\item As already stated: all formal groups and formal modules in this paper are implicitly understood to be one-dimensional. The symbols $p,\ell$ shall always denote distinct prime numbers.
\item
Throughout, $A$ will be the ring of integers in a finite extension of $\mathbb{Q}_p$. We fix a uniformizer $\varpi$ for $A$.
\item
When $\mathbb{G}$ is a formal $A$-module over a complete local commutative $A$-algebra $B$ with maximal ideal $\mathfrak{m}_B$, we will write $\mathbb{G}(\mathfrak{m}_B)$ for the resulting group consisting of the elements of $\mathfrak{m}_B$ equipped with the group structure and $A$-action given by $\mathbb{G}$.
We will reserve the symbol $n$ to always mean the $A$-height of $\mathbb{G}$, and we will always assume that $n$ is finite and positive.
\item It will frequently happen that we will have a formal module $\mathbb{G}$ and a deformation $\tilde{\mathbb{G}}$ of $\mathbb{G}$ to an unspecified local Artinian ring. We adopt the convention that $\mathfrak{m}_{\tilde{\mathbb{G}}}$ denotes the maximal ideal of the ground ring of $\tilde{\mathbb{G}}$.
\item
Given a positive integer $i$, we write $S/i$ for the mod $i$ Moore spectrum.
\item
Throughout, we use the framework for (pre)sheaves of spectra due to Thomason \cite{MR826102} and developed further by Jardine \cite{MR3309296}. In this framework, a model structure is used to encode homotopical information. Certainly all the constructions and arguments involving (pre)sheaves of spectra in this paper can equally well be carried out in the framework of Lurie's Derived Algebraic Geometry series, which begins with \cite{dagi}. In that framework, rather than a model structure, simplicial methods are used to encode homotopical data. We chose to write this paper using Thomason's framework because we imagine that some readers outside of homotopy theory might find this paper more digestible as a result. We apologize to readers who would have preferred we write this paper using the $\infty$-categorical framework instead.
\end{itemize}
\end{convention}

\section{Vanishing cycles.}

Vanishing cycle sheaves play a central role in the construction of the cohomology groups that realize the $\ell$-adic Jacquet-Langlands correspondence for a local field. In this section we develop some machinery for handling vanishing cycles sheaves {\em of spectra,} in the sense of stable homotopy theory. This machinery is used in \cref{l-adic JL cohomology theory} and \cref{l-adic TJL duality}. 

\subsection{The stalk spectral sequence of a morphism of sites.}
\label{The stalk spectral...}

Suppose we are given Grothendieck sites $(\mathcal{C}_1,\tau_1)$ and $(\mathcal{C}_2,\tau_2)$. A {\em homomorphism of sites} is a functor $f: \mathcal{C}_1\rightarrow \mathcal{C}_2$ such that, if $\{ U_i \stackrel{\phi_i}{\longrightarrow} U\}$ is a covering family in $\tau_1$, then $\{ f(U_i) \stackrel{f(\phi_i)}{\longrightarrow} f(U)\}$ is a covering family in $\tau_2$. 

Write $\Sp\Sh(\mathcal{C},\tau)$ for the category of presheaves of spectra on a site $(\mathcal{C},\tau)$. We refer to \cite{MR3309296} for the Jardine model structure on $\Sp\Sh(\mathcal{C},\tau)$. For our purposes, its most important property is that the fibrant objects are those presheaves of spectra which satisfy the appropriate homotopical analogue of the sheaf axiom: for every open $U$, the value of a fibrant presheaf $\mathcal{F}$ at $U$ is weakly equivalent to the totalization of the evaluation of $\mathcal{F}$ on the \v{C}ech nerve of any open cover of $U$. Consequently,  in the setting of spectral presheaves, fibrant replacement in the Jardine model structure on $\Sp\Sh(\mathcal{C},\tau)$ is the correct analogue of sheafification. When convenient, we will sometimes write ``sheaf of spectra'' as shorthand for ``fibrant presheaf of spectra.''

A homomorphism of sites $f: (\mathcal{C}_1,\tau_1)\rightarrow (\mathcal{C}_2,\tau_2)$ induces two functors on the associated categories of presheaves of spectra. The functor $f_*: \Sp\Sh(\mathcal{C}_2,\tau_2)\rightarrow \Sp\Sh(\mathcal{C}_1,\tau_1)$ sends a presheaf $\mathcal{F}$ to the composite 
\[ \mathcal{C}_1^{\op}\stackrel{f^{\op}}{\longrightarrow} \mathcal{C}_2^{\op}\stackrel{\mathcal{F}}{\longrightarrow} \Sp .\]
The functor $f^*: \Sp\Sh(\mathcal{C}_1,\tau_1)\rightarrow \Sp\Sh(\mathcal{C}_2,\tau_2)$ is left adjoint to $f_*$. See paragraph 1.26 of \cite{MR826102} for this material\footnote{It is deliberate, and for good reason, that $f_*$ is of the opposite variance of $f$ itself. This is standard among algebraic geometers, but we hope this paper's audience will include homotopy theorists who will not already have thought about site homomorphisms, so we explain the strange variance of $f_*$. Consider the case of a scheme morphism $g: X \rightarrow Y$. Then $g$ induces a covariant pushforward functor $g_*: \Ab(X_{\Zar}) \rightarrow  \Ab(Y_{\Zar})$ on the categories of sheaves of abelian groups on $X$ and on $Y$. However, if we regard $g_*$ as being induced by a site morphism, it is induced by a morphism of sites $Y_{\Zar} \rightarrow X_{\Zar}$. So, in order to get $g_*$ to have the correct variance on morphisms of schemes, we must accept that it will be contravariant on morphisms of sites.}.

Now suppose we are given a presheaf $\mathcal{F}$ of spectra on $(\mathcal{C}_2,\tau_2)$. In general there is no single Leray-like spectral sequence which relates the sheaf cohomology of the sheafification of the presheaf of graded abelian groups $R^*f_*\pi_*(\mathcal{F})$ to the homotopy groups of the global sections of the sheafification (i.e., fibrant replacement) of $f_*\mathcal{F}$. Some discussion of this difficulty can be found immediately preceding Lemma 2.10 of \cite{MR826102}. 

However, if we take a {\em stalk} rather than taking global sections, then the situation improves. Recall that a geometric point of a site $(\mathcal{C},\tau)$ is, by definition, a geometric morphism of toposes from the category of sets to the category of sheaves of sets on $(\mathcal{C},\tau)$. A geometric point induces a stalk functor $\Ab\Sh(\mathcal{C},\tau) \rightarrow \Ab$, which is exact; see \cite[\href{https://stacks.math.columbia.edu/tag/00Y3}{Tag 00Y3}]{stacks-project} for a nice presentation. Given a sheaf of abelian groups $\mathcal{F}$ on $(\mathcal{C},\tau)$ and a geometric point $x$ of $(\mathcal{C},\tau)$, we will write $\mathcal{F}_x$ for the stalk of $\mathcal{F}$ at $x$. 

When we specialize the ideas in the previous paragraph to the case where $(\mathcal{C},\tau)$ is the small Zariski or small \'{e}tale site of some scheme $X$, to each geometric point $x$ of $(\mathcal{C},\tau)$ we can find a geometric point of $X$ (i.e., a scheme morphism $\Spec k \rightarrow X$ with $k$ an algebraically closed field) whose stalk functor is isomorphic to that of $x$, and conversely. 


We adopt the following notation: given a presheaf $\mathcal{G}$ on a site, we will write $\widetilde{\mathcal{G}}$ for its sheafification.
\begin{prop}\label{stalk ss prop}
Suppose we are given a fibrant presheaf of spectra $\mathcal{F}$ on $(\mathcal{C}_2,\tau_2)$ and a geometric point $x$ of $(\mathcal{C}_1,\tau_1)$. We continue to suppose that $f: (\mathcal{C}_1,\tau_1)\rightarrow (\mathcal{C}_2,\tau_2)$ is a morphism of sites. Make the following assumptions\footnote{These assumptions can surely be significantly weakened without ruining the result, but the given level of generality is more than sufficient for all applications in this paper.}:
\begin{enumerate}
\item There exists a uniform lower bound on the homotopy groups of $\mathcal{F}$, i.e., there exists an integer $N$ such that $\pi_i(\mathcal{F}(U))\cong 0$ for all $i<N$ and all $U\in \ob\mathcal{C}_2$.
\item The induced functor on the categories of abelian sheaves $f_*: \Ab\Sh(\mathcal{C}_2,\tau_2) \rightarrow \Ab\Sh(\mathcal{C}_1,\tau_1)$ has finite cohomological dimension, i.e., there exists some integer $M$ such that $R^mf_*: \Ab\Sh(\mathcal{C}_2,\tau_2) \rightarrow \Ab\Sh(\mathcal{C}_1,\tau_1)$ is trivial for all $m>M$.
\end{enumerate}
Then we have a strongly convergent spectral sequence 
\begin{align}
\label{stalk ss 1} E_1^{s,t} \cong \left(R^sf_*\widetilde{\pi_t(\mathcal{F})})\right)_{x}
  &\Rightarrow \pi_{t-s}\left( (f_*\mathcal{F})_x\right),\\
 d_r: E_r^{s,t} &\rightarrow E_r^{s+r,t+r-1}
\end{align}
natural in the variable $\mathcal{F}$. 
\end{prop}
\begin{proof}
We do not know anywhere in the literature where this spectral sequence has appeared, but its construction uses only the well-known tools of \cite{MR826102} in a straightforward way. 
For any presheaf $\mathcal{G}$ of abelian groups on a site $(\mathcal{C},\tau)$, we may form the presheaf of spectra $H\mathcal{G}$ of $(\mathcal{C},\tau)$, whose value on an object $U$ of $\mathcal{C}$ is the Eilenberg-Mac Lane spectrum $H(\mathcal{G}(U))$. 

Now form the fibrant sectionswise Postnikov tower
\begin{equation}\label{postnikov tower 1}\xymatrix{
 \dots \ar[r] & 
  \mathcal{F}^{\leq 2} \ar[r] &
  \mathcal{F}^{\leq 1} \ar[r] &
  \mathcal{F}^{\leq 0} \ar[r] &
  \dots \\
 &
  \widetilde{\Sigma^2 H\pi_2\mathcal{F}} \ar[u] 
 &
  \widetilde{\Sigma H\pi_1\mathcal{F}} \ar[u] 
 &
  \widetilde{H\pi_0\mathcal{F}}, \ar[u] &
}\end{equation}
i.e., the levelwise fibrant replacement of the sectionswise Postnikov tower of $\mathcal{F}$. 
Apply the functor $f_*$ levelwise to \eqref{postnikov tower 1}. This yields a tower of fibrations of sheaves of spectra on $(\mathcal{C}_1,\tau_1)$. Take the levelwise stalk at $x$ to get a tower $(f_*\Post(\mathcal{F}))_x^{\bullet}$ of fibrations of spectra. 
It is standard that applying $\pi_*$ to a tower of fibrations yields an exact couple and consequently a spectral sequence. 

We claim that the homotopy limit of the tower $(f_*\Post(\mathcal{F}))_x^{\bullet}$ is $(f_*\mathcal{F})_{x}$. Clearly the claim is true if the Postnikov system $\Post(\mathcal{F})$ is finite, since then the homotopy limit is eventually constant in both the domain and the codomain of the canonical comparison map
\begin{align} 
\label{comparison map 1} (f_*  \holim \Post(\mathcal{F}) )_x &\rightarrow \holim (f_*\Post(\mathcal{F}) )_x.\end{align}
If the Postnikov system is not finite, consider the effect of the map \eqref{comparison map 1} in $\pi_i$, for any given integer $i$. The $m$th filtration quotient $\widetilde{\Sigma^m H\pi_m\mathcal{F}}$ in the Postnikov tower has homotopy groups concentrated in degrees $m-M, m-M+1, \dots ,m-1, m$, since $\pi_j(\widetilde{\Sigma^m H\pi_m\mathcal{F}}) \cong \left(R^{m-j}f_*)(\widetilde{\pi_m\mathcal{F}})\right)_x$.

Hence, for any given value of $i$, the effect of the comparison map \eqref{comparison map 1} in $\pi_i$ depends only on the $(i+M)$th Postnikov truncation of $\mathcal{F}$, which is a {\em finite} Postnikov system. Hence \eqref{comparison map 1} induces an isomorphism in $\pi_i$. This is true for each integer $i$, so \eqref{comparison map 1} is a weak equivalence. Hence the spectral sequence obtained by applying $\pi_*$ to $(f_*\Post(\mathcal{F}))_x^{\bullet}$ converges to $\pi_*\left((f_*\mathcal{F})_{x}\right)$, as desired.

The $E_1$-term of the spectral sequence is the direct sum, over all integers $n$, of $\pi_*$ of the stalk at $x$ of $f_*$ of the fibrant replacement $\widetilde{\Sigma^n H\pi_n\mathcal{F}}$ of $\Sigma^n H\pi_n\mathcal{F}$. The fibrant replacement of $\Sigma^n H\pi_n\mathcal{F}$ in $\Sp\Sh(\mathcal{C}_2,\tau_2)$ is simply the $n$-fold suspension of the total derived functor of the presheaf of graded abelian groups $\pi_n\mathcal{F}$ on $(\mathcal{C}_2,\tau_2)$. Consequently its stalk at $x$ has homotopy groups simply given by $(R^*f_*\mathcal{F})_x$, yielding our description of the $E_1$-term of the spectral sequence.

The first and second vanishing assumptions ensure that the spectral sequence is concentrated between two horizontal lines crossed by all sufficiently long differentials, yielding strong convergence.
\end{proof}

We refer to \eqref{stalk ss 1} as the {\em stalk spectral sequence.}

\subsection{Nearby cycles and vanishing cycles, for presheaves of spectra.}
\label{Nearby cycles and vanishing cycles...}

We recall \cite{MR0354656} the classical definition of nearby cycles and vanishing cycles for schemes over the spectrum of a Henselian discrete valuation ring $S$. The most important case for us is the case where $S$ is the Witt ring $W(k)$, with $k$ a finite field. To be as explicit as possible: $W(\mathbb{F}_{p^j})$ is isomorphic to
$\hat{\mathbb{Z}}_p[\zeta_{p^j-1}]$, where $\zeta_{p^j-1}$ is a primitive $(p^j-1)$st root of unity. Write $s$ for the special (i.e., closed) point of $S$, write $\eta$ for the generic point of $S$, and write $\overline{\eta}$ for the separable closure of $\eta$.

Now suppose we are given a scheme $X$ over $S$. Write $X_s$ for the fiber of $X$ over $s$, and write $X_{\overline{\eta}}$ for the fiber of $X$ over $\overline{\eta}$. We will furthermore write $i: X_s \hookrightarrow X$ and $j: X_{\overline{\eta}}\hookrightarrow X$ for the respective inclusions. Given a bounded-below cochain complex $C^{\bullet}$ of abelian sheaves on the small \'{e}tale site $X_{\et}$, the {\em nearby cycles of $C^{\bullet}$}, written $R\Psi C^{\bullet}$,  is the object of the derived category $D^+((X_s)_{\et})$ defined as
\begin{align}
\label{def of nearby cycles 1} R\Psi C^{\bullet} & := i^*Rj_*(j^*C^{\bullet}).
\end{align}
We emphasize that we are working only with sheaves of abelian groups here, not modules over the structure ring sheaves! That is why there is no need to derive the functors $j^*$ or $i^*$ in \eqref{def of nearby cycles 1}: unlike categories of modules over structure ring sheaves, the pullback maps on categories of abelian sheaves are always exact. 

We have a natural map $i^*C^{\bullet}\rightarrow i^*Rj_*(j^*C^{\bullet})$ in $D^+((X_s)_{\et})$, arising from applying $i^*$ to the unit map of the adjunction $j^*\dashv Rj_*$. The {\em vanishing cycles of $C^{\bullet}$}, written $\Phi(C^{\bullet})$, is the cofiber of that natural map. Consequently $C^{\bullet}$ and its nearby cycles and vanishing cycles fit into a single cofiber sequence in $D^+((X_s)_{\et})$,
\[ i^*C^{\bullet} \rightarrow R\Psi C^{\bullet}\rightarrow \Phi C^{\bullet}. \]

As a special case, when $\mathcal{F}$ is an abelian sheaf on the small \'{e}tale site $X_{\et}$, we write $\mathcal{F}[0]$ for the object of $D^+(X_{\et})$ with $\mathcal{F}[0]$ in degree zero and trivial in nonzero degrees. We then write $R^s\Phi \mathcal{F}$ for the $s$th cohomology group of the vanishing cycle complex $\Phi(\mathcal{F}[0])$:
\begin{align}\label{rspsi def} R^s\Phi \mathcal{F} &:= H^s(\Phi(\mathcal{F}[0])).\end{align}

The above definitions also extend in a straightforward way to the situation where $C^{\bullet}$ is replaced by a presheaf of spectra on $X_{\et}$ (or indeed, a presheaf on $X_{\et}$ taking values in any reasonably well-behaved stable model category or stable $\infty$-category). 
Using the general constructions of \cref{The stalk spectral...}, we have derived pullback functors 
\begin{align*}
 Li^*: \Sp\Sh(X_{\et}) &\rightarrow \Sp\Sh((X_s)_{\et}) \mbox{\ \ \ and} \\
 Lj^*: \Sp\Sh(X_{\et}) &\rightarrow \Sp\Sh((X_{\overline{\eta}})_{\et}) ,\end{align*}
a derived pushforward functor
\begin{align*}
 Rj_*: \Sp\Sh((X_{\overline{\eta}})_{\et}) &\rightarrow \Sp\Sh(X_{\et}),
\end{align*}
and an adjunction $Lj^*\dashv Rj_*$. This yields a unit natural transformation $\id \rightarrow Rj_*Lj^*$. For our purposes, the essential properties of these derived pushforward and pullback functors are as follows:
\begin{itemize}
\item Given a geometric point $x$ of $X$, the homotopy groups of the stalk $(Rj_*\mathcal{F})_x$ are computable using the spectral sequence of Proposition \ref{stalk ss prop}, provided that the hypotheses of Proposition \ref{stalk ss prop} are satisfied.
\item By a similar argument and using the acyclicity of $j^*$, we have 
\begin{align}
\label{iso 3241} \pi_*\left( (Lj^*\mathcal{F})_x\right) &\cong \left(j^*\widetilde{\pi_*\mathcal{F}}\right)_x 
\end{align} 
for every geometric point $x$ of $X$. A similar isomorphism holds for $Li^*\mathcal{F}$.
\end{itemize}

\begin{definition}
Let $S,X,s,\eta,i,j$ be as defined in the previous paragraphs. Suppose that $\mathcal{F}$ is a presheaf of spectra on $X_{\et}$. 
The {\em nearby cycles of $\mathcal{F}$} is the presheaf $\Psi \mathcal{F}$ of spectra on $(X_s)_{\et}$:
\begin{align}
\label{def of nearby cycles 2} \Psi \mathcal{F} &:= Li^*(Rj_*(Lj^*\mathcal{F})).
\end{align}
The {\em vanishing cycles of $\mathcal{F}$}, written $\Phi\mathcal{F}$, is the presheaf of spectra on $(X_s)_{\et}$ given by the cofiber of the natural map $Li^*\mathcal{F} \rightarrow Li^*(Rj_*(Lj^*\mathcal{F}))$. 
\end{definition}
Consequently we have a canonical cofiber sequence in $\Sp\Sh((X_s)_{\et})$,
\[ Li^*\mathcal{F}\rightarrow \Psi\mathcal{F}\rightarrow \Phi\mathcal{F}.\]

However, the above is not enough: for our purposes, we need vanishing cycles for presheaves of spectra not when $X$ is a scheme over $S$, but instead when $X$ is a {\em formal} scheme over $S$. Recall that the generic fibre of a formal scheme is not simply the fiber over the generic point of $S$ (this would be empty); instead it is a rigid analytic space (\cite{MR0470254}, \cite{MR1202394}, \cite{MR1225983}). In the article \cite{MR1262943}, Berkovich generalized the construction of vanishing cycles to the case where $X$ is a formal scheme with suitable finiteness properties. 

Berkovich's construction of vanishing cycles for formal schemes is 
by building morphisms between the \'{e}tale sites of $X_s$ and of $X_{\overline{\eta}}$, using Berkovich's ``quasi-\'{e}tale site'' as an intermediary: see the beginning of section 4 of \cite{MR1262943} for this. Hence the generalization of Berkovich's vanishing cycles from abelian sheaves to spectral sheaves can be carried out formally, using the general constructions of \cref{The stalk spectral...}. 
That is, even when $X$ is a formal scheme---most importantly for our purposes, $\Def(\mathbb{G})^{\degen}_{\lvl p^m}$, as defined in \cref{The degenerating Lubin-Tate tower.}---for each presheaf of spectra $\mathcal{F}$ on $X_{\et}$ there exists a ``vanishing cycles'' presheaf of spectra $\Phi\mathcal{F}$ on $(X_s)_{\et}$. 
The essential property one would like to establish for $\Phi\mathcal{F}$ is the existence of a strongly convergent spectral sequence
\begin{align}\label{vanishing cycles stalk ss 1}
 E_1^{s,t} \cong R^s\Phi\left( \widetilde{\pi_t\mathcal{F}}\right)_x
 &\Rightarrow \pi_{t-s}\Gamma\left(\widetilde{\Phi(\mathcal{F})}_x\right),\\
\nonumber d_r: E_r^{s,t} &\rightarrow E_r^{s+r,t+r-1}
\end{align}
where:
\begin{itemize}
\item $x$ is a geometric point of $X$, 
\item $\widetilde{\pi_t\mathcal{F}}$ is the sheafification of the presheaf of abelian groups $\pi_t\mathcal{F}$ on $X_{\et}$, 
\item $\widetilde{\Phi(\mathcal{F})}$ is the sheafification (i.e., Jardine fibrant replacement) of the presheaf of spectra $\Phi(\mathcal{F})$ on $(X_s)_{\et}$,
\item and the notation $R^s\Phi$ is as defined above, in \eqref{rspsi def}.
\end{itemize}
This spectral sequence, with all the above properties, indeed exists, under the assumptions that:
\begin{enumerate}
\item $X$ is a degenerating Lubin-Tate space,
\item there is a uniform lower bound on the homotopy groups of $\mathcal{F}$ (cf. the same assumption in Proposition \ref{stalk ss prop}), and
\item for each integer $i$, the sheaf of abelian groups $\widetilde{\pi_i\mathcal{F}}$ on $X_{\et}$ is constructible, of torsion orders prime to $p$.
\end{enumerate}
The construction of spectral sequence \eqref{vanishing cycles stalk ss 1} is basically the same as that of the stalk spectral sequence \eqref{stalk ss 1}: take the fibrant sectionswise Postnikov tower $\Post(\mathcal{F})$ of $\mathcal{F}$, 
apply $\Phi$ levelwise, take the stalk, then take homotopy groups to get an exact couple and consequently a spectral sequence. 
However, there is something nontrivial in establishing that the spectral sequence's abutment is as claimed.
To guarantee that the spectral sequence converges to the homotopy groups of the stalk of $\Phi \mathcal{F}$, one needs to know that the canonical map
\begin{align} \label{comparison map 2} (\Phi  \holim \Post(\mathcal{F}) )_x &\rightarrow \holim\left( (\Phi\Post(\mathcal{F}) )_x\right)\end{align}
is a weak equivalence. To establish this, one can use the same argument as given in the proof of Proposition \ref{stalk ss prop} to identify the abutment of the stalk spectral sequence \eqref{stalk ss 1}.  To make that argument, one needs to know that there is an integer $N$ such that the homotopy groups of $(\Phi \Sigma^m\widetilde{H\pi_m\mathcal{F}})_x$ are concentrated in degrees between $m-N$ and $m$. 

This is where we use the {\em quasi-affineness} of the generic fibre of degenerating Lubin-Tate space---which we prove in Proposition \ref{generic fibre decomp}---alongside some results of Berkovich. 
Write $\mathcal{D}$ for the generic fibre of $\left(\Def(\mathbb{G})^{\degen}_{\lvl \varpi^m}\right)_{\et}$.
Given a constructible abelian sheaf $\mathcal{A}$ on $\left(\Def(\mathbb{G})^{\degen}_{\lvl \varpi^m}\right)_{\et}$ of torsion orders prime to the characteristic of $A/\underline{m}$, the stalk of $R^s\Phi\left( \mathcal{A}\right)$ at the single point in the special fibre of $\Def(\mathbb{G})^{\degen}_{\lvl \varpi^m}$ is canonically isomorphic to the compactly supported cohomology group $H^s_c(\mathcal{D};\mathcal{A})$; see Lemma 2.5.1 of \cite{MR2383890} for this argument. Since $\mathcal{D}$ is paracompact, quasi-affine of dimension $n-1$, and regular (hence smooth, by perfectness of $k$) by Corollary 6.2 of \cite{MR1395723} and by Proposition \ref{generic fibre decomp}, $H^s_c(\mathcal{D};\mathcal{A})$ vanishes for $s<n-1$. The Poincar\'{e} duality of Theorem 7.3.1 of \cite{MR1259429} furthermore yields that $H^s_c(\mathcal{D};\mathcal{A})$ vanishes for $s>2n-2$. In the case $\mathcal{A} = \Sigma^m \widetilde{\pi_m\mathcal{F}}$, this vanishing result yields that the desired bound $N$ in the previous paragraph can be taken to be $2n-2$. 

The same cohomological dimension bound also yields that the spectral sequence \eqref{vanishing cycles stalk ss 1} is concentrated between two horizontal lines which are crossed by any sufficiently long differential. By a standard argument (as in \cite{MR1718076}), this yields strong convergence of the spectral sequence.

We will call \eqref{vanishing cycles stalk ss 1} the {\em vanishing cycles stalk spectral sequence.}

\section{The $\ell$-adic Jacquet-Langlands cohomology theory $JL_{\ell}(\mathbb{G})^*$.}
\label{l-adic JL cohomology theory}

As mentioned in the introduction, we recommend that readers browse the appendix (joint with M. Strauch) before beginning to read this section. The appendix does not logically depend on anything earlier in this paper, but beginning in this section, we will use some definitions and some theorems from the appendix.

From Theorem \ref{realizability thm 1}, we have a topological realization of the degenerating Lubin-Tate tower, i.e., for each formal $\hat{\mathbb{Z}}_p$-module $\mathbb{G}$ of positive finite $p$-height $n$, we have a presheaf $\mathcal{O}^{\topsym}_{\lvl p^m}$ of complex-oriented $E_{\infty}$-ring spectra on the small \'{e}tale site of $\Def(\mathbb{G})^{\degen}_{\lvl p^m}$ such that the formal group law of $\pi_0(\mathcal{O}^{\topsym}_{\lvl p^m}(U))$ arising from the complex orientation agrees with the universal formal group law on $\Gamma(U)$, for every \'{e}tale open $U$. From Theorem \ref{coh splitting thm}, we know that $H^i_c(\Def(\mathbb{G})^{}_{\lvl p^m};\overline{\mathbb{Q}}_{\ell})$  is a $GL_n(\mathbb{Z}/p^m\mathbb{Z})\times \Aut(\mathbb{G})$-equivariant summand of $H^i_c(\Def(\mathbb{G})^{\degen}_{\lvl p^m};\overline{\mathbb{Q}}_{\ell})$. From Theorem \ref{strauch thm}, we know that the $\ell$-adic Jacquet-Langlands correspondence is realized in $\colim_m H^{n-1}_c(\Def(\mathbb{G})_{\lvl p^m};\overline{\mathbb{Q}}_{\ell})$, and by Theorem \ref{coh splitting thm}, we know that the $\ell$-adic Jacquet-Langlands correspondence is also realized---although with some redundant summands!---in $\colim_m H^{n-1}_c(\Def(\mathbb{G})^{\degen}_{\lvl p^m};\overline{\mathbb{Q}}_{\ell})$. In this section we bring these facts together.

For each \'{e}tale open $U$ of $\Def(\mathbb{G})^{\degen}_{\lvl p^m}$, the homotopy groups $\pi_*(\mathcal{O}^{\topsym}_{\lvl p^m}(U))$ form a $\hat{\mathbb{Z}}_p$-module. 
In order to relate $\mathcal{O}^{\topsym}_{\lvl p^m}$ to the $\ell$-adic cohomology of $\Def(\mathbb{G})^{\degen}_{\lvl p^m}$, we must find some well-behaved operation which would take as input the presheaf $\mathcal{O}^{\topsym}_{\lvl p^m}$ of {\em $p$-adically complete} spectra, and produce as output a suitable presheaf of {\em $\ell$-adically complete} spectra.

This is a bit unusual: there are not many well-studied operations on spectra which have such an effect. However, the algebraic $K$-theory spectrum $\mathcal{K}(R)$ of an $E_{\infty}$-ring spectrum is another $E_{\infty}$-ring spectrum, and even if $\pi_*(R)$ is $p$-adically complete, $\pi_*(\mathcal{K}(R))$ will almost never be $p$-adically complete. Instead, it is typical to have $\pi_0(\mathcal{K}(R))\cong \mathbb{Z}$. We propose to simply apply the algebraic $K$-theory functor $\mathcal{K}$ sections-wise to $\mathcal{O}^{\topsym}_{\lvl p^m}$, yielding the presheaf\footnote{Sometimes the term {\em spectral Deligne-Mumford stack} is used to refer to a presheaf of connective $E_{\infty}$-ring spectra $\mathcal{F}$ on the small \'{e}tale site $\mathcal{X}_{\et}$, where $\mathcal{X}$ is a Deligne-Mumford stack in the standard sense (e.g. \cite{MR1771927}), and where the presheaf of commutative rings $\pi_0\circ \mathcal{F}$ on $\mathcal{X}_{\et}$ coincides with the structure sheaf of $\mathcal{X}_{et}$. See \cite{MR2597740} for discussion of this perspective. Our presheaf $\mathcal{O}^{\topsym}_{\lvl p^m}$, constructed in Theorem \ref{realizability thm 1}, satisfies all these conditions except the condition that each of the ring spectra $\mathcal{F}(U)$ is connective. We might say that $\mathcal{O}^{\topsym}_{\lvl p^m}$ is a ``nonconnective spectral Deligne-Mumford stack.'' We point this out only to be clear that $\mathcal{K}\circ \mathcal{O}^{\topsym}_{\lvl p^m}$ is {\em not} a spectral Deligne-Mumford stack, not even in the nonconnective sense, since $\pi_0\left(\left(\mathcal{K}\circ \mathcal{O}^{\topsym}_{\lvl p^m}\right)(U)\right) \cong K_0(\mathcal{O}^{\topsym}_{\lvl p^m}(U))$ is not a $\Gamma(U)$-module in any natural way.} of $E_{\infty}$-ring spectra $\mathcal{K}\circ \mathcal{O}^{\topsym}_{\lvl p^m}$ on the small \'{e}tale site of $\Def(\mathbb{G})^{\degen}_{\lvl p^m}$. After $\ell$-adic completion, we then have a presheaf of spectra on $\Def(\mathbb{G})^{\degen}_{\lvl p^m}$ whose homotopy groups are $\hat{\mathbb{Z}}_{\ell}$-modules.

Using algebraic $K$-theory in this way seems ridiculous. Converting presheaves of $p$-complete spectra into presheaves of $\ell$-complete spectra is not anywhere among the usual applications for algebraic $K$-theory! At a glance, it appears that this approach ought not to lead anywhere: owing to the notorious difficulty in calculating algebraic $K$-groups, one would expect it might be impossible to calculate or prove anything about vanishing cycles of the presheaf of spectra $(\mathcal{K}\circ \mathcal{O}^{\topsym}_{\lvl p^m})^{\wedge}_{\ell}$. 

The surprise is that this ridiculous idea actually works quite well. In Theorem \ref{main thm}, below, we will show that the supercuspidal part of the Jacquet-Langlands correspondence for $GL_n(\mathbb{Q}_p)$ is realized in the rational homotopy groups of the stalk of the vanishing cycles of the presheaf of spectra $(\mathcal{K}\circ \mathcal{O}^{\topsym}_{\lvl p^m})^{\wedge}_{\ell}$. We will even be able to make very explicit calculations of the resulting representation of $GL_n(\hat{\mathbb{Z}}_p)\times \Aut(\mathbb{G})$ on on those rational homotopy groups in the homotopy colimit as $m\rightarrow \infty$, leading to the results on automorphic $L$-factors in Theorem \ref{main automorphic thm}. 
%
This is possible because we can give a very explicit description of the homotopy type of the $\ell$-adic completion of $\mathcal{K}\circ \mathcal{O}^{\topsym}_{\lvl p^m}$. Here is our analysis of that homotopy type:
\begin{theorem}\label{splitting of k-thy spectrum}
Fix a prime $\ell\neq p$ and a positive integer $j$. 
Let $U$ be a connected \'{e}tale open of $\Def(\mathbb{G})^{\degen}_{\lvl p^m}$. Write $k(U)$ for the residue field of the local ring $\Gamma(U)$. 
Then the spectrum $\left( S/\ell^j\wedge \left(\mathcal{K}\circ\mathcal{O}^{\topsym}_{\lvl p^m}\right)\right)(U)$ is weakly equivalent to \[\left( S/\ell^j\wedge \mathcal{K}(k(U))\right) \vee \Sigma \left( S/\ell^j\wedge \mathcal{K}(k(U))\right).\] This weak equivalence is natural in the variable $U$, i.e., the presheaf of spectra $S/\ell^j \wedge \left(\mathcal{K}\circ\mathcal{O}^{\topsym}_{\lvl p^m}\right)$ is weakly equivalent to the presheaf of spectra
\begin{align}
\label{presheaf 0934} U &\mapsto S/\ell^j\wedge \left( \mathcal{K}(k(U))\vee \Sigma \mathcal{K}(k(U))\right)\end{align} on the small \'{e}tale site of $\Def(\mathbb{G})^{\degen}_{\lvl p^m}$. This weak equivalence is also natural in the variable $j$, so that 
the presheaf of spectra $\left(\mathcal{K}\circ\mathcal{O}^{\topsym}_{\lvl p^m}\right)^{\wedge}_{\ell}$ is weakly equivalent to the presheaf of spectra
\begin{align}
\label{presheaf 0935} U &\mapsto \mathcal{K}(k(U))^{\wedge}_{\ell}\vee \Sigma \mathcal{K}(k(U))^{\wedge}_{\ell}.\end{align}
\end{theorem}
\begin{proof}
By construction, we have the weak equivalence
\begin{align*}
 \Gamma(\mathcal{O}^{\topsym}_{\lvl p^m}) 
  &\cong F\left(\Sigma^{\infty}_+BC_{p^m}^n,E(\mathbb{G})\right) 
\end{align*}
and the isomorphisms
\begin{align*}
 \pi_0\Gamma(\mathcal{O}^{\topsym}_{\lvl p^m}) 
  &\cong \Gamma\left(\Def(\mathbb{G})^{\degen}_{\lvl p^m}\right),\\
 \pi_*\Gamma(\mathcal{O}^{\topsym}_{\lvl p^m}) 
  &\cong \pi_0\Gamma(\mathcal{O}^{\topsym}_{\lvl p^m})[u^{\pm 1}], \mbox{\ \ \ and}\\
 \pi_*\left( \mathcal{O}^{\topsym}_{\lvl p^m}(U)\right)
  &\cong \pi_*\Gamma(\mathcal{O}^{\topsym}_{\lvl p^m}) \otimes_{\pi_0\Gamma(\mathcal{O}^{\topsym}_{\lvl p^m})} \Gamma(U),
\end{align*}
with $u\in \pi_2\Gamma(\mathcal{O}^{\topsym}_{\lvl p^m})$, 
where $E(\mathbb{G})$ is the Morava $E$-theory spectrum of $\mathbb{G}$. 

Let $V^U_m$ denote the connective cover of the $E_{\infty}$-ring spectrum $\mathcal{O}^{\topsym}_{\lvl p^m}(U)$. Then the element $u\in \pi_2(V^U_m)$ has the property that $V^U_m[u^{-1}]\simeq \mathcal{O}^{\topsym}_{\lvl p^m}(U)$. In \cite{MR2413133}\footnote{The sequence \eqref{loc seq 1} is not spelled out explicitly in \cite{MR2413133}, but in \cite{MR4096617} it is pointed out (see page vii) that \eqref{loc seq 1} is produced by the methods of \cite{MR2413133}.}, Blumberg--Mandell use Waldhausen's fibration theorem \cite{MR802796} and a devissage argument to produce a fiber sequence
\begin{equation}\label{loc seq 1} \mathcal{K}\left(\pi_0(E(\mathbb{G}))\right) \rightarrow \mathcal{K}(BP_n)\rightarrow\mathcal{K}(E(\mathbb{G})),\end{equation}
where $BP_n$ is the connective cover of $E(\mathbb{G})$.
The same argument applies to the $BP_n$-algebra $V_m^U$ and the $E(\mathbb{G})$-algebra $\mathcal{O}^{\topsym}_{\lvl p^m}(U)$, yielding a fiber sequence
\begin{equation}\label{loc seq 2} \mathcal{K}\left(\pi_0(\mathcal{O}^{\topsym}_{\lvl p^m}(U)\right)\rightarrow \mathcal{K}(V_m^U) \rightarrow\mathcal{K}\left(\mathcal{O}^{\topsym}_{\lvl p^m}(U)\right).\end{equation}

Since $V^U_m$ is connective, the work of Dundas--Goodwillie--McCarthy \cite{MR3013261} yields a homotopy pullback square
\begin{equation}\label{htpy pullback 1}\xymatrix{
 \mathcal{K}\left(V^U_m\right)\wedge S/\ell^j \ar[r]\ar[d] &
  TC\left(V^U_m\right) \wedge S/\ell^j \ar[d] \\
 \mathcal{K}\left(\pi_0V^U_m\right)\wedge S/\ell^j \ar[r] &
  TC\left(\pi_0V^U_m\right) \wedge S/\ell^j.
}\end{equation}
Since $\ell\neq p$, 
the topological cyclic homology spectra in \eqref{htpy pullback 1} 
become contractible after reduction modulo $\ell^j$; see \cite{MR1341845} or \cite{MR3904731} for these properties of topological cyclic homology. Hence the natural map 
\[\mathcal{K}\left(V^U_m\right) \rightarrow \mathcal{K}\left(\pi_0V^U_m\right)\cong \mathcal{K}\left(\pi_0\mathcal{O}^{\topsym}_{\lvl p^m}(U)\right)\]
is a mod $\ell^j$ equivalence. 

We have $\pi_0\mathcal{O}^{\topsym}_{\lvl p^m}(U)\cong \Gamma(U)$ by Theorem \ref{realizability thm 1}, a complete local $W(k)$-algebra with residue field $k(U)$. Hence, by Gabber rigidity (\cite{MR1156502}, or see Theorem IV.2.10 of \cite{MR3076731} for a textbook reference), reduction to the residue field induces a mod $\ell^j$ weak equivalence 
\begin{align*}
 \mathcal{K}(\pi_0\mathcal{O}^{\topsym}_{\lvl p^m}(U)) &\stackrel{}{\longrightarrow} \mathcal{K}(k(U)).\end{align*}
The upshot is that, after smashing with $S/\ell^j$, the localization sequence \eqref{loc seq 2} reduces to a fiber sequence
\begin{equation}\label{loc seq 3} S/\ell^j\wedge \mathcal{K}(k(U)) \rightarrow  S/\ell^j\wedge \mathcal{K}(k(U)) \rightarrow S/\ell^j\wedge \left(\mathcal{K}\left(\mathcal{O}^{\topsym}_{\lvl p^m}(U)\right) \right).\end{equation}
In a sufficiently small \'{e}tale open $U$, $k(U)$ will be separably closed, so that $S/\ell^j\wedge \mathcal{K}(k(U))\simeq S/\ell^j\wedge ku$ by Suslin rigidity \cite{MR934225}. A map $\hat{ku}_{\ell}\rightarrow \hat{ku}_{\ell}$ is determined by its effect on $\pi_0(\hat{KU}_{\ell})$, by the main result of \cite{MR0844906}, i.e., it is determined by a stable Adams operation in $\hat{\mathbb{Z}}^{\times}_{\ell}$.
Hence, for $k(U)$ separably closed, the map $\mathcal{K}(k(U))^{\wedge}_{\ell} \rightarrow  \mathcal{K}(k(U))^{\wedge}_{\ell}$ in \eqref{loc seq 3} is determined by its effect on $\pi_0$, i.e., by the effect of the map 
\begin{align}
\label{map 304895}  K_0\left(\pi_0(\mathcal{O}^{\topsym}_{\lvl p^m}(U));\hat{\mathbb{Z}}_{\ell}\right)
 & \rightarrow   K_0\left(V_m^U;\hat{\mathbb{Z}}_{\ell}\right).
\end{align}
The map \eqref{map 304895} is zero, by inspection of the long exact sequence in homotopy induced by \eqref{loc seq 2}, and the observation that the map $K_0(V_m^U)\rightarrow K_0(\mathcal{O}^{\topsym}_{\lvl p^m}(U))$ sends $[V_m^U]$ to $[\mathcal{O}^{\topsym}_{\lvl p^m}(U)]$, and those elements generate the domain and codomain, respectively, of the map
\[ \mathbb{Q}\otimes_{\mathbb{Z}}K_0(V_m^U;\hat{\mathbb{Z}}_{\ell})\rightarrow \mathbb{Q}\otimes_{\mathbb{Z}}K_0(\mathcal{O}^{\topsym}_{\lvl p^m}(U);\hat{\mathbb{Z}}_{\ell}).\]

Finally, since the map \eqref{map 304895} is zero, the map of presheaves of spectra $\mathcal{K}(k(-))^{\wedge}_{\ell} \rightarrow  \mathcal{K}(k(-))^{\wedge}_{\ell}$ in \eqref{loc seq 3} is nulhomotopic on stalks. Hence, after sheafification (i.e., fibrant replacement in the Jardine model structure), it is nulhomotopic. Hence the fiber sequence \eqref{loc seq 3} splits $(\mathcal{K}\circ \mathcal{O}^{\topsym}_{\lvl p^m})^{\wedge}_{\ell}$ as the wedge sum \eqref{presheaf 0935}. Since the mod $\ell^j$ reduction depends only on the mod $\ell^j$ reduction of the $\ell$-adic completion, the wedge splitting \eqref{presheaf 0935} implies the wedge splitting \eqref{presheaf 0934}.
\end{proof}

\begin{remark}\label{ht 1 aut(G) triviality}
In \cref{l-adic TJL duality}, we will need to understand the action of $\Aut(\mathbb{G})$ on $\mathcal{K}(E(\mathbb{G}))^{\wedge}_{\ell}$, where $E(\mathbb{G})$ the Morava $E$-theory spectrum of a $p$-height $1$ formal group $\mathbb{G}$. 
It follows from the proof of Theorem \ref{splitting of k-thy spectrum} that this action is trivial. However, for $\mathbb{G}$ of $p$-height $>1$, the action is no longer trivial.

This deserves some explanation. Write $k$ for the field over which $\mathbb{G}$ is defined. As in the proof of Theorem \ref{splitting of k-thy spectrum}, we will write $BP_n$ for the connective cover of $E(\mathbb{G})$. Now use the localization sequence \eqref{loc seq 3}: of the two wedge summands $\mathcal{K}(k)^{\wedge}_{\ell}$ and $\Sigma\mathcal{K}(k)^{\wedge}_{\ell}$ of $\mathcal{K}(E(\mathbb{G}))^{\wedge}_{\ell}$, the action of $\Aut(\mathbb{G})$ on the bottom summand $\mathcal{K}(k)^{\wedge}_{\ell}$ is trivial since $\Aut(\mathbb{G})$ acts trivially on $\pi_0E(\mathbb{G})\cong \pi_0BP_n$ and since the natural map \begin{align}\label{map 30494}\mathcal{K}(BP_n)^{\wedge}_{\ell}&\rightarrow \mathcal{K}(k)^{\wedge}_{\ell}\end{align} is an equivalence. The map \eqref{map 30494} is induced by taking the Postnikov $0$-sections 
\begin{align}
\label{postnikov 0-sect} BP_n&\rightarrow H\pi_0BP_n\end{align}
and then the reduction map 
\begin{equation}
\label{reduction 4} \pi_0BP_n\stackrel{\cong}{\longrightarrow} W(k) \twoheadrightarrow k.
\end{equation} 
To be explicit, the map induced by \eqref{postnikov 0-sect} in $\ell$-complete algebraic $K$-theory is an equivalence due to a Dundas-Goodwillie-McCarthy pullback square, a special case of \eqref{htpy pullback 1}, and $\ell$-adic vanishing of topological cyclic homology of a $p$-complete ring spectrum. Meanwhile, the map induced by \eqref{reduction 4} is an equivalence in $\ell$-complete algebraic $K$-theory, by Gabber rigidity. Both arguments are special cases of the argument given in the proof of Theorem \ref{splitting of k-thy spectrum}.

The triviality of the action of $\Aut(\mathbb{G})$ on the top summand $\Sigma\mathcal{K}(k)^{\wedge}_{\ell}$ of $\mathcal{K}(E(\mathbb{G}))^{\wedge}_{\ell}$ takes a bit more work to see. Recall that the splitting $\mathcal{K}(E(\mathbb{G}))^{\wedge}_{\ell}\simeq \mathcal{K}(k)^{\wedge}_{\ell} \vee \Sigma\mathcal{K}(k)^{\wedge}_{\ell}$ was produced by showing that a certain fiber sequence, \eqref{loc seq 1} from the proof of Theorem \ref{splitting of k-thy spectrum}, splits. That fiber sequence arises by application of Waldhausen's fibration theorem. In this case, the relevant Waldhausen categories are:
\begin{itemize}
\item the Waldhausen category $f\Mod(BP_n)$ of finite cell $BP_n$-module spectra
\item the category $f\Mod(E(\mathbb{G}))$ of finite cell $E(\mathbb{G})$-module spectra,
\item and the category $f\Mod(BP_n)^{wE(\mathbb{G})}$ of finite cell $BP_n$-module spectra which become contractible upon inverting the Bott element $u\in \pi_2BP_n$.
\end{itemize}
In all three of these Waldhausen categories, the cofibrations can be taken to be $w$-cofibrations\footnote{See section VI.3 of \cite{MR1417719} for these ideas.}, and the weak equivalences can be taken to be homotopy equivalences.
An application of Waldhausen's fibration theorem, Theorem 1.6.4 of \cite{MR802796}, yields a fiber sequence of Waldhausen $K$-theory spectra
\[ \mathcal{K}\left(f\Mod(BP_n)^{wE(\mathbb{G})}\right)
\rightarrow \mathcal{K}\left(f\Mod(BP_n)\right)
\rightarrow \mathcal{K}\left(f\Mod(E(\mathbb{G}))\right).\]
Waldhausen's ``spherical object argument'' from section 1.7 of \cite{MR802796}, is an analogue of Quillen's d\'{e}vissage which is adapted for use in Waldhausen categories. An application of the spherical object argument yields that the spectrum $\mathcal{K}\left(f\Mod(BP_n)^{wE(\mathbb{G})}\right)$ is weakly equivalent to $\mathcal{K}\left(f\Mod(\pi_0BP_n)\right)$, which in turn (by, for example, Theorem IV.4.3 of \cite{MR1417719}) is weakly equivalent to $\mathcal{K}\left(\pi_0BP_n\right)$, the algebraic $K$-theory spectrum of the ordinary ring $\pi_0BP_n$. Checking the hypotheses of Waldhausen's fibration theorem and spherical objects argument here is nontrivial, but since $E(\mathbb{G})$ is an \'{e}tale extension of the $p$-complete topological $K$-theory spectrum $KU$, the verification is the same as that already carried out in \cite{MR2413133}. Finally, by another application of Gabber rigidity, $\mathcal{K}(\pi_0BP_n)$ agrees with $\mathcal{K}(k)$ after $\ell$-completion.

We are really repeating ourselves, since the above is only a special case of the argument for Theorem \ref{splitting of k-thy spectrum}. The point here is that the action of $\Aut(\mathbb{G})$ on the top summand $\Sigma\mathcal{K}(k)^{\wedge}_{\ell}$ of $\mathcal{K}(E(\mathbb{G}))^{\wedge}_{\ell}$ is determined by the action of $\Aut(\mathbb{G})$ on the Waldhausen category $f\Mod(BP_n)^{wE(\mathbb{G})}$. The action of $\Aut(\mathbb{G})$ on $f\Mod(BP_n)^{wE(\mathbb{G})}$ is nontrivial, but by the spherical object argument, the $K$-theory spectrum of $f\Mod(BP_n)^{wE(\mathbb{G})}$ agrees with the $K$-theory spectrum of the category of those finite cell $BP_n$-module spectra on which $u\in \pi_2BP_n$ acts nulhomotopically ({\em not} just nilpotently). Such $BP_n$-module spectra are naturally equivalent to $H\pi_0BP_n$-module spectra, and since $\Aut(\mathbb{G})$ acts trivially on $\pi_0BP_n$, we have also that $\Aut(\mathbb{G})$ acts trivially on $H\pi_0BP_n$, hence trivially on the top summand $\Sigma\mathcal{K}(k)^{\wedge}_{\ell}$ of $\mathcal{K}(E(\mathbb{G}))^{\wedge}_{\ell}$. 

If the $p$-height of $\mathbb{G}$ is greater than $1$, then the above argument does {\em not} yield triviality of the action of $\Aut(\mathbb{G})$ on $\mathcal{K}(E(\mathbb{G}))^{\wedge}_{\ell}$. Indeed, in case $\Ht_p(\mathbb{G})>1$, the group $\Aut(\mathbb{G})$ acts nontrivially on $\pi_0E(\mathbb{G})$, so in general one ought to expect that $\Aut(\mathbb{G})$ acts nontrivially on $\mathcal{K}(E(\mathbb{G}))^{\wedge}_{\ell}$.
\end{remark}

Now we are ready for the definition of Jacquet-Langlands cohomology. It will be the generalized cohomology theory represented by a certain spectrum $JL_{\ell}(\mathbb{G})$. Let $p$ be a prime, and let $\mathbb{G}$ be a one-dimensional formal group law, of finite positive $p$-height $n$, over some finite field $k$ of $\mathbb{F}_p$. Fix a prime $\ell\neq p$. The representing spectrum $JL_{\ell}(\mathbb{G})$ will be built by composing a sequence of constructions already discussed in this paper. Because we compose a large number of constructions, it is much easier to read the resulting formulas if we use prefix notation for each construction. Hence we write $\St_x\mathcal{F}$ for the stalk of a sheaf $\mathcal{F}$ at a point $x$.
We also write $\Fib \mathcal{F}$ for fibrant replacement, in the Jardine model structure, 
of a presheaf $\mathcal{F}$ of spectra. Step by step, here\footnote{The author thinks that this step-by-step presentation is the clearest way to present the construction, even though it is repetitive. We apologize to readers who are annoyed by the repetitiveness.} is the process:
\begin{enumerate}
\item For each nonnegative integer $m$, we have the presheaf $\mathcal{O}^{\topsym}_{\lvl p^m}$ of $E_{\infty}$-ring spectra on the small \'{e}tale site of $\Def(\mathbb{G})^{\degen}_{\lvl p^m}$. 
\item Apply algebraic $K$-theory. Now, for each nonnegative integer $m$, we have the presheaf of $E_{\infty}$-ring spectra $\mathcal{K}\circ \mathcal{O}^{\topsym}_{\lvl p^m}$ on the small \'{e}tale site of $\Def(\mathbb{G})^{\degen}_{\lvl p^m}$. 
\item Reduce modulo $\ell^j$, i.e., smash with the mod $\ell^j$ Moore spectrum $S/\ell^j$. Now, for each nonnegative integer $m$ and each positive integer $j$, we have the presheaf of spectra $S/\ell^j\wedge \left(\mathcal{K}\circ \mathcal{O}^{\topsym}_{\lvl p^m}\right)$ on the small \'{e}tale site of $\Def(\mathbb{G})^{\degen}_{\lvl p^m}$. 
\item Fibrantly replace, and apply the vanishing cycles functor $\Phi$. Now, for each nonnegative integer $m$ and each positive integer $j$, we have the presheaf of 
spectra $\Phi \Fib \left( S/\ell^j \wedge \left(\mathcal{K}\circ \mathcal{O}^{\topsym}_{\lvl p^m}\right)\right)$ on the small \'{e}tale site of the special fiber $\left(\Def(\mathbb{G})^{\degen}_{\lvl p^m}\right)_s$ of $\Def(\mathbb{G})^{\degen}_{\lvl p^m}$. 
\item Take the stalk at the closed point $x$ of the special fibre $\left(\Def(\mathbb{G})^{\degen}_{\lvl p^m}\right)_s$. Now, for each positive integer $j$, we have a spectrum \[ \St_x\Phi \Fib \left( S/\ell^j \wedge \left(\mathcal{K}\circ \mathcal{O}^{\topsym}_{\lvl p^m}\right)\right). \]
with a natural action of $GL_n(\mathbb{Z}/p^m\mathbb{Z})\times \Aut(\mathbb{G})$.
\item All the above constructions are natural in the variable $j$. Take the homotopy limit over $j$ to get an $\ell$-complete spectrum 
\[ \underset{j\rightarrow\infty}{\holim} \St_x\Phi \Fib \left( S/\ell^j \wedge \left(\mathcal{K}\circ \mathcal{O}^{\topsym}_{\lvl p^m}\right)\right)\]
for each nonnegative integer $m$.
\item All the above constructions are natural in the variable $m$. Take the homotopy colimit
\begin{equation}
\label{jl def 1}
 \underset{m\rightarrow\infty}{\hocolim}\ \underset{j\rightarrow\infty}{\holim} \St_x \Phi \Fib \left( S/\ell^j \wedge \left(\mathcal{K}\circ \mathcal{O}^{\topsym}_{\lvl p^m}\right)\right).
\end{equation}
\end{enumerate}

\begin{definition}\label{def of jl homology}
By the {\em $\ell$-adic Jacquet-Langlands cohomology associated to $\mathbb{G}$}, we mean the generalized cohomology theory represented by the $GL_n(\hat{\mathbb{Z}}_p)\times \Aut(\mathbb{G})$-equivariant spectrum \eqref{jl def 1}. We write $JL_{\ell}(\mathbb{G})$ for this spectrum.
\end{definition}

\begin{remark}
Beginning in this section, we shall sometimes have to refer to the homotopy fixed points of a group $G$ acting on a spectrum $X$. We will use the standard notation $X^{hG}$ for the homotopy fixed points. 

The group $G$ is profinite, hence by $X^{hG}$ one ought to mean the {\em continuous} homotopy fixed points. The paper \cite{MR2030586} is the standard reference, but it is a bit tricky to define continuous homotopy fixed points in such a way that $\mathcal{K}(E(\mathbb{G}))^{h\Aut(\mathbb{G})}$ has desirable properties (e.g. those predicted by some forms of the Ausoni-Rognes redshift conjecture \cite{AR08}). See Davis's preprint \cite{davis2020construction} for discussion and progress on this point. In Remark 1.8 of \cite{davis2023homotopy}, Davis comments also on constructions of $\mathcal{K}(E(\mathbb{G}))^{h\Aut(\mathbb{G})}$ using condensed methods, and using pyknotic methods. However, for what we do in this paper, we do not need very much, so even a very na\"{i}ve model for continuous homotopy fixed points will do: we will only need to consider $(\mathcal{K}(E(\mathbb{G}))^{\wedge}_{\ell})^{h\Aut(\mathbb{G})}$, the continuous homotopy fixed points of the $\Aut(\mathbb{G})$ action on the $\ell$-adic completion of $\mathcal{K}(E(\mathbb{G}))$, and the only property we really need these continuous homotopy fixed points to have is that the natural map 
\begin{eqnarray}\nonumber\pi_*\left( F_{H\overline{\mathbb{Q}}_{\ell}}\left( H\overline{\mathbb{Q}}_{\ell}\wedge_{\hat{S}_{\ell}} X,H\overline{\mathbb{Q}}_{\ell}\wedge_{\hat{S}_{\ell}} JL_{\ell}(\mathbb{G})\right)^{h\Aut(\mathbb{G})}\right)\ \ \ \ \ \ \ \ \ \ \ \ \ \ \ \ \ \ \ \ \ \ \ \ 
 \\ \label{weak continuity map} \ \ \ \ \ \ \ \ \ \ \ \ \ \ \ \ \ \ \ \ \ \rightarrow \hom_{\overline{\mathbb{Q}}_{\ell}}\left( \overline{\mathbb{Q}}_{\ell} \otimes_{\hat{\mathbb{Z}}_{\ell}}\pi_*X,\overline{\mathbb{Q}}_{\ell} \otimes_{\mathbb{Q}_{\ell}} H_*(JL_{\ell}(\mathbb{G});\mathbb{Q})\right)^{\Aut(\mathbb{G})}
\end{eqnarray}
is an isomorphism, where $X$ is a finite $\hat{S}_{\ell}$-module with a suitable action of $\Aut(\mathbb{G})$, and where the spectrum $JL_{\ell}(\mathbb{G})$ is defined below, in Definition \ref{def of jl homology}. Since $\Aut(\mathbb{G})$ is an extension of the unit group of $\mathbb{F}_{p^n}$ by the pro-$p$-group $\strict\Aut(\mathbb{G})$ of strict automorphisms of $\mathbb{G}$, the cohomology of $\Aut(\mathbb{G})$ with coefficients in any $\ell$-power-torsion topological $\Aut(\mathbb{G})$-module will be $\ell$-power-torsion in positive degrees, hence annihilated by taking a colimit over repeated multiplication by $\ell$. The point is that the continuous cohomology of $\Aut(\mathbb{G})$ with coefficients in 
\begin{align*}\pi_*F_{H\overline{\mathbb{Q}}_{\ell}}\left( H\overline{\mathbb{Q}}_{\ell}\wedge_{\hat{S}_{\ell}} X,H\overline{\mathbb{Q}}_{\ell}\wedge_{\hat{S}_{\ell}} JL_{\ell}(\mathbb{G})\right) &\cong \left[\Sigma^* X, H\overline{\mathbb{Q}}_{\ell}\wedge_{\hat{S}_{\ell}} JL_{\ell}(\mathbb{G})\right]\end{align*}
vanishes in positive degrees, so any reasonable model for the continuous homotopy $\Aut(\mathbb{G})$-fixed points will yield the isomorphism \eqref{weak continuity map}.

This paper does not investigate $p$-adic (rather than $\ell$-adic) TJL duality, but any such investigations would likely require a more subtle model for $h\Aut(\mathbb{G})$ of the kinds discussed and constructed by Davis, particularly for the sake of having suitable fixed-point spectral sequences.
\end{remark}

Now here is the main theorem of this section:
\begin{theorem}\label{main thm}
Let $n$ be the $p$-height of a one-dimensional formal group law $\mathbb{G}$ over $\overline{\mathbb{F}}_p$.
Suppose $n$ is finite. 
Then the supercuspidal part of the $\ell$-adic Jacquet-Langlands correspondence for $\mathbb{Q}_p$ is realized in the rational stable homotopy group 
\begin{align}
\label{ratl htpy 0 1}
 \overline{\mathbb{Q}}_{\ell}\otimes_{\hat{\mathbb{Z}}_{\ell}} \pi_{1-n}\left( JL_{\ell}(\mathbb{G})\right).
\end{align}

In more detail: for each supercuspidal irreducible $\overline{\mathbb{Q}}_{\ell}$-linear representation $\pi$ of $GL_n(\mathbb{Q}_p)$, the $\overline{\mathbb{Q}}_{\ell}$-linear representation $\left(\pi\mid_{GL_n(\hat{\mathbb{Z}}_p)}\right)\otimes \mathcal{JL}(\pi)$ of $GL_n(\hat{\mathbb{Z}}_p)\times \Aut(\mathbb{G})$ occurs as a summand in \eqref{ratl htpy 0 1}, where $\mathcal{JL}(\pi)$ is the representation of 
\begin{equation*}\Aut(\mathbb{G})\cong \mathcal{O}^{\times}_{D_{1/n,\mathbb{Q}_p}} \subseteq D^{\times}_{1/n,\mathbb{Q}_p}\end{equation*} associated to $\pi$ by the Jacquet-Langlands correspondence constructed in \cite{MR0771672} and \cite{MR0700135}.
\end{theorem}
\begin{proof}
We begin with a modification to the construction of the vanishing cycles stalk spectral sequence \eqref{vanishing cycles stalk ss 1}. Write $\mathcal{F}_j$ as an abbreviation for the fibrant replacement of the presheaf of spectra $S/\ell^j \wedge \left( \mathcal{K}\circ \mathcal{O}^{top}_{\lvl p^m}\right)$ on $\left(\Def(\mathbb{G})^{\degen}_{\lvl p^m}\right)_{\et}$. We have a sequence of fibrant presheaves of spectra
\[ \dots \rightarrow \mathcal{F}_2 \rightarrow \mathcal{F}_1 \rightarrow \mathcal{F}_0.\]
Let $\Post(\mathcal{F}_j)$ denote the 
fibrant sectionswise Postnikov tower of $\mathcal{F}_j$. Apply the stalk functor $\St_x$. Since all these constructions are functorial in the choice of presheaf, we have a sequence of towers of fibrations of spectra
\[ \dots \rightarrow \St_x\Phi\Post(\mathcal{F}_2) \rightarrow \St_x\Phi\Post(\mathcal{F}_1) \rightarrow \St_x\Phi\Post(\mathcal{F}_0),\]
consequently a levelwise homotopy limit $\holim_{j\rightarrow\infty} \St_x\Phi\Post(\mathcal{F}_j)$.

Consider the spectral sequence obtained by taking homotopy groups of the tower of fibrations $\holim_{j\rightarrow\infty} \St_x\Phi \Post(\mathcal{F}_j)$. We will call this spectral sequence the ``modified vanishing cycles stalk spectral sequence.'' It converges to 
\[ \pi_*\left(\underset{j\rightarrow\infty}{\holim}\St_x \Phi  \left( S/\ell^j \wedge \left(\mathcal{K}\circ \mathcal{O}^{\topsym}_{\lvl p^m}\right)\right)\right),\]
by the same analysis we gave, in \cref{Nearby cycles and vanishing cycles...}, to identify the abutment of the vanishing cycles stalk spectral sequence \eqref{vanishing cycles stalk ss 1}. 

The $E_1$-term of the modified vanishing cycles stalk spectral sequence is computable by a Milnor sequence:
\begin{align}\label{milnor seq 23} 0 \rightarrow
 \lim_{j\rightarrow\infty}{}^1 \left( R^s\Phi \widetilde{\pi_{t+1}(\mathcal{F}_j)}\right)_x \rightarrow
 E_1^{s,t} 
 \rightarrow \lim_{j\rightarrow\infty} \left(R^s\Phi \widetilde{\pi_t(\mathcal{F}_j)}\right)_x \rightarrow 0.\end{align}
The $\lim^1$-term in \eqref{milnor seq 23} vanishes, since $\lim^1$ vanishes on sequences of finite abelian groups, and 
since each $\left(R^s\Phi \widetilde{\pi_t(\mathcal{F}_j)}\right)_x$ is finite as a consequence of Proposition 4.4 of \cite{MR1395723}; this uses the assumption that $k$ is separably closed, and the finiteness of the disjoint union in Proposition \ref{generic fibre decomp}.

By Theorem \ref{splitting of k-thy spectrum}, for each integer $t$ and each nonnegative integer $j$, $\pi_t(\mathcal{F}_j)$ is a constructible sheaf on $\left(\Def(\mathbb{G})^{\degen}_{\lvl p^m}\right)_{\et}$, and it is of torsion order prime to $p$ (since $\ell\neq p$).
The group $\left(R^s\Phi \widetilde{\pi_t(\mathcal{F}_j)}\right)_x$ vanishes for $s<n-1$ and for $s>2n-2$, by the vanishing argument given in \cref{Nearby cycles and vanishing cycles...}. Consequently $E_1^{*,*}$ is concentrated between the $s=n-1$ and $s=2n-2$ lines, and the spectral sequence converges strongly.

We claim that the modified vanishing cycles stalk spectral sequence in fact has no nonzero differentials. The argument is very simple: Theorem \ref{splitting of k-thy spectrum} splits $\mathcal{F}_j$ as a wedge of two presheaves of spectra, $\mathcal{F}_j \simeq \mathcal{F}_j^{\prime} \vee \Sigma \mathcal{F}_j^{\prime}$. Each of the operations in the construction of Jacquet-Langlands homology in Definition \eqref{def of jl homology} is functorial on presheaves of spectra, and consequently the modified vanishing cycles stalk spectral sequence splits as the direct sum of two spectral sequences:
\begin{itemize}
\item a spectral sequence converging to $\pi_*\underset{j\rightarrow\infty}{\holim} \St_x \Phi \Fib \left(\mathcal{F}_j^{\prime}\right)$, 
\item and an isomorphic copy of that same spectral sequence, with all internal degrees ($t$) increased by one.
\end{itemize}
Since $\widetilde{\pi_*(\mathcal{F}_j^{\prime})}$ is concentrated in even degrees, the first summand spectral sequence collapses at $E_1$ for degree reasons. The second spectral sequence collapses at $E_1$ by the same argument.

Consequently $\pi_{1-n}\left( JL_{\ell}(\mathbb{G})\right)$ has a finite filtration in which the filtration quotients are
\begin{align} 
\label{summand 1}
\colim_{m\rightarrow\infty}H^{n-1}_c\left(\Def(\mathbb{G})^{\degen}_{\lvl p^m};\hat{\mathbb{Z}}_{\ell}(0)\right),\ &
\colim_{m\rightarrow\infty}H^{n}_c\left(\Def(\mathbb{G})^{\degen}_{\lvl p^m};\hat{\mathbb{Z}}_{\ell}(0)\right), \\
\label{summand 2}
\colim_{m\rightarrow\infty}H^{n+1}_c\left(\Def(\mathbb{G})^{\degen}_{\lvl p^m};\hat{\mathbb{Z}}_{\ell}(1)\right),\ &
\colim_{m\rightarrow\infty}H^{n+2}_c\left(\Def(\mathbb{G})^{\degen}_{\lvl p^m};\hat{\mathbb{Z}}_{\ell}(1)\right), \\
\label{summand 3}
\colim_{m\rightarrow\infty}H^{n+3}_c\left(\Def(\mathbb{G})^{\degen}_{\lvl p^m};\hat{\mathbb{Z}}_{\ell}(2)\right),\ &
\colim_{m\rightarrow\infty}H^{n+4}_c\left(\Def(\mathbb{G})^{\degen}_{\lvl p^m};\hat{\mathbb{Z}}_{\ell}(2)\right), \\
\dots,\  &\ \  \dots ,
\end{align}
eventually ending with the $\colim_{m\rightarrow\infty}H^{2n-2}_c$ summand. Rationally, the filtration splits $GL_n(\hat{\mathbb{Z}}_p)\times \Aut(\mathbb{G})$-equivariantly.
By Theorem \ref{coh splitting thm}, after tensoring up to $\overline{\mathbb{Q}}_{\ell}$, the bottommost filtration quotient has $\colim_{m\rightarrow\infty}H^{n-1}_c(\Def(\mathbb{G})_{\lvl p^m};\mathbb{Q}_{\ell})$ as an $GL_n(\hat{\mathbb{Z}}_p)\times \Aut(\mathbb{G})$-equivariant summand. This summand realizes the supercuspidal part of the $\ell$-adic Jacquet-Langlands correspondence, by \cite{MR2680204}, as explained in \cref{Review of the role...}.
\end{proof}
The middle-dimensional cohomology $\colim_{m\rightarrow\infty}H^{n-1}_c(\Def(\mathbb{G})_{\lvl p^m}; \overline{\mathbb{Q}}_{\ell})$ of the Lubin-Tate tower is where the supercuspidal representations and, by \cite{MR2680204}, {\em only} the supercuspidal representations of $GL_n$. In $\colim_{m\rightarrow\infty}H^{i}_c(\Def(\mathbb{G})_{\lvl p^m}; \mathbb{Q}_{\ell})$ for $i>n-1$, one instead finds with principal series representations and non-supercuspidal discrete series representations of $GL_n$.  It follows from the method of proof of Theorem \ref{main thm}---namely, the collapse of the modified vanishing cycles stalk spectral sequence and the splitting of the cohomology of $\Def(\mathbb{G})^{\degen}_{\lvl p^m}$ obtained in Theorem \ref{splitting of k-thy spectrum}---that each cohomology group $\colim_{m\rightarrow\infty}H_c^i(\Def(\mathbb{G})^{\degen}_{\lvl p^m};\overline{\mathbb{Q}}_{\ell})$ occurs $GL_n(\hat{\mathbb{Z}}_p)\times \Aut(\mathbb{G})$-equivariantly in the rational stable homotopy groups $\overline{\mathbb{Q}}_{\ell}\otimes_{\hat{\mathbb{Z}}_{\ell}} \pi_{*}\left( JL_{\ell}(\mathbb{G})\right)$ of the Jacquet-Langlands spectrum $JL_{\ell}(\mathbb{G})$. Consequently it is not only the supercuspidal part of the Jacquet-Langlands correspondence (coming from the middle-degree cohomology) which is realized in the rational stable homotopy of $JL_{\ell}(\mathbb{G})$: the summands providing the {\em rest} of the correspondence are all present as well, but they are spread out across various degrees, and present with many redundant copies and with various Tate twists, due to the degenerating Lubin-Tate tower having much more cohomology than just that coming from the natural copy of the classical Lubin-Tate tower sitting inside of it. See \cref{l-adic TJL duality} for an explicit example at height $1$.

As a consequence of Theorem \ref{main thm}, we have:
\begin{corollary}\label{main cor}
Let $n$ be the $p$-height of the one-dimensional formal group law $\mathbb{G}$ over $\overline{\mathbb{F}}_p$. 
Suppose $n$ is finite, suppose $\ell$ is a prime distinct from $p$, and suppose that $X$ is an $\ell$-complete spectrum equipped with an 
action of $\Aut(\mathbb{G})$. Then, for each integer $i$ and each irreducible subrepresentation $\rho$ of the $\overline{\mathbb{Q}}_{\ell}$-linear representation $\overline{\mathbb{Q}}_{\ell}\otimes_{\hat{\mathbb{Z}}_{\ell}} \pi_i(X)$ of $\Aut(\mathbb{G})$ such that the representation $\mathcal{JL}(\rho)$ of $GL_n$ is supercuspidal, $\mathcal{JL}(\rho)$ occurs as a factor in $\pi_{1-n-i}\left( F(X,H\overline{\mathbb{Q}}_{\ell}\wedge_{\hat{S}_{\ell}} JL_{\ell}(\mathbb{G}))^{h\Aut(\mathbb{G})}\right)$.
\end{corollary}
\begin{proof}
We have isomorphisms
\begin{align*}
 & \pi_{1-n-i}\left( F(X,H\overline{\mathbb{Q}}_{\ell}\wedge_{\hat{S}_{\ell}} JL_{\ell}(\mathbb{G}))^{h\Aut(\mathbb{G})}\right) \\
  &\cong \pi_{1-n-i}\left( F_{H\overline{\mathbb{Q}}_{\ell}}\left( H\overline{\mathbb{Q}}_{\ell}\wedge_{\hat{S}_{\ell}} X,H\overline{\mathbb{Q}}_{\ell}\wedge_{\hat{S}_{\ell}} JL_{\ell}(\mathbb{G})\right)^{h\Aut(\mathbb{G})}\right) \\
  &\cong \hom_{\overline{\mathbb{Q}}_{\ell}}\left( \overline{\mathbb{Q}}_{\ell} \otimes_{\hat{\mathbb{Z}}_{\ell}}\Sigma^{1-n-i} \pi_*X,\overline{\mathbb{Q}}_{\ell} \otimes_{\hat{\mathbb{Z}}_{\ell}} \pi_*JL_{\ell}(\mathbb{G})\right)^{\Aut(\mathbb{G})} \\
  &\cong \hom_{\overline{\mathbb{Q}}_{\ell}[\Aut(\mathbb{G})]}\left( \overline{\mathbb{Q}}_{\ell} \otimes_{\hat{\mathbb{Z}}_{\ell}}\Sigma^{1-n-i} \pi_*X,\overline{\mathbb{Q}}_{\ell} \otimes_{\hat{\mathbb{Z}}_{\ell}} \pi_*JL_{\ell}(\mathbb{G})\right).
\end{align*}
In particular, by Theorem \ref{main thm}, $\pi_{1-n-i}\left( F(X,H\overline{\mathbb{Q}}_{\ell}\wedge_{\hat{S}_{\ell}} JL_{\ell}(\mathbb{G}))^{h\Aut(\mathbb{G})}\right)$ contains 
\begin{equation}\label{hom 4349} \hom_{\overline{\mathbb{Q}}_{\ell}[\Aut(\mathbb{G})]}\left( \overline{\mathbb{Q}}_{\ell} \otimes_{\hat{\mathbb{Z}}_{\ell}} \pi_iX,\overline{\mathbb{Q}}_{\ell} \otimes_{\hat{\mathbb{Z}}_{\ell}} \pi_{1-n}JL_{\ell}(\mathbb{G})\right)
\end{equation}
as a $GL_n(\hat{\mathbb{Z}}_p)$-equivariant summand.
\end{proof}

\section{$\ell$-adic topological Jacquet-Langlands duality.}
\label{l-adic TJL duality}

We continue to assume that the formal group $\mathbb{G}$ is defined over a field of characteristic $p\neq \ell$. 
For any spectrum $X$ with an 
action of $\Aut(\mathbb{G})$, the mapping spectrum $F(X,JL_{\ell}(\mathbb{G}))$ has an action of $GL_n(\hat{\mathbb{Z}}_p)$, and its homotopy groups are $\hat{Z}_{\ell}$-modules. 
By Corollary \ref{main cor}, if we take continuous homotopy fixed-points of the action of $\Aut(\mathbb{G})$ on $F(X,H\overline{\mathbb{Q}}_{\ell}\wedge_{\hat{S}_{\ell}}JL_{\ell}(\mathbb{G}))$, the resulting spectrum has the $\ell$-adic Jacquet-Langlands ``dual'' of $\overline{\mathbb{Q}}_{\ell}\otimes_{\hat{\mathbb{Z}}_{\ell}}\pi_*X$ as a summand in its rational homology. 

If the homotopy groups of $X$ are already $p$-complete (not $\ell$-complete!), then $F(X,JL_{\ell}(\mathbb{G}))$ will be contractible. The same is true even if the homotopy groups of $X$ are merely assumed to be $\hat{\mathbb{Z}}_p$-modules. In algebraic topology, when do we find ourselves with a spectrum with an action of $\Aut(\mathbb{G})$, but whose homotopy groups are {\em not} $\hat{\mathbb{Z}}_p$-modules? The most natural idea is to simply smash a spectrum $Y$ with the Morava $E$-theory spectrum $E(\mathbb{G})$, but this idea will not work, since the homotopy groups of that smash product will be $\hat{\mathbb{Z}}_p$-modules.

When $\ell\neq p$, the author knows only of one reasonably natural source of non-Eilenberg-Mac Lane $\ell$-complete spectra with an action of $\Aut(\mathbb{G})$. Let $E$ be an $E_{\infty}$-ring spectrum equipped with an action of $\Aut(\mathbb{G})$; the most obvious example is $E = E(\mathbb{G})$. Then the $\ell$-complete algebraic $K$-theory spectrum $\mathcal{K}(E)^{\wedge}_{\ell}$ is also an $E_{\infty}$-ring spectrum, and has an action of $\Aut(\mathbb{G})$ by the functoriality of $\mathcal{K}$, but $\mathcal{K}(E)^{\wedge}_{\ell}$ is also $\ell$-complete. The action of $\Aut(\mathbb{G})$ on $\mathcal{K}(E(\mathbb{G}))$ has been studied very recently: see \cite{davis2023continuous} and \cite{davis2023homotopy}, for example.

By Corollary \ref{main cor}, whenever the action of $\Aut(\mathbb{G})$ on the rationalized $\ell$-complete algebraic $K$-groups \[ K_*(E;\overline{\mathbb{Q}}_{\ell}) := \overline{\mathbb{Q}}_{\ell}\otimes_{\hat{\mathbb{Z}}_{\ell}}\pi_*\left( \mathcal{K}(E)^{\wedge}_{\ell}\right)\] has a summand which is the Jacquet-Langlands dual $\mathcal{JL}(\pi)$ of some irreducible supercuspidal $\overline{\mathbb{Q}}_{\ell}$-linear representation of $GL_n$, the representation $\pi$ occurs as a summand in the homotopy groups of $F(\mathcal{K}(E)^{\wedge}_{\ell},H\overline{\mathbb{Q}}_{\ell}\wedge_{\hat{S}_{\ell}} JL_{\ell}(\mathbb{G}))^{h\Aut(\mathbb{G})}$.

More broadly: $\mathcal{K}(E)^{\wedge}_{\ell}$ represents the generalized homology theory which sends a spectrum $X$ to 
\begin{align}\label{gen hom thy 1} \left(\mathcal{K}(E)^{\wedge}_{\ell}\right)_*(X) &= \pi_*(\mathcal{K}(E)^{\wedge}_{\ell}\wedge X).\end{align}
This generalized homology theory takes its values in $\hat{\mathbb{Z}}_{\ell}$-linear representations of $\Aut(\mathbb{G})$. We also have the generalized cohomology theory which sends a spectrum $X$ to 
\begin{align}\label{gen coh thy 1} \pi_*\left( F(\mathcal{K}(E)^{\wedge}_{\ell}\wedge X,H\overline{\mathbb{Q}}_{\ell}\wedge_{\hat{S}_{\ell}} JL_{\ell}(\mathbb{G}))^{h\Aut(\mathbb{G})}\right).\end{align}
This generalized cohomology theory takes its values in $\overline{\mathbb{Q}}_{\ell}$-linear representations of $GL_n(\hat{\mathbb{Z}}_p)$. For every summand $\mathcal{JL}(\pi)$ in the base change of \eqref{gen hom thy 1} to $\overline{\mathbb{Q}}_{\ell}$ with $\pi$ irreducible and supercuspidal, $\pi$ appears with multiplicity at least $n$ in \eqref{gen coh thy 1}. We will refer to this relationship between summands in \eqref{gen hom thy 1} and summands in \eqref{gen coh thy 1} as ``$\ell$-adic topological Jacquet-Langlands duality,'' or ``TJL duality'' for short. We will refer to \eqref{gen coh thy 1} as the ``TJL dual'' of $X$ (with respect to $\mathbb{G}$ and $\ell$).

The author of this paper feels that topological versions of Langlands or Jacquet-Langlands correspondences only have much value if they have some bearing on computation, e.g. if they can be used to get a better understanding of homotopy groups. We therefore owe the reader some explanation of how the constructions of this paper relate to honest calculations of $\pi_*$. Consider the case where the formal group $\mathbb{G}$ is the multiplicative formal group over $\overline{\mathbb{F}}_p$. Then $E(\mathbb{G})$ agrees with $p$-complete complex $K$-theory $\hat{KU}_p$ with a primitive $(p^i-1)$th root of unity adjoined for all $i\geq 1$, so that $\pi_*E(\mathbb{G})\cong W(\overline{\mathbb{F}}_p)[u^{\pm 1}]$. By the argument given in the proof of Theorem \ref{splitting of k-thy spectrum}, the algebraic $K$-theory spectrum $\mathcal{K}(E(\mathbb{G}))^{\wedge}_{\ell}$ is weakly equivalent to the $\ell$-adic completion of $\mathcal{K}(\overline{\mathbb{F}}_p) \vee \Sigma \mathcal{K}(\overline{\mathbb{F}}_p)$, which Suslin rigidity identifies with $\widehat{ku}_{\ell}\vee \Sigma \widehat{ku}_{\ell}$. 

Consequently, for $\mathbb{G}$ multiplicative over $\overline{\mathbb{F}}_p$, the generalized homology theory \eqref{gen hom thy 1} on the homological side of TJL duality is simply the theory that sends a spectrum $X$ to $(\widehat{ku}_{\ell})_*(X) \oplus \Sigma (\widehat{ku}_{\ell})_*(X)$, two copies of the $\ell$-complete connective complex $K$-homology of $X$.
Meanwhile, it follows from Example \ref{ht 1 example 1}, together with the calculation of $H^0$ of the height $1$ Lubin-Tate tower in section 3.4 of \cite{MR1044827} (or Theorem 4.4 of \cite{MR2441699} for a more general result), that the TJL dual $H\overline{\mathbb{Q}}_{\ell}\wedge_{\hat{S}_{\ell}}JL_{\ell}(\mathbb{G})$ of $\mathcal{K}(E(\mathbb{G}))\wedge X$ splits as a wedge of countably infinitely many copies of $\Aut(\mathbb{G})\times GL_1(\hat{\mathbb{Z}}_p)$-equivariant Eilenberg-Mac Lane spectra
\[ \coprod_{\omega} \left( H((\omega^{-1})\otimes \omega)\vee \Sigma H((\omega^{-1})\otimes \omega)\right),\]
where the coproduct is taken over all the continuous characters $\omega: \hat{\mathbb{Z}}_p^{\times} \rightarrow \overline{\mathbb{Q}}_{\ell}^{\times}$ where $\overline{\mathbb{Q}}_{\ell}^{\times}$ is equipped with the discrete topology. 

Since $\Ht_p(\mathbb{G})=1$, the action of $\Aut(\mathbb{G})\cong \hat{\mathbb{Z}}_p^{\times}$ on $\mathcal{K}(E(\mathbb{G}))^{\wedge}_{\ell}$ is trivial, by Remark \ref{ht 1 aut(G) triviality}. Consequently, for $n=1$, the cohomological side \eqref{gen coh thy 1} of TJL duality reduces to 
\begin{align}
\nonumber  & \pi_*F\left(\mathcal{K}(E(\mathbb{G}))^{\wedge}_{\ell}\wedge X,\left(H\overline{\mathbb{Q}}_{\ell}\wedge_{\hat{S}_{\ell}} JL_{\ell}(\mathbb{G})\right)^{h\Aut(\mathbb{G})}\right) \\
  \nonumber &\cong \left[\mathcal{K}(E(\mathbb{G}))^{\wedge}_{\ell}\wedge X,\left(H\overline{\mathbb{Q}}_{\ell}\wedge_{\hat{S}_{\ell}} JL_{\ell}(\mathbb{G})\right)^{h\Aut(\mathbb{G})}\right] \\
  \nonumber &\cong \left[\mathcal{K}(E(\mathbb{G}))^{\wedge}_{\ell}\wedge X,\left(\left( \coprod_{\omega} \left( H((\omega^{-1})\otimes \omega)\vee \Sigma H((\omega^{-1})\otimes \omega)\right)\right)^{\coprod\mathbb{N}}\right)^{h\Aut(\mathbb{G})}\right] \\
  \label{gen coh 409} &\cong \left[\left( \hat{ku}_{\ell} \vee \Sigma\hat{ku}_{\ell}\right)\wedge X ,\left( H\overline{\mathbb{Q}}_{\ell}\vee \Sigma H\overline{\mathbb{Q}}_{\ell}\right)^{\coprod\mathbb{N}}\right] .\end{align}
To be clear, here we are using the standard notation $[A,B]$ for homotopy classes of maps from a spectrum $A$ to a spectrum $B$.

This dramatic simplification \eqref{gen coh 409} of the TJL duals in the $n=1$ case is in fact a ``feature'' and not a ``bug.'' It is the key\footnote{See observation 1, at the end of this section, for further explanation of why this simplification is in fact absolutely necessary in order for the $L$-factors to be preserved.} to proving that, when $n=1$, 
the $\ell$-adic topological Jacquet-Langlands duality {\em preserves the $L$-factors},
in a sense that we will now explain. 

Suppose that $n=1$ and that $X$ is a finite CW-complex. The ``provisional $KU$-local zeta-function'' $\dot{\zeta}_{KU}(s,X)$ of $X$ is defined in the preprint \cite{salch2023kulocal} by taking the Euler product which defines the Hasse-Weil zeta-function of a variety, replacing the variety with the CW-complex $X$, replacing the Weil cohomology theory with complex topological $K$-theory, and replacing the Frobenius operators on Weil cohomology with the Adams operations on $K$-theory. In Theorem 2.4 of \cite{salch2023kulocal} it is shown that the resulting Euler product analytically continues to a meromorphic function $\dot{\zeta}_{KU}(s,X)$ on the complex plane. Write $P$ for the set of primes consisting of $2$ and all of the primes $p$ at which the cohomology of $X$ has nontrivial $p$-torsion. 
By Theorem 2.8 of \cite{salch2023kulocal}, if the cohomology $H^*(X;\mathbb{Z}[P^{-1}])$ is concentrated in even degrees, then the special values of $\dot{\zeta}_{KU}(s,X)$ in a left-hand half-plane recover the orders of the $KU$-local stable homotopy groups $\pi_*L_{KU}DX$ of the Spanier-Whitehead dual $DX$ of $X$, in the following sense. More precisely: write $b$ for the greatest integer $i$ such that $H^{2i}(X;\mathbb{Q})$ is nontrivial. Then 
\begin{align*}
 \left| \pi_{-2k-1}(L_{KU}DX)\right|
  &= \denom\left( \dot{\zeta}_{KU}(1-k,X)\right)
\end{align*}
for all $k \geq b+1$,
up to powers of irregular primes\footnote{At irregular primes, the order of $\pi_{-2k-1}(L_{KU}DX)$ is still describable in terms of special values of $\dot{\zeta}_{KU}(s,X)$, but a more complicated statement is required, given in terms of the ``isoweight factors'' of $\dot{\zeta}_{KU}(s,X)$. See Theorem 2.8 of \cite{salch2023kulocal} for the precise statement.} and primes in $P$. 

Now $\pi_*\left( F(\mathcal{K}(E(\mathbb{G}))^{\wedge}_{\ell}\wedge X,H\overline{\mathbb{Q}}_{\ell}\wedge_{\hat{S}_{\ell}} JL_{\ell}(\mathbb{G}))^{h\Aut(\mathbb{G})}\right)$ is a free graded module over the graded ring \begin{align}
 \nonumber &\pi_*\left( F(\mathcal{K}(E(\mathbb{G}))^{\wedge}_{\ell},H\overline{\mathbb{Q}}_{\ell}\wedge_{\hat{S}_{\ell}} JL_{\ell}(\mathbb{G}))^{h\Aut(\mathbb{G})}\right)\ \ \ \ \ \ \  \\
\label{gr gl ring 1}  &\ \ \ \cong \pi_*\left( \hat{ku}_{\ell},\left( H\overline{\mathbb{Q}}_{\ell}\vee \Sigma H\overline{\mathbb{Q}}_{\ell}\right)^{\coprod\mathbb{N}}\right) \oplus \Sigma^{-1}\pi_*\left( \hat{ku}_{\ell},\left( H\overline{\mathbb{Q}}_{\ell}\vee \Sigma H\overline{\mathbb{Q}}_{\ell}\right)^{\coprod\mathbb{N}}\right) ,\end{align}
which is trivial in degrees $>1$, by the connectivity of $ku$. 

The ring \eqref{gr gl ring 1} is a product of infinitely many copies of the same (trivial) representation of $GL_1(\hat{\mathbb{Z}}_p)$. The large number of ``redundant'' factors is the price we have paid for using the degenerating Lubin-Tate tower rather than the classical Lubin-Tate tower (see Example \ref{ht 1 example 1}), which was necessary in order to get the topological realizations $\mathcal{O}^{\topsym}_{\lvl p^1}, \mathcal{O}^{\topsym}_{\lvl p^2}, \dots$ in \cref{Topological realizability...}. 

We can pare down the redundant factors as follows.
Choose a homogeneous basis $\mathcal{B}$ for \[\pi_*\left( F(\mathcal{K}(E(\mathbb{G}))^{\wedge}_{\ell}\wedge X,H\overline{\mathbb{Q}}_{\ell}\wedge_{\hat{S}_{\ell}} JL_{\ell}(\mathbb{G}))^{h\Aut(\mathbb{G})}\right)\] over the graded ring \eqref{gr gl ring 1}. For each homogeneous generator $x\in\mathcal{B}$, the submodule of $\pi_*\left( F(\mathcal{K}(E(\mathbb{G}))^{\wedge}_{\ell}\wedge X,H\overline{\mathbb{Q}}_{\ell}\wedge_{\hat{S}_{\ell}} JL_{\ell}(\mathbb{G}))^{h\Aut(\mathbb{G})}\right)$ generated by $x$ is again a product of countably infinitely many copies of the same representation $\rho^{\mathbb{G}}_{x}$ of $GL_1(\hat{\mathbb{Z}}_p)$. Assume that $p\notin P$. Then, since $X$ was assumed to have rational cohomology concentrated in even degrees, the homogeneous generator $x$ is in degree $2w_x$ for some integer $w_x$. Let $\rho^{\mathbb{G}}_x(w_x)$ be the $w_x$th Tate twist of the $GL_1(\hat{\mathbb{Z}}_p)$-representation $\rho^{\mathbb{G}}_x$. Write $L(s,\rho)$ for the $L$-factor associated to a representation $\rho$ of $GL_1(\hat{\mathbb{Z}}_p)$ (cf. \cite{MR2808915}). By the {\em automorphic $L$-factor associated to $\mathbb{G}$ and $X$} we mean the product
\begin{equation}\label{automorphic l-factor}
 \prod_{x\in \mathcal{B}} L(s, \rho^{\mathbb{G}}_x(w_x)).
\end{equation}

Now, given the work already done in this paper as well as in \cite{salch2023kulocal}, it is easy to prove that TJL duality preserves the $L$-factors, as follows:
\begin{theorem}\label{main automorphic thm}
Let $X$ be a finite CW-complex. Suppose that $H^*(X;\mathbb{Q})$ is concentrated in even degrees. Let $p$ be a prime such that $H^*(X;\mathbb{Z})$ has no nontrivial $p$-torsion. Let $\mathbb{G}$ be a height $1$ formal group law over $\overline{\mathbb{F}}_p$. Then the automorphic $L$-factor \eqref{automorphic l-factor} associated to $\mathbb{G}$ and $X$ is equal to the $p$-local Euler factor in the provisional $KU$-local zeta-function $\dot{\zeta}_{KU}(s,X)$.
\end{theorem}
\begin{proof}
Since $\Aut(\mathbb{G})$ acts trivially on $\mathcal{K}(E(\mathbb{G}))^{\wedge}_{\ell}$ by Remark \ref{ht 1 aut(G) triviality}, the automorphic $L$-factor \eqref{automorphic l-factor} is simply the product
\begin{equation}
 \prod_{w\in\mathbb{Z}} (1-p^{w-s})^{-\beta_{2w}},
\end{equation}
where $\beta_{2w} = \dim_{\mathbb{Q}}H^{2w}(X;\mathbb{Q})$ is the $(2w)$th Betti number of $X$. The latter is the $p$-local Euler factor of $\dot{\zeta}_{KU}(s,X)$: see Definition 2.5 of \cite{salch2023kulocal}.
\end{proof}

The author would like to again stress that, if TJL duality actually {\em matters}, it matters because it can help to understand and calculate homotopy groups of finite spectra. Based on the Adams-Baird-Ravenel calculation of $\pi_*L_{KU}S^0$ in terms of special values of the Riemann zeta-function, and the generalizations in \cite{rmjpaper} and in \cite{salch2023kulocal}, the hope is that we ought to be able to use Langlands-like dualities to associate Galois representations or automorphic representations to a finite spectrum $X$ in such a way that special values of the $L$-functions of the representations recover the orders of the (Bousfield-localization) stable homotopy groups of $X$. 

Consider what we get from TJL duality: we do not get a full-fledged $L$-function from a finite spectrum $X$, but only an $L$-factor. The deep and mysterious behavior of special values of complex-analytic $L$-functions, e.g. the Lichtenbaum conjecture \cite{MR0406981} or Deligne's conjecture \cite{MR0546622} on critical values of $L$-functions, involve analytically continuing a well-understood function beyond its natural domain, and evaluating that analytic continuation {\em outside} the original domain. Trying to do this with only a single $L$-{\em factor} does not yield much of interest. One needs global data, $L$-factors at infinitely many primes, to get Lichtenbaum-conjecture-like phenomena for special values in the left half-plane.

Hence we are motivated to make a global statement. Let $\mathbb{G}$ denote a formal group law defined over $\mathbb{Z}$, or perhaps $\mathbb{Z}$ with finitely many primes inverted. For each prime $p$, write $\mathbb{G}\otimes_{\mathbb{Z}} \overline{\mathbb{F}}_p$ for $\mathbb{G}$ base-changed to $\overline{\mathbb{F}}_p$. For each prime $p$, suppose we have decided on a way to pare down the many redundant factors of the TJL dual $GL_n(\hat{\mathbb{Z}}_p)$-representation of $\left(\mathcal{K}(E(\mathbb{G}))^{\wedge}_{\ell}\right)_*(X)$; we have already explained how to do this in the height $1$ case. To that $\ell$-adic $GL_n(\hat{\mathbb{Z}}_p)$-representation, fix an associated automorphic $p$-local $L$-factor. By taking a product over all primes $p$---i.e., an Euler product of the $p$-local $L$-factors---we get a full-fledged $L$-{\em function} associated to $X$. Call this the {\em automorphic Euler product associated to $\mathbb{G}$ and $X$}.

If the automorphic Euler product associated to $\mathbb{G}$ and $X$ analytically continues to a meromorphic function on the complex plane, we will call that analytic continuation the {\em automorphic $L$-function associated to $\mathbb{G}$ and $X$}, and we write $L_{\mathbb{G}}(s,X)$ for it.

Now suppose that $\mathbb{G}$ is the formal group law of some complex oriented cohomology theory $E^*$. 
\begin{question}\label{main automorphic question}
For some reasonable class of finite CW-complexes $X$:
\begin{enumerate}
\item Does the automorphic Euler product associated to $\mathbb{G}$ and $X$ analytically continue? That is, does the automorphic $L$-function $L_{\mathbb{G}}(s,X)$ associated to $\mathbb{G}$ and $X$ exist?
\item Suppose that $L_{\mathbb{G}}(s,X)$ exists. Are the orders of the Bousfield $E$-local stable homotopy groups of $X$, $\pi_*(L_EX)$, describable in terms of special values of $L_{\mathbb{G}}(s,X)$?
\end{enumerate}
\end{question}
As a corollary of Theorem \ref{main automorphic thm}, in the height $1$ case, the Euler product of the automorphic $L$-factors associated to $\mathbb{G}$ and $X$ is simply the Euler product of the provisional $KU$-local zeta-function of $X$. Hence, as a corollary of the results of section 2 of \cite{salch2023kulocal}, we get a positive answer to both parts of Question \ref{main automorphic question} in the case where $E^*$ is complex $K$-theory:
\begin{theorem}\label{main automorphic cor}
Let $\mathbb{G}$ be the formal group law of complex $K$-theory, i.e., $\mathbb{G}$ is the multiplicative formal group law over $\mathbb{Z}$.
Let $X$ be a finite CW-complex with torsion-free cohomology, and whose cohomology is concentrated in even degrees\footnote{A few familiar examples: $X$ could be taken to be a sphere of even dimension, or a complex projective space of any dimension.}. Write $b$ for the greatest integer $i$ such that $H^{2i}(X;\mathbb{Q})$ is nontrivial.
Then the automorphic $L$-function $L_{\mathbb{G}}(s,X)$ associated to $\mathbb{G}$ and $X$ exists. Furthermore, we have an equality
\begin{align*}
 \left| \pi_{-2k-1}(L_{KU}DX)\right|
  &= \denom L_{\mathbb{G}}(1-k,X)
\end{align*}
up to powers of $2$ and powers of irregular primes, for all integers $k \geq b+1$.
\end{theorem}

The author hopes that a positive answer to Question \ref{main automorphic question} can be obtained for some formal groups $\mathbb{G}$ over $\mathbb{Z}$, or perhaps over $\mathbb{Z}[P^{-1}]$ for some finite set of primes $P$, such that $\mathbb{G}$ has height $h>1$ at some primes. The next case of Question \ref{main automorphic question} to consider is that of a Landweber exact elliptic curve $C$ over $\mathbb{Z}[\Delta_C^{-1}]$, since any such elliptic curve has an associated elliptic cohomology theory whose underlying formal group law agrees with that of $C$. Hence the formal group will have $p$-height $2$ at the supersingular primes $p$ of $C$. The author is optimistic that this can be made to work for an elliptic curve $C$ with complex multiplication by a quadratic imaginary number field $F$ of class number $1$, based on some computer calculations of the critical values of twists of the Hecke $L$-function of the Grossencharacter of $C$. The author also has some preliminary calculations, not using elliptic curves or Jacquet-Langlands correspondences, at heights $2$ and $4$ using the ``height-shifting'' method of \cite{height4}, \cite{formalmodules4}, and \cite{cmah7} to compare abelian $L$-values over quadratic and quartic fields to orders of $K(2)$-local and $K(4)$-local stable homotopy groups of finite spectra. These height $>1$ calculations go far beyond the scope of the present paper, though, and will have to be described elsewhere. 

We finish with two observations about \eqref{gen coh 409}, each of which leads to natural questions that go well beyond the scope of this paper:
\begin{enumerate}
\item 
The action of $GL_1(\hat{\mathbb{Z}}_p)$ on \eqref{gen coh 409} is trivial. This made much of the work in this section of the paper very easy in the case $n=1$: there are never any nontrivial representations of $\Aut(\mathbb{G})$ on the homological side \eqref{gen hom thy 1} of the duality, or nontrivial representations of $GL_1(\hat{\mathbb{Z}}_p)$ on the cohomological side \eqref{gen coh thy 1} of the duality. 

To be clear, triviality of the $GL_1(\hat{\mathbb{Z}}_p)$-action on \eqref{gen coh 409} is indeed absolutely necessary for the height $1$ automorphic $L$-functions of finite CW-complexes to correctly describe the $KU$-local stable homotopy groups of those complexes. This is because the TJL dual \eqref{gen coh thy 1} of $(\mathcal{K}(E(\mathbb{G})^{\wedge}_{\ell})_*(X)$ is determined by the action of the group of Adams operations $\hat{\mathbb{Z}}_p^{\times}\cong \Aut(\mathbb{G})$ on a {\em rational} spectrum associated to $X$. But the action of Adams operations is well-known to be rationally diagonalizable, i.e., the associated automorphic $L$-factors {\em must} be products of Tate twists of $L$-factors of trivial $GL_1(\hat{\mathbb{Z}}_p)$-representations.

To see a duality in which more interesting representations of $GL_n(\hat{\mathbb{Z}}_p)$ arise, one can of course look to heights $n>1$, where the $\Aut(\mathbb{G})$ action on \eqref{gen hom thy 1} will generally be nontrivial, as explained at the end of Remark \ref{ht 1 aut(G) triviality}. But it is still not clear whether {\em supercuspidal} representations of $GL_n(\hat{\mathbb{Z}}_p)$ ever occur in \eqref{gen coh thy 1}. This ought to be decidable by explicit calculation in the height $n=2$ case, but not by any calculation which is short enough to reasonably include in this paper. The concern is simply that relatively few representations of $\Aut(\mathbb{G})$ can be realized in the form $\left(\mathcal{K}(E)^{\wedge}_{\ell}\right)_*(X)$ for $E$ an $E_{\infty}$-ring spectrum with an action of $\Aut(\mathbb{G})$. 

The right thing to do is probably to just {\em get rid of the algebraic $K$-theory} on the homological side \eqref{gen hom thy 1} of TJL duality. Rather than considering the action of $\Aut(\mathbb{G})$ on $\left(\mathcal{K}(E(\mathbb{G}))^{\wedge}_{\ell}\right)_*(X)$, we should simply consider the action of $\Aut(\mathbb{G})$ on $E(\mathbb{G})_*(X)$, since the cohomology of that action is the input for the descent spectral sequence converging to a small extension of $\pi_*(L_{K(n)}X)$; see \cite{MR2030586} for this well-known spectral sequence, which is of great importance in calculations. 

However, algebraic $K$-theory appeared in \eqref{gen hom thy 1} because we needed to get an {\em $\ell$-adic representation} of $\Aut(\mathbb{G})$ for $\ell\neq p$, rather than a {\em $p$-adic representation} of $\Aut(\mathbb{G})$, in order to use the existing results on the $\ell$-adic Jacquet-Langlands correspondence. It is entirely reasonable to ask for a $p$-adic (rather than $\ell$-adic) topological Jacquet-Langlands duality, which relates certain $p$-adic representations of $\Aut(\mathbb{G})$ occurring in $E(\mathbb{G})_*(X)$ to $p$-adic representations of $GL_n(\hat{\mathbb{Z}}_p)$ occurring in the homotopy groups of an appropriate dual of $E(\mathbb{G})_*(X)$ along the lines of \eqref{gen hom thy 1}. This will be more difficult than the $\ell$-adic case, probably involving recent progress toward a $p$-adic Jacquet-Langlands correspondence (e.g. as in \cite{MR3861564} and \cite{DospinescuPaskunasSchraen+2023+57+114}), and it is a subject for another paper.
\item The cohomological side of TJL duality, \eqref{gen coh 409}, is a {\em rational} invariant of the $\ell$-adic homotopy type of $X$. For example, given finite CW-complexes $X_1$ and $X_2$, \eqref{gen coh 409} can distinguish $X_1$ from $X_2$ if and only if Betti numbers can distinguish $X_1$ from $X_2$.


This insensitivity of \eqref{gen coh 409} is an artifact of the fact that we have taken a smash product with $H\overline{\mathbb{Q}}_{\ell}$ in \eqref{gen coh thy 1}. That smash product was required in order to get a $\overline{\mathbb{Q}}_{\ell}$-linear representation of $GL_n(\hat{\mathbb{Z}}_p)$, so we could use known results on $\ell$-adic Jacquet-Langlands correspondence. But of course we could have left out that smash product, and simply considered
\begin{align*}\label{gen coh thy 2} \pi_*\left( F(\mathcal{K}(E)^{\wedge}_{\ell}\wedge X,JL_{\ell}(\mathbb{G}))^{h\Aut(\mathbb{G})}\right),\end{align*}
a $\hat{\mathbb{Z}}_{\ell}$-linear representation of $GL_n(\hat{\mathbb{Z}}_p)$. This is probably a much richer invariant of $X$ than the cohomology theory \eqref{gen coh 409}, and worth some further investigation---but not in this paper, which is already too long.
\end{enumerate}

\appendix
\section{Topological realizability and the LT tower (by A. Salch and M. Strauch)}
\subsection{Review of the Lubin-Tate tower.}
\label{Review of the Lubin-Tate tower.}

Given a commutative ring $A$ and a commutative $A$-algebra $R$, recall that a {\em formal $A$-module over $R$} is a formal group law $F(X,Y)\in R[[X,Y]]$ over $R$ together with a ring map $\rho: A\rightarrow \End(F)$ such that $\rho(a)(X) \equiv aX$ modulo $X^2$ for all $a\in A$. For any prime $p$, formal $\hat{\mathbb{Z}}_p$-modules are equivalent to formal group laws over commutative $\hat{\mathbb{Z}}_p$-algebras. Consequently the case $A = \mathbb{\hat{Z}}_p$ is the base case, and the most relevant case for the topological applications which begin in \cref{Topological realizability...}. Nevertheless for the algebraic applications of the Lubin-Tate tower, cases of $A$ other than $A = \hat{\mathbb{Z}}_p$ are also important: for example, to realize the $\ell$-adic Jacquet-Langlands correspondence for a $p$-adic number field $F$, one uses the Lubin-Tate tower of a formal $\mathcal{O}_F$-module. We will present the ideas for general $A$ (satisfying the constraints of Conventions \ref{main conventions}, i.e., $A$ is the ring of integers in a finite field extension of $\mathbb{Q}_p$) until it becomes necessary to restrict to the case $A = \hat{\mathbb{Z}}_p$. 

One can define the $A$-height of a formal $A$-module in an intrinsic way (see \cite{MR2987372} for a textbook account), but the most relevant facts about $A$-height are as follows:
\begin{itemize}
\item The $A$-height of a formal $A$-module is either $\infty$ or a nonnegative integer. It is zero if and only if $\varpi$ is a unit in $R$.
\item Suppose $\mathbb{F},\mathbb{G}$ are formal $A$-modules defined over a separably closed field extension of the residue field of $A$. Then $\mathbb{F}$ and $\mathbb{G}$ are isomorphic if and only if their $A$-heights are equal.
\item Write $\Frac(A)$ for the fraction field of $A$. Then the $A$-height of a formal $A$-module $\mathbb{G}$ is equal to the $p$-height of the underlying formal group law of $\mathbb{G}$, divided by the degree $[\Frac(A): \mathbb{Q}_p]$ of the field extension $\Frac(A)/\mathbb{Q}_p$.
\end{itemize}

Given a formal $A$-module $F$ over $R$ and an element $a\in A$, the power series $\rho(a)(X)\in R[[X]]$ is often called the {\em $a$-series} of $F$, and written $[a]_F(X)$. The case in which $a=\varpi$, a uniformizer for $A$, plays a special role in Definition \ref{def of drinfeld level struct}. This definition originated with \cite{MR0384707}:
\begin{definition}\label{def of drinfeld level struct} 
Let $m$ be a nonnegative integer. 
Let $\mathbb{G}$ be a formal $A$-module, of $A$-height $n<\infty$, over a complete local commutative $A$-algebra $B$ with maximal ideal $\mathfrak{m}_B$. 
By a {\em Drinfeld level $\varpi^m$-structure on $\mathbb{G}$} we mean a 
homomorphism 
$\phi: \left( \varpi^{-m}A/A\right)^n \rightarrow \mathbb{G}(\mathfrak{m}_B)$
of $A$-modules such that
the $\varpi$-series $[\varpi]_{\mathbb{G}}(X)\in B[[X]]$ of $\mathbb{G}$
is divisible by the product
\[ \prod_{a\in \left( \varpi^{-1}A/A\right)^n} \left( X - \phi(a)\right) .\]
\end{definition}
Suppose now that $\mathbb{G}$ is a formal $A$-module over a separable field extension $k$ of the residue field of $A$. Let $R$ be a complete local Noetherian commutative ring with residue field $k$. A {\em deformation of $\mathbb{G}$ to $R$} is a pair $(\mathbb{F},\psi)$, where $\mathbb{F}$ is a formal $A$-module over $R$, and $\psi$ is an isomorphism of formal $A$-modules $\mathbb{F}\otimes_R k \stackrel{\cong}{\longrightarrow} \mathbb{G}$. There exists an affine formal scheme $\Def(\mathbb{G})$ such that the set of {\em local} ring homomorphisms $\Gamma(\Def(\mathbb{G})) \rightarrow R$ is in natural bijection with the set of isomorphism classes of deformations of $\mathbb{G}$ to $R$. The formal affine scheme $\Def(\mathbb{G})$ is non-canonically isomorphic to $\Spf W(k)[[u_1, \dots ,u_{n-1}]]$, where $n$ is the $A$-height of $\mathbb{G}$.

More generally, a {\em deformation of $\mathbb{G}$ to $R$ with Drinfeld level $\varpi^m$-structure} is a triple $(\mathbb{F},\psi,\phi)$, where $(\mathbb{F},\psi)$ is a deformation of $\mathbb{G}$ to $R$ and where $\phi$ is a Drinfeld level $\varpi^m$-structure on $\mathbb{F}$. There again exists an affine formal scheme $\Def(\mathbb{G})_{\lvl\varpi^{m}}$ such that the set of local ring homomorphisms $\Gamma(\Def(\mathbb{G}))_{\lvl\varpi^m} \rightarrow R$ is in natural bijection with the set of isomorphism classes of deformations of $\mathbb{G}$ to $R$ with Drinfeld level $\varpi^m$-structure. This is a consequence of Theorem 4.3 of \cite{MR0384707}.

Here are some well-known properties of $\Def(\mathbb{G})_{\lvl\varpi^m}$ established in \cite{MR0384707}:
\begin{itemize}
\item In the case $m=0$, we have $\Def(\mathbb{G})_{\lvl\varpi^0} = \Def(\mathbb{G})$.
\item Composing the natural inclusion $\left( \varpi^{-(m-1)}A/A\right)^n \hookrightarrow \left( \varpi^{-m}A/A\right)^n$ with a Drinfeld level $\varpi^m$-structure $\phi: \left( \varpi^{-m}A/A\right)^n \rightarrow \tilde{\mathbb{G}}(\mathfrak{m}_B)$ on a deformation $\tilde{\mathbb{G}}$ of $\mathbb{G}$, we get a Drinfeld level $\varpi^{m-1}$-structure on $\tilde{\mathbb{G}}$. By naturality and the Yoneda lemma, we get a map of formal schemes 
\begin{align}\label{map 2320}\Def(\mathbb{G})_{\lvl\varpi^m}&\rightarrow\Def(\mathbb{G})_{\lvl\varpi^{m-1}}.\end{align} This map is flat, but not \'{e}tale. It is \'{e}tale after inverting $\varpi$, and the resulting ring map $\varpi^{-1} \Gamma(\Def(\mathbb{G})_{\lvl\varpi^{0}})\rightarrow \varpi^{-1} \Gamma(\Def(\mathbb{G})_{\lvl\varpi^{m}})$ is a Galois cover with Galois group $GL_h(A/\varpi^m)$. 
\item The resulting sequence of formal schemes
\begin{equation}\label{lt tower} \dots 
 \rightarrow \Def(\mathbb{G})_{\lvl\varpi^2} 
 \rightarrow \Def(\mathbb{G})_{\lvl\varpi^1} 
 \rightarrow \Def(\mathbb{G})_{\lvl\varpi^0} \end{equation}
is called the {\em Lubin-Tate tower} of $\mathbb{G}$.
\end{itemize}

Now suppose that $A = \hat{\mathbb{Z}}_p$.
From Goerss--Hopkins \cite{MR2125040} 
it is known that there exists a presheaf $\mathcal{O}^{\topsym}$ of $E_{\infty}$-ring spectra on the small \'{e}tale site $\Def(\mathbb{G})_{\et}$ satisfying the following properties:
\begin{itemize}
\item For each \'{e}tale open $U$ of $\Def(\mathbb{G})$, we have an isomorphism of graded rings $\pi_*\mathcal{O}^{\topsym}(U) \cong \Gamma(U)[u^{\pm 1}]$, with $u$ in degree $2$ and $\Gamma(U)$ concentrated in degree $0$. In particular, the presheaf of commutative rings $\pi_0\circ\mathcal{O}^{\topsym}$ coincides with the structure sheaf on $\Def(\mathbb{G})_{\et}$.
\item For each \'{e}tale open $U = \Def(\mathbb{G}) \times_{\Spf W(k)} \Spf W(F)$ with $F$ a separable field extension of $k$, the ring spectrum $\mathcal{O}^{\topsym}(U)$ is complex oriented, and the formal group law on the ring $\pi_0\mathcal{O}^{\topsym}(U)$ arising from the complex orientation of $\mathcal{O}^{\topsym}(U)$ is equal to the universal formal group law on $\Gamma(U)$. See \cite{MR1324104} or \cite{coctalos} for good expository accounts of complex oriented cohomology theories and their associated formal group laws.
\end{itemize}

Even before the publication of \cite{MR2125040}, the height $1$ case was in some ways well-known: if $\mathbb{G}$ is the multiplicative formal group law, then $\Gamma(\mathcal{O}^{\topsym})$ is the $p$-complete periodic complex $K$-theory spectrum, $\hat{KU}_p$, which was known to be an $E_{\infty}$-ring spectrum much earlier (e.g. in \cite{MR0494077}). The paper \cite{MR1718085} established that there cannot exist an $E_{\infty}$ $\hat{KU}_p$-algebra $B$ such that $\pi_0(B)$ is isomorphic to $\Gamma(\Def(\mathbb{G})_{\lvl p^m})$ as a $\Gamma(\Def(\mathbb{G})_{\lvl p^0})$-algebra, except in the base case $m=0$. That is, even in the height $1$ case, the sheaf of $E_{\infty}$-rings $\mathcal{O}^{\topsym}$ on $\Def(\mathbb{G})_{\et}$ cannot lift to a sheaf of $E_{\infty}$-rings on any of the higher stages in the Lubin-Tate tower \eqref{lt tower}. In that sense, the Lubin-Tate tower is not ``topologically realizable.''
The same phenomenon for heights $>1$ is demonstrated in \cite{MR4099803}.

\subsection{The degenerating Lubin-Tate tower.}
\label{The degenerating Lubin-Tate tower.}

We address the problem of topological nonrealizability of the Lubin-Tate tower by a simple method: we embed the Lubin-Tate tower into a larger tower, the {\em degenerating Lubin-Tate tower,} which {\em is} topologically realizable.

There are two fundamental properties of the degenerating Lubin-Tate tower which make it useful:
\begin{enumerate}
\item Its topological realizability, as we prove below, in Theorem \ref{realizability thm 1}.
\item The tower admits a stratification, which we construct in this section, and which we call the {\em stratification by degeneration type.} In Proposition \ref{open and closed subset} we show that one of the strata splits off, on the Raynaud generic fibres, as a coproduct summand, and this stratum is precisely a copy of the classical Lubin-Tate tower. We use this fact to get a handle on the \'{e}tale cohomology of $\Def(\mathbb{G})^{\degen}_{\lvl\varpi^m}$, in Theorem \ref{coh splitting thm} and Example \ref{ht 1 example 2}.
\end{enumerate}

The idea of the degenerating Lubin-Tate tower is straightforward. By definition, a Drinfeld level $\varpi^m$-structure on a formal $A$-module $\mathbb{G}$ over $B$ is an $A$-module homomorphism $\phi: \left( \varpi^{-m}A/A\right)^n \rightarrow \mathbb{G}(\mathfrak{m}_B)$ such that 
\begin{equation}
\label{drinfeld level condition 1}
\prod_{a\in \left( \varpi^{-1}A/A\right)^n} \left( X - \phi(a)\right)\mbox{\ divides\ } [\varpi]_{\mathbb{G}}(X).
\end{equation}
Condition \eqref{drinfeld level condition 1} is equivalent to:
\begin{equation}
\label{drinfeld level condition 2}
\prod_{a\in \left( \varpi^{-m}A/A\right)^n} \left( X - \phi(a)\right)\mbox{\ divides\ } [\varpi^m]_{\mathbb{G}}(X).
\end{equation}
We consider a relaxation of \eqref{drinfeld level condition 2} in which the Drinfeld divisibility condition \eqref{drinfeld level condition 2} is allowed to fail, but only within some fixed subset $S$ of $\left( \varpi^{-m}A/A\right)^n$, as follows:
\begin{definition}\label{def of degenerating level struct} 
Let $m$ be a nonnegative integer. 
Let $\mathbb{G}$ be a formal $A$-module of $A$-height $n$ over a complete local commutative $A$-algebra $B$ with maximal ideal $\mathfrak{m}_B$. 
Let $S$ be a subset of $\left(\varpi^{-m}A/A\right)^n$.
By a {\em level $\varpi^m$ structure on $\mathbb{G}$ of degeneration type $S$,} we mean a homomorphism 
$\phi: \left( \varpi^{-m}A/A\right)^n \rightarrow \mathbb{G}(\mathfrak{m}_B)$
of $A$-modules 
such that
\begin{equation}
\label{drinfeld level condition 3}
\prod_{a\notin S} \left( X - \phi(a)\right)\mbox{\ divides\ } [\varpi^m]_{\mathbb{G}}(X).
\end{equation}
We call such level structures---i.e., level $\varpi^m$ structures of degeneration type $S$, for various $S$---{\em degenerating level $\varpi^m$ structures.}
\end{definition}

Here are some observations about degenerating level structures.
\begin{enumerate}
\item In the case that $S$ is the empty set, a degenerating level $\varpi^m$ structure of type $S$ is precisely a Drinfeld level $\varpi^m$ structure. So Drinfeld level structures are a special case of degenerating level structures.
\item In the opposite extreme case, when $S$ is the entirety of $\left( \varpi^{-m}A/A\right)^n$, a degenerating level $\varpi^m$ structure of type $S$ is simply an $A$-module homomorphism $\phi: \left( \varpi^{-m}A/A\right)^n \rightarrow \mathbb{G}(\mathfrak{m}_B)$. 
\item
Given a formal $A$-module $\mathbb{G}$ over a field extension $k$ of the residue field of $A$, it is not difficult to see that there exists a formal affine scheme $\Def(\mathbb{G})_{\lvl\varpi^m}^{\degen S}$ which parameterizes isomorphism classes of deformations of $\mathbb{G}$ equipped with level $\varpi^m$-structure of degeneration type $S$.
\item 
If $S$ is a subset of $T\subseteq \left( \varpi^{-m}A/A\right)^n$, then $\Def(\mathbb{G})^{\degen S}_{\lvl\varpi^m}$ is a formal subscheme of $\Def(\mathbb{G})^{\degen T}_{\lvl\varpi^m}$. 
Consequently, we have a stratification of $\Def(\mathbb{G})^{\degen}_{\lvl\varpi^m}$ by the formal subschemes $\Def(\mathbb{G})^{\degen S}_{\lvl\varpi^m}$, ranging over various subsets $S$ of $\left( \varpi^{-m}A/A\right)^n$. We refer to this stratification as the {\em stratification by degeneration type.} This stratification is only a minor variation of the ``typed subgroups'' of Strickland \cite{MR1473889}.

Below, in Proposition \ref{open and closed subset}, we prove that in the extreme case that $S = \emptyset$ and $T = \left( \varpi^{-m}A/A\right)^n$, this inclusion is both closed and open on the generic fibre\footnote{We suspect that our argument for Proposition \ref{open and closed subset} generalizes to yield the same conclusion for general $S$ and $T$, but there seems to be no immediate need for this extra generality.}. That is, the generic fibre of $\Def(\mathbb{G})_{\lvl\varpi^m}$ is a closed and open rigid analytic subspace of the generic fibre of $\Def(\mathbb{G})^{\degen \left( \varpi^{-m}A/A\right)^n}_{\lvl\varpi^m}$, the deformation space for {\em all} degenerating level $\varpi^m$ structures. We will write $\Def(\mathbb{G})^{\degen}_{\lvl\varpi^m}$ as an abbreviation for $\Def(\mathbb{G})^{\degen \left( \varpi^{-m}A/A\right)^n}_{\lvl\varpi^m}$.
\item 
We have a natural map $\Def(\mathbb{G})^{\degen S}_{\lvl\varpi^m} \rightarrow \Def(\mathbb{G})^{\degen \varpi S}_{\lvl\varpi^{m-1}}$, given by the same construction as the map \eqref{map 2320}, above. In the case $S=\left( \varpi^{-m}A/A\right)^n$, We call the resulting sequence of formal schemes
\[ \dots\rightarrow\Def(\mathbb{G})^{\degen}_{\lvl\varpi^{2}}\rightarrow\Def(\mathbb{G})^{\degen}_{\lvl\varpi^{1}}\rightarrow\Def(\mathbb{G})^{\degen}_{\lvl\varpi^{0}}\]
the {\em degenerating Lubin-Tate tower.}
\item The formal scheme $\Def(\mathbb{G})^{\degen}_{\lvl\varpi^{m}}$ is isomorphic to $\Spf$ of the $\hat{Z}_p$-algebra
\begin{equation*}
 \Gamma(\Def(\mathbb{G}))[[X]]/[\varpi^m]_{\widetilde{\mathbb{G}}}(X) \otimes_{\Gamma(\Def(\mathbb{G}))} \dots \otimes_{\Gamma(\Def(\mathbb{G}))}\Gamma(\Def(\mathbb{G}))[[X]]/[\varpi^m]_{\widetilde{\mathbb{G}}}(X) ,\end{equation*}
an $n$-fold tensor power of copies of $\Gamma(\Def(\mathbb{G}))[[X]]/[\varpi^m]_{\widetilde{\mathbb{G}}}(X)$, 
where $\widetilde{\mathbb{G}}$ is the universal deformation of $\mathbb{G}$ defined over $\Gamma(\Def(\mathbb{G}))$. 
\end{enumerate}
\begin{example}\label{ht 1 example 1}
Suppose that $\mathbb{G}$ has $A$-height $1$. Then, for each positive integer $m$, we have a unique maximal-length sequence of $A$-submodules of $\left( \varpi^{-m}A/A\right)^n$:
\[ 
\varpi^{-m}A/A \supseteq \varpi^{-m+1}A/A \supseteq \dots
 \supseteq \varpi^{-2}A/A \supseteq \varpi^{-1}A/A \supseteq 0.\] 
Write $S_j$ for the complement of $\varpi^{-j}A/A$ in $\varpi^{-j-1}A/A$. Then $\Def(\mathbb{G})^{\degen S_{j}}_{\lvl\varpi^{m}}$ classifies deformations $\widetilde{\mathbb{G}}$ of $\mathbb{G}$ equipped with an $A$-module map $\eta: \varpi^{-m}A/A \rightarrow\widetilde{\mathbb{G}}(\mathfrak{m}_{\widetilde{\mathbb{G}}})$ whose kernel is contained in the $A$-submodule $\varpi^{-j}A/A\subseteq \varpi^{-m}A/A$. 

Consequently, the complement of $\Def(\mathbb{G})^{\degen S_{j+1}}_{\lvl\varpi^{m}}$ in $\Def(\mathbb{G})^{\degen S_j}_{\lvl\varpi^m}$ is isomorphic to $\Def(\mathbb{G})_{\lvl\varpi^j}$. On the generic fibres, we have a splitting of the deformation space $\Def(\mathbb{G})^{\degen}_{\lvl\varpi^{m}}$ with degenerating level $\varpi^m$-structures:
\begin{align}
\label{splitting 0120} \Def(\mathbb{G})^{\degen}_{\lvl\varpi^{m}} 
  &\cong \Def(\mathbb{G})_{\lvl\varpi^m}\\
\nonumber  &\ \ \ \coprod \Def(\mathbb{G})_{\lvl\varpi^{m-1}} \\
\nonumber  &\ \ \ \coprod \dots \\
\nonumber  &\ \ \ \coprod \Def(\mathbb{G})_{\lvl\varpi^{1}} \\
\nonumber  &\ \ \ \coprod \Def(\mathbb{G})_{\lvl\varpi^{0}}. 
\end{align}
The map $\Def(\mathbb{G})^{\degen}_{\lvl\varpi^{m}} \rightarrow \Def(\mathbb{G})^{\degen}_{\lvl\varpi^{m-1}}$ in the degenerating Lubin-Tate tower sends the component $\Def(\mathbb{G})_{\lvl\varpi^{j}}$ of the generic fibre to the component $\Def(\mathbb{G})_{\lvl\varpi^{j-1}}$ of the generic fibre, by the same map as in the classical Lubin-Tate tower.
Consequently, on the generic fibres, the height $1$ degenerating Lubin-Tate tower is the disjoint union of countably infinitely many copies of the classical Lubin-Tate tower, one for each nonnegative integer $i$, with the $i$th tower ``delayed'' by $i$ terms:
\begin{equation}\label{map of towers 1} 
\xymatrix{
\vdots \ar[d] & \vdots \ar[d] \\
 \Def(\mathbb{G})_{\lvl\varpi^{2}} \ar[r]\ar[d] &
  \Def(\mathbb{G})_{\lvl\varpi^{2}}\coprod \Def(\mathbb{G})_{\lvl\varpi^{1}}\coprod \Def(\mathbb{G})_{\lvl\varpi^{0}} \ar[d]\\
 \Def(\mathbb{G})_{\lvl\varpi^{1}} \ar[r]\ar[d] &
  \Def(\mathbb{G})_{\lvl\varpi^{1}}\coprod \Def(\mathbb{G})_{\lvl\varpi^{0}} \ar[d]\\
 \Def(\mathbb{G})_{\lvl\varpi^{0}} \ar[r] &
  \Def(\mathbb{G})_{\lvl\varpi^{0}} .}
\end{equation}
We caution the reader that the disjoint union decomposition of the degenerating Lubin-Tate tower described in \eqref{splitting 0120} and pictured in the right-hand column of \eqref{map of towers 1} only occurs after inverting $\varpi$, i.e., it is only a splitting of the generic fibres. 

To build intuition, it is helpful to see the splitting \eqref{splitting 0120} more explicitly in a special case. Let $A = \hat{\mathbb{Z}}_p$, and let $\mathbb{G}$ be the multiplicative formal group, so that the $p$-series $[p]_{\mathbb{G}}(X)$ is simply the polynomial $(X+1)^p-1$. We have $\Spf\Def(\mathbb{G})_{\lvl\varpi^{0}} \cong \hat{\mathbb{Z}}_p$, and furthermore:
\begin{align*}
 \Spf\Def(\mathbb{G})^{\degen}_{\lvl p^{m}} 
  &\cong \hat{\mathbb{Z}}_p[[X]]/[p^m]_{\mathbb{G}}(X) \\
  &\cong \hat{\mathbb{Z}}_p[[Y]]/[p^m]_{\mathbb{G}}(Y-1) \\
  &= \hat{\mathbb{Z}}_p[[Y]]/(Y^{p^m}-1) \\
  &= \hat{\mathbb{Z}}_p[[Y]]/\prod_{j=0}^m \phi_j(Y),
\end{align*}
where $\phi_0(Y) = Y$ and where, for $j>0$, the polynomial $\phi_j(Y)$ is the minimal polynomial of a primitive $p^j$th root of unity $\zeta_{p^j}$ over $\hat{\mathbb{Z}}_p[\zeta_{p^{j-1}}]$. We have $\Spf \hat{\mathbb{Z}}_p[[Y]]/\phi_j(Y) \cong \coprod \Def(\mathbb{G})_{\lvl\varpi^{j}}$ {\em after inverting $p$}. Consequently the generic fibre of $\Spf \hat{\mathbb{Z}}_p[[Y]]/\prod_{j=0}^m \phi_j(Y)$ is then the union of the generic fibres of $\Spf \hat{\mathbb{Z}}_p[[Y]]/\phi_j(Y)$ for $j=0, \dots ,m$.
\end{example}
In Example \ref{ht 1 example 1}, for each $A$-submodule $V$ of $\varpi^{-m}A/A$, it was convenient to be able to refer to some subset $S_j$ of $\varpi^{-m}A/A$ such that $\Def(\mathbb{G})^{\degen S_{j}}_{\lvl\varpi^{m}}$ would classify deformations with a level structure which would satisfy the divisibility condition on the elements of $V$. More generally, for arbitrary (finite) heights: for each $A$-submodule $V$ of $(\varpi^{-m}A/A)^n$, we will write $\cancel{V}$ for the complement of $V$ in $(\varpi^{-m}A/A)^n$. Consequently $\Def(\mathbb{G})^{\degen \cancel{V}}_{\lvl\varpi^{m}}$ classifies deformations with a level structure which satisfies the Drinfeld divisibility condition on the elements of $V$.

\begin{prop}\label{open and closed subset}
Let $m$ be a positive integer. 
Let $\mathbb{G}$ be a one-dimensional formal $A$-module, of finite $A$-height $n$, over some field extension $k$ of the residue field $A/\varpi$ of $A$. 
Then $\Def(\mathbb{G})_{\lvl \varpi^m}$ is closed and open in $\Def(\mathbb{G})_{\lvl \varpi^m}^{\degen}$.
\end{prop}
\begin{proof}
For any $A$-submodule $S$ of $\left( \varpi^{-m}A/A\right)^n$,
consider $\Def(\mathbb{G})_{\lvl \varpi^m}^{\degen S}$ as a formal subscheme of $\Def(\mathbb{G})_{\lvl \varpi^m}^{\degen}$. The former is characterized by its classifying those triples $(\widetilde{\mathbb{G}},\phi,\psi)$ whose degenerating level structure $\phi: \left( \varpi^{-m}A/A\right)^n \rightarrow \widetilde{\mathbb{G}}(\mathfrak{m}_{\widetilde{\mathbb{G}}})$ has the property that the image of $\phi\mid_{\left( \varpi^{-m}A/A\right)^n\backslash S}$ spans an $A/\varpi$-submodule of $\widetilde{\mathbb{G}}(\mathfrak{m}_{\widetilde{\mathbb{G}}})/\varpi$ which is as large as possible, i.e., isomorphic to a maximal $A$-submodule of $\left( \varpi^{-m}A/A\right)^n$ not containing $S$. This is an open condition. Hence $\Def(\mathbb{G})_{\lvl \varpi^m}^{\degen S}$ is open in $\Def(\mathbb{G})_{\lvl \varpi^m}^{\degen}$. 

The formal scheme $\Def(\mathbb{G})_{\lvl \varpi^m}^{\degen}$ is isomorphic to 
\begin{eqnarray*} \Gamma(\Def(\mathbb{G}))[[\theta_1]]/[\varpi^m]_{\widetilde{\mathbb{G}}}(\theta_1) \otimes_{\Gamma(\Def(\mathbb{G}))}\Gamma(\Def(\mathbb{G}))[[\theta_2]]/[\varpi^m]_{\widetilde{\mathbb{G}}}(\theta_2)  \\
\ \ \ \ \ \otimes_{\Gamma(\Def(\mathbb{G}))} \dots \otimes_{\Gamma(\Def(\mathbb{G}))} \Gamma(\Def(\mathbb{G}))[[\theta_n]]/[\varpi^m]_{\widetilde{\mathbb{G}}}(\theta_n).\end{eqnarray*}
In the case $S = \emptyset$, in the proof of the lemma preceding Proposition 4.4 of \cite{MR0384707}, Drinfeld gives a presentation for $\Def(\mathbb{G})_{\lvl \varpi^m}^{\degen S} = \Def(\mathbb{G})_{\lvl \varpi^m}$ by calculating the power series $\Theta$ such that the joint vanishing of $\Theta(\theta_1), \dots ,\Theta(\theta_n)$ characterizes the subset $\Def(\mathbb{G})_{\lvl \varpi^m}$ of $\Def(\mathbb{G})_{\lvl \varpi^m}^{\degen}$. Vanishing of $\Theta(\theta_1), \dots ,\Theta(\theta_n)$ is a closed condition, so $\Def(\mathbb{G})_{\lvl \varpi^m}^{\degen S}$ is a closed subset of $\Def(\mathbb{G})_{\lvl \varpi^m}^{\degen}$. 
\end{proof}

\subsection{Topological realizability of the degenerating Lubin-Tate tower.}
\label{Topological realizability...}

We now demonstrate the topological realizability of the degenerating Lubin-Tate tower, but there is very little work to do, since the essential calculation is a special case of the results of the influential paper of Hopkins--Kuhn--Ravenel \cite{MR1758754}.
\begin{theorem}\label{realizability thm 1}
Let $m$ be a positive integer. 
Let $\mathbb{G}$ be a one-dimensional formal $\hat{\mathbb{Z}}_p$-module, of positive finite $p$-height $n$, over some algebraic field extension $k$ of $\mathbb{F}_p$. There exists a compatible sequence $\mathcal{O}^{\topsym}_{\lvl p^0},\mathcal{O}^{\topsym}_{\lvl p^1},\mathcal{O}^{\topsym}_{\lvl p^2},\dots$ of presheaves, where for each $m$, $\mathcal{O}^{\topsym}_{\lvl p^m}$ is a presheaf of $E_{\infty}$-ring spectra on the small \'{e}tale site of $\Def(\mathbb{G})^{\degen}_{\lvl p^m}$. These presheaves satisfy the following conditions:
\begin{itemize}
\item For each \'{e}tale open $U$ of $\Def(\mathbb{G})^{\degen}_{\lvl p^m}$, we have an isomorphism of graded rings $\pi_*\mathcal{O}^{\topsym}_{\lvl p^m}(U) \cong \Gamma(U)[u^{\pm 1}]$, with $u$ in degree $0$ and $\Gamma(U)$ concentrated in degree $0$.
\item For each \'{e}tale open $U = \Def(\mathbb{G})^{\degen}_{\lvl p^m} \times_{\Spf W(k)} \Spf W(F)$ with $F$ a separable field extension of $k$, the ring spectrum $\mathcal{O}^{\topsym}_{\lvl p^m}(U)$ is complex oriented, and the formal group law on $\pi_0\mathcal{O}^{\topsym}_{\lvl p^m}(U)$ arising from the complex orientation is equal to the universal formal group law on $\Gamma(U)$.
\item $\mathcal{O}^{\topsym}_{\lvl p^0}$ recovers the presheaf $\mathcal{O}^{\topsym}$ constructed by Goerss-Hopkins theory, described above in \cref{Review of the Lubin-Tate tower.}.
\end{itemize}
\end{theorem}
\begin{proof}
Let $C_{p^m}^n$ be the Cartesian product of $n$ copies of the cyclic group $C_{p^m}$. A well-known result of \cite{MR1758754} establishes that, for any Landweber exact complex oriented cohomology theory $E^*$, the $E$-cohomology $E^*(BC_{p^m}^n)$ is isomorphic to the tensor product
\[ E^*[[X]]/[p]_{\mathbb{F}}(X) \otimes_{E^*} \dots \otimes_{E^*} E^*[[X]]/[p]_{\mathbb{F}}(X)\]
of $n$ copies of $E^*[[X]]/[p^m]_{\mathbb{F}}(X)$, where $[p^m]_{\mathbb{F}}(X)$ is the $p^m$-series of the formal group law $\mathbb{F}$ on $E^*$ associated to the complex orientation of $E^*$.

In particular, consider the Morava-Lubin-Tate $E_{\infty}$-ring spectrum $E(\mathbb{G})$ constructed in \cite{MR2125040}, whose fundamental property is that it represents a Landweber-exact complex-oriented cohomology theory with coefficient ring $E(\mathbb{G})^* \cong \Gamma(\Def(\mathbb{G}))[u^{\pm 1}]$, and such that the formal group law on $E(\mathbb{G})^0$ associated to the complex orientation is the universal deformation $\widetilde{\mathbb{G}}$ of $\mathbb{G}$. Here $u$ is in degree $2$, while $\Gamma(\Def(\mathbb{G}))$ is concentrated in degree zero. 

Observe that 
\begin{itemize}
\item the spectrum $\Sigma^{\infty}_+ B(C_{p^m}^n)$ is a suspension spectrum, hence it has a diagonal map,
\item and $E(\mathbb{G})$ is an $E_{\infty}$-ring.
\end{itemize}
Consequently the function spectrum $F\left( \Sigma^{\infty}_+ B(C_{p^m}^n), E(\mathbb{G})\right)$ is an $E_{\infty}$-ring spectrum.
By the result of Hopkins--Kuhn--Ravenel, we have
\begin{align} 
\nonumber \pi_*F\left( \Sigma^{\infty}_+ B(C_{p^m}^n), E(\mathbb{G})\right) 
  &\cong E^*(B(C_{p^m}^n)) \\
\label{iso 4309}  &\cong \left( \Gamma(\Def(\mathbb{G}))[[X]]/[p^m]_{\widetilde{\mathbb{G}}}(X)\otimes_{\Gamma(\Def(\mathbb{G}))} \dots\right. \\ 
\nonumber &\ \ \ \ \ \left. \dots \otimes_{\Gamma(\Def(\mathbb{G}))}\Gamma(\Def(\mathbb{G}))[[X]]/[p^m]_{\widetilde{\mathbb{G}}}(X)\right)[u^{\pm 1}].
\end{align}
By \eqref{iso 4309} and observation \#6 in the list of observations on degenerating level structures in \cref{The degenerating Lubin-Tate tower.}, the graded $E(\mathbb{G})^*$-algebra \eqref{iso 4309} is isomorphic to $\Gamma(\Def^{\degen}_{\lvl p^m})$.

Hence $F\left( \Sigma^{\infty}_+ B(C_{p^m}^n), E(\mathbb{G})\right)$ is an even-periodic $E_{\infty}$ $E(\mathbb{G})$-algebra whose $\pi_0$ agrees with $\Gamma(\Def(\mathbb{G})^{\degen}_{\lvl p^m})$. For each affine \'{e}tale open $U$ of $\Gamma\left(\Def(\mathbb{G})^{\degen}_{\lvl p^m}\right)$, let $\mathcal{O}^{\topsym}_{\lvl p^m}(U)$ be the unique-up-to-contractible-choice $E_{\infty}$-ring $E(\mathbb{G})$-algebra such that $\pi_0\left(\mathcal{O}^{\topsym}_{\lvl p^m}(U)\right) \cong \Gamma(U)$ as $\Gamma(\Def(\mathbb{G})^{\degen}_{\lvl p^m})$-algebras and such that
\begin{align*} \pi_*\left(\mathcal{O}^{\topsym}_{\lvl p^m}(U)\right) &\cong \Gamma(U)\otimes_{\Gamma(\Def(\mathbb{G})^{\degen}_{\lvl p^m})}\pi_*F\left( \Sigma^{\infty}_+ B(C_{p^m}^n), E(\mathbb{G})\right).\end{align*}
\end{proof}

\subsection{Review of the role of cohomology of the LT tower in Langlands correspondences.}
\label{Review of the role...}

Given a formal scheme $X$ with ideal of definition containing the prime $p$, and given a prime $\ell\neq p$, we will write $H_{\et}^i(X;\mathbb{Q}_{\ell})$ for $\mathbb{Q}_{\ell}\otimes_{\hat{\mathbb{Z}}_{\ell}}\lim_{j\rightarrow\infty}H_{\et}^i(X;\mathbb{Z}/\ell^j\mathbb{Z})$. The cohomology of the rigid analytic generic fiber of $X$, with constructible torsion coefficients of torsion order prime to $p$, agrees with the \'{e}tale cohomology of $X$ with the same coefficients (see Lemma 2.1 of \cite{MR1262943}), so it does no harm to confuse $X$ and its generic fibre when we write $H_{\et}^i(X;\mathbb{Q}_{\ell})$. We also write $H_{c}^i(X;\mathbb{Q}_{\ell})$ for $\mathbb{Q}_{\ell}\otimes_{\hat{\mathbb{Z}}_{\ell}}\lim_{j\rightarrow\infty}H_{c}^i(X;\mathbb{Z}/\ell^j\mathbb{Z})$, where $H_{c}^i(X;\mathbb{Z}/\ell^j\mathbb{Z})$ is the cohomology {\em with compact support} of the rigid analytic generic fibre of $X$. Similarly, $H_{\et}^i(X;\overline{\mathbb{Q}}_{\ell})$ and $H_{c}^i(X;\overline{\mathbb{Q}}_{\ell})$ will denote $H_{\et}^i(X;\mathbb{Q}_{\ell})\otimes_{\mathbb{Q}_{\ell}}\overline{\mathbb{Q}}_{\ell}$ and $H_{c}^i(X;\mathbb{Q}_{\ell})\otimes_{\mathbb{Q}_{\ell}}\overline{\mathbb{Q}}_{\ell}$, respectively.

Let $\mathbb{G}$ be a one-dimensional formal $A$-module, of $A$-height $n$, over the algebraic closure of the residue field $A/\varpi$ of $A$. Write $F$ for the field of fractions of $A$. The Galois group of the cover $\Def(\mathbb{G})_{\lvl\varpi^m} \rightarrow \Def(\mathbb{G})$ is isomorphic to $GL_n(A/\varpi^m)$, so in the limit, $GL_n(A)$ acts on $\colim_{m\rightarrow\infty}H^*_{\et}(\Def(\mathbb{G})_{\lvl\varpi^m};\mathbb{Q}_{\ell})$. The automorphism group $\Aut(\mathbb{G})$ also acts on each stage of the classical Lubin-Tate tower, hence also on $\colim_{m\rightarrow\infty}H^*_{\et}(\Def(\mathbb{G})_{\lvl\varpi^m};\mathbb{Q}_{\ell})$. We may incorporate an action of the Weil group\footnote{A nice textbook treatment of the Weil group can be found in \cite{MR2392026}. Its one-sentence characterization is as follows: $W_F$ is the preimage of $\mathbb{Z}\subseteq \hat{\mathbb{Z}}$ under the canonical map \[ \Gal(\overline{F}/F) \rightarrow \Gal\left(\overline{(A/\varpi)}/(A/\varpi)\right)\stackrel{\cong}{\longrightarrow}\hat{\mathbb{Z}}.\]} by extending the base field of the generic fibre $\Def(\mathbb{G})_{\lvl\varpi^m}$ before taking cohomology, so that $\Gal(\overline{F}/F^{nr})$ acts on $\Def(\mathbb{G})_{\lvl\varpi^m} \times_{F^{nr}} \overline{F}$ and hence on its cohomology. This action of $\Gal(\overline{F}/F^{nr})$ extends to an action of $W_F$ (see Proposition 2.3.2 of \cite{MR1719811}), so that $GL_n(A)\times \Aut(\mathbb{G})\times W_F$ acts on $\colim_m H^*_{\et}(\Def(\mathbb{G})_{\lvl\varpi^m}\times_{F^{nr}} \overline{F};\mathbb{Q}_{\ell})$ and on $\colim_m H^*_{c}(\Def(\mathbb{G})_{\lvl\varpi^m}\times_{F^{nr}} \overline{F};\mathbb{Q}_{\ell})$. In fact a bit more is true: the automorphism group $\Aut(\mathbb{G})$ is isomorphic to the unit group of the maximal order $\mathcal{O}_{D_{1/n,F}}$ of the central division $F$-algebra of Hasse invariant $1/n$, and the action of $GL_n(A)\times \Aut(\mathbb{G})\times W_F$ on $\colim_m H^*_{\et}(\Def(\mathbb{G})_{\lvl\varpi^m}\times_{F^{nr}} \overline{F};\mathbb{Q}_{\ell})$ and on $\colim_m H^*_{c}(\Def(\mathbb{G})_{\lvl\varpi^m}\times_{F^{nr}} \overline{F};\mathbb{Q}_{\ell})$
extends to an action of the subgroup of $GL_n(F)\times D^{\times}_{1/n,F} \times W_F$ consisting of those triples $(M,x,\sigma)$ such that $\left|\det M\right|$ is equal to the product of $\left|\det x\right|$ with the norm of the image of $\sigma$ under the composite of the projection $W_F\rightarrow W_F^{\ab}$ with the inverse of the Artin map $F^{\times}\stackrel{\cong}{\longrightarrow} W_F^{ab}$. (To be clear, we will not need to make use of all of this structure in this paper: in the new results in this paper we shall only have need to use the action of $GL_n(A)\times \Aut(\mathbb{G})\times W_F$, and in fact we will almost always leave off the Weil group factor, for simplicity, and only work with $GL_n(A)\times \Aut(\mathbb{G})$.)

The upshot is that, given a $\overline{\mathbb{Q}}_{\ell}$-linear representation $\rho$ of $\Aut(\mathbb{G})$, 
for each integer $i$ we get two $\overline{\mathbb{Q}}_{\ell}$-linear representations of $GL_n(A)\times W_F$:
\begin{align*}
 \hom_{\overline{\mathbb{Q}}_{\ell}[\Aut(\mathbb{G})]}\left( \rho , \colim_m H^i_{\et}(\Def(\mathbb{G})_{\lvl\varpi^m}\times_{F^{nr}}\overline{F};\overline{\mathbb{Q}}_{\ell})\right)\mbox{\ \ and} \\
 \hom_{\overline{\mathbb{Q}}_{\ell}[\Aut(\mathbb{G})]}\left( \rho , \colim_m H^i_{c}(\Def(\mathbb{G})_{\lvl\varpi^m}\times_{F^{nr}}\overline{F};\overline{\mathbb{Q}}_{\ell})\right).
\end{align*}
Conversely, given a $\overline{\mathbb{Q}}_{\ell}$-linear representation $\pi$ of $GL_n(A)$, we get 
$\overline{\mathbb{Q}}_{\ell}$-linear representations of $\Aut(\mathbb{G})\times W_F$:
\begin{align}
\nonumber \hom_{\overline{\mathbb{Q}}_{\ell}[GL_n(A)]}\left( \colim_m H^i_{\et}(\Def(\mathbb{G})_{\lvl\varpi^m}\times_{F^{nr}}\overline{F};\overline{\mathbb{Q}}_{\ell}),\pi\right)\mbox{\ \ and} \\
\label{dual of cont coh 1} \hom_{\overline{\mathbb{Q}}_{\ell}[GL_n(A)]}\left( \colim_m H^i_{c}(\Def(\mathbb{G})_{\lvl\varpi^m}\times_{F^{nr}}\overline{F};\overline{\mathbb{Q}}_{\ell}),\pi\right).
\end{align}

If $\pi$ is irreducible and supercuspidal, then the representation \eqref{dual of cont coh 1} of $D^{\times}_{1/n,F}$ is known \cite{MR1876802} to be finite-dimensional, smooth, and to satisfy the relation
\begin{align*}
 \sum_{i\geq 0} (-1)^{i+n-1}\hom_{\overline{\mathbb{Q}}_{\ell}[GL_n(F)]}\left( \colim_m H^i_{c}(\Def(\mathbb{G})_{\lvl\varpi^m}\times_{F^{nr}}\overline{F};\overline{\mathbb{Q}}_{\ell}),\pi\right)
 &= n\cdot \mathcal{JL}(\pi)
\end{align*}
in the Grothendieck group of admissible representations of $D^{\times}_{1/n,F}$. Here $\mathcal{JL}(\pi)$ denotes the representation of $D^{\times}_{1/n,F}$ associated to $\pi$ by the local Jacquet-Langlands correspondence (\cite{MR0771672}, \cite{MR0700135}, in the $n=2$ case \cite{MR0401654}).
In this paragraph we have largely ignored the action of $W_F$, but if one considers the analogue of \eqref{dual of cont coh 1} with $D^{\times}_{1/n,F}$ replaced by $W_F$, one realizes the local Langlands correspondence rather than the local Jacquet-Langlands correspondence. 

We have sketched only the handful of ideas most immediately relevant for this paper, and this sketch gives only a very incomplete picture of the full story told by \cite{MR1876802}. 

We recall one further development of the theory. In 4.2.7 of \cite{MR2141708}, ``Hypothesis (H)'' is stated: it is a mild, but unproven, assumption of finite-dimensionality of the $\ell$-adic cohomology of the strata in Strauch's stratification\footnote{To avoid any possible misunderstanding, we point out that Strauch's stratification of \cite{MR2141708} is not the same as the stratification of $\Def(\mathbb{G})^{\degen}_{\lvl \varpi^m}$ by degeneration type. Indeed, the two stratifications are not defined on the same spaces.} of the ``mod $\varpi$ boundary'' of the generic fibre of each space in the Lubin-Tate tower.

Both the \'{e}tale and the compactly-supported cohomology of $\Def(\mathbb{G})_{\lvl\varpi^m}\times_{F^{nr}}\overline{F}$ are concentrated in degrees $\leq 2n-2$. More specifically, Lemma 2.5.1 of \cite{MR2141708} establishes that 
\[ H^{i}_{\et}(\Def(\mathbb{G})_{\lvl\varpi^m}\times_{F^{nr}}\overline{F};\overline{\mathbb{Q}}_{\ell})\]
is finite-dimensional for each $m$, and trivial if $i>n-1$, while
\[ H^{i}_{c}(\Def(\mathbb{G})_{\lvl\varpi^m}\times_{F^{nr}}\overline{F};\overline{\mathbb{Q}}_{\ell})\]
is finite-dimensional for each $m$, and trivial if $i<n-1$ or $i>2n-2$.

In Theorem 2.5.2 of \cite{MR2141708} (conjectured by Carayol in \cite{MR1044827}), it is proven conditionally that the cohomology in the middle degree plays a special role:
\begin{theorem} \label{strauch thm}
Let $\ell$ be a prime distinct from the characteristic $p$ of $A/\varpi$, and let $\pi$ be an irreducible supercuspidal representation of $GL_n(F)$ which is realizable over $\overline{\mathbb{Q}}_{\ell}$. Suppose that {\em Hypothesis (H)} holds. Then, for all $i\neq n-1$, \[ \hom_{\overline{\mathbb{Q}}_{\ell}[GL_n(F)]}\left( \colim_m H^i_{c}(\Def(\mathbb{G})_{\lvl\varpi^m}\times_{F^{nr}}\overline{F};\overline{\mathbb{Q}}_{\ell}),\pi\right)\] is trivial. Furthermore, in the middle degree, we have
\begin{align*}
 \hom_{\overline{\mathbb{Q}}_{\ell}[GL_n(F)]}\left( \colim_m H^{n-1}_{c}(\Def(\mathbb{G})_{\lvl\varpi^m}\times_{F^{nr}}\overline{F};\overline{\mathbb{Q}}_{\ell}),\pi\right)
  & \cong \mathcal{JL}(\pi)^{\oplus n}.
\end{align*}
\end{theorem}
That is, {\em the supercuspidal irreducible representations of $GL_n(F)$ all occur in the middle-degree cohomology of the Lubin-Tate tower}. In the other degrees, one can find principal series and Steinberg-like representations instead.

Today Theorem \ref{strauch thm} is known unconditionally: it was proven by Mieda in \cite{MR2680204}, by a different method, which avoids the assumption of Hypothesis (H).

\subsection{Partial Drinfeld level structures and the cohomology of the degenerating LT tower.}

Proposition \ref{open and closed subset} yields an equivariant splitting in cohomology:
\begin{theorem}\label{coh splitting thm}
Let $\mathbb{G}$ be a one-dimensional formal $A$-module, of $A$-height $n$, over some field extension $k$ of the residue field $A/\varpi$ of $A$. 
Let $m$ be a positive integer. Then, for each integer $i$, the $\mathbb{Q}_{\ell}$-linear representation $H^i_{\et}(\Def(\mathbb{G})_{\lvl p^m};\mathbb{Q}_{\ell})$ of $GL_n(A/\varpi^m)\times \Aut(\mathbb{G})$ is a summand in the $\mathbb{Q}_{\ell}$-linear representation $H^i_{\et}(\Def(\mathbb{G})_{\lvl p^m}^{\degen};\mathbb{Q}_{\ell})$ of $GL_n(A/\varpi^m)\times \Aut(\mathbb{G})$.

Similarly:
\begin{itemize}
\item the $\mathbb{Q}_{\ell}$-linear representation $H^i_{c}(\Def(\mathbb{G})_{\lvl p^m};\mathbb{Q}_{\ell})$ of $GL_n(A/\varpi^m)\times \Aut(\mathbb{G})$ is a summand in the $\mathbb{Q}_{\ell}$-linear representation $H^i_{c}(\Def(\mathbb{G})_{\lvl p^m}^{\degen};\mathbb{Q}_{\ell})$ of $GL_n(A/\varpi^m)\times \Aut(\mathbb{G})$,
\item the $\mathbb{Q}_{\ell}$-linear representation $H^i_{\et}(\Def(\mathbb{G})_{\lvl p^m}\times_{F^{nr}} \overline{F};\mathbb{Q}_{\ell})$ of $GL_n(A/\varpi^m)\times \Aut(\mathbb{G})\times W_F$ is a summand in the $\mathbb{Q}_{\ell}$-linear representation $H^i_{\et}(\Def(\mathbb{G})_{\lvl p^m}^{\degen}\times_{F^{nr}} \overline{F};\mathbb{Q}_{\ell})$ of $GL_n(A/\varpi^m)\times \Aut(\mathbb{G})\times W_F$,
\item and the $\mathbb{Q}_{\ell}$-linear representation $H^i_{c}(\Def(\mathbb{G})_{\lvl p^m}\times_{F^{nr}} \overline{F};\mathbb{Q}_{\ell})$ of $GL_n(A/\varpi^m)\times \Aut(\mathbb{G})\times W_F$ is a summand in the $\mathbb{Q}_{\ell}$-linear representation $H^i_{c}(\Def(\mathbb{G})_{\lvl p^m}^{\degen}\times_{F^{nr}} \overline{F};\mathbb{Q}_{\ell})$ of $GL_n(A/\varpi^m)\times \Aut(\mathbb{G})\times W_F$.
\end{itemize}
\end{theorem}

\begin{example}\label{ht 1 example 2}
Consider the case of Theorem \ref{coh splitting thm} in which $\mathbb{G}$ has $A$-height $1$. The $A$-submodules of $\varpi^{-m}A/A$ all fit into the sequence
\[ \varpi^{-m}A/A \supseteq \varpi^{-m+1}A/A \supseteq \dots \supseteq \varpi^{-1}A/A \supseteq 0.\]
In Example \ref{ht 1 example 1}, we wrote $S_j$ for the complement of $\varpi^{-j}A/A$ in $\varpi^{-j-1}A/A$, and we showed that the complement of $\Def(\mathbb{G})^{\degen S_{j+1}}_{\lvl \varpi^m}$ in $\Def(\mathbb{G})^{\degen S_j}_{\lvl \varpi^m}$ is isomorphic to $\Def(\mathbb{G})_{\lvl \varpi^j}$. 
Consequently a strengthening of Theorem \ref{coh splitting thm} holds: not only is the cohomology of $\Def(\mathbb{G})_{\lvl \varpi^m}$ a summand in the cohomology of $\Def(\mathbb{G})^{\degen}_{\lvl \varpi^m}$, it is even true that $\Def(\mathbb{G})^{\degen}_{\lvl \varpi^m}$ splits equivariantly as the direct sum of the cohomologies of the formal schemes $\Def(\mathbb{G})_{\lvl \varpi^j}$ for each $j\in \{ 0, \dots ,m\}$.
\end{example}

\begin{example}\label{ht 2 example}
Now consider the case of Theorem \ref{coh splitting thm} in which $\mathbb{G}$ has $A$-height $2$. The $A$-submodules of $(\varpi^{-1}A/A)^2$ are of three kinds:
\begin{enumerate}
\item the zero module,
\item $A/\varpi$-rank $1$ submodules of $(\varpi^{-1}A/A)^2$, which are in bijection with the (geometric) points in the projective space $\mathbb{P}^1_{A/\varpi}$,
\item and $(\varpi^{-1}A/A)^2$ itself.
\end{enumerate}
This yields a decomposition of $\Def(\mathbb{G})^{\degen}_{\lvl\varpi^1}$ as follows: 
\begin{description}
\item[Rank $2$ degeneration] We will use a notation introduced in \cref{The degenerating Lubin-Tate tower.}: for each rank $A/\varpi$-linear subspace $V$ of $(A/\varpi)^2$, write $\cancel{V}$ for the complement of $V$ in $(A/\varpi)^2$.

The complement of $\bigcup_{\rank_{A/\varpi}(V)=1} \Def(\mathbb{G})^{\degen \cancel{V}}_{\lvl \varpi^1}$ in $\Def(\mathbb{G})^{\degen}_{\lvl\varpi^1}$ classifies only the trivial degenerating level $\varpi^1$-structure, i.e., 
\begin{align*}
 \Def(\mathbb{G})^{\degen}_{\lvl\varpi^1} \backslash\bigcup_{\rank_{A/\varpi}(V)=1} \Def(\mathbb{G})^{\degen \cancel{V}}_{\lvl \varpi^1} 
  &=\Def(\mathbb{G}),\end{align*}
the classical Lubin-Tate space without level structures.
\item[Rank $1$ degeneration] For each $A/\varpi$-rank $1$ $A/\varpi$-linear subspace $V$ of \linebreak $(\varpi^{-1}A/A)^2$, the complement of $\Def(\mathbb{G})_{\lvl \varpi^1}$ in $\Def(\mathbb{G})^{\degen \cancel{V}}_{\lvl\varpi^1}$ classifies deformations of $\mathbb{G}$ equipped with a degenerating level $\varpi^1$-structure which degenerates in $V$ and {\em only} in $V$. Up to isomorphism, this does not depend on the choice of $V$: it is equivalent to giving a deformation $\widetilde{\mathbb{G}}$ of $V$ equipped with a map \begin{align} \label{map 304224}\phi: \varpi^{-1}A/A &\rightarrow \widetilde{\mathbb{G}}(\mathfrak{m}_{\widetilde{\mathbb{G}}})\end{align} satisfying the divisibility condition on all elements of $\varpi^{-1}A/A$. Note that the domain of \eqref{map 304224} is $\varpi^{-1}A/A$, not $(\varpi^{-1}A/A)^2$!

Drinfeld, in \cite{MR0384707}, described the space classifying deformations of $\mathbb{G}$ equipped with a map of the form \eqref{map 304224} satisfying the stated divisibility condition: it is the formal scheme $\Spf L_1$, where $L_1$ is the ring 
\[\Gamma(\Def(\mathbb{G}))[[\theta]]/\left( [\varpi]_{\widetilde{\mathbb{G}}}(\theta)/\theta\right).\]
\item[Rank $0$ degeneration type] The intersection of the subspaces $\Def(\mathbb{G})^{\degen \cancel{V}}_{\lvl\varpi^1}$ of $\Def(\mathbb{G})^{\degen}_{\lvl\varpi^1}$, across all $A/\varpi$-rank $1$ $A/\varpi$-linear subspaces $V$ of $(\varpi^{-1}A/A)^2$, is $\Def(\mathbb{G})_{\lvl \varpi^1}$, the classical Lubin-Tate space with Drinfeld level $\varpi^1$-structures.
\end{description}
Consequently the $\ell$-adic \'{e}tale cohomology of $\Def(\mathbb{G})^{\degen}_{\lvl\varpi^1}$ splits $GL_2(A/\varpi)\times \Aut(\mathbb{G})$-equivariantly as follows:
\begin{align*}
 H^i_{\et}\left(\Def(\mathbb{G})^{\degen}_{\lvl\varpi^1}; \mathbb{Q}_{\ell}\right)
  &\cong H^i_{\et}\left(\Def(\mathbb{G}); \mathbb{Q}_{\ell}\right) \\
  &\ \ \ \ \ \oplus \bigoplus_{x\in \mathbb{P}^1_{A/\varpi}}H^i_{\et}\left(\Spf L_1; \mathbb{Q}_{\ell}\right)\\
  &\ \ \ \ \ \oplus H^i_{\et}\left( \Def(\mathbb{G})_{\lvl \varpi^1}; \mathbb{Q}_{\ell}\right)
\end{align*}
More importantly for our purposes, a similar decomposition holds for compactly supported cohomology $H^i_c$ in place of \'{e}tale cohomology. It also holds ($W_F$-equivariantly) for the cohomology of these deformation spaces base-changed to $\overline{F}$. 

It is not difficult to extend this decomposition of the cohomology of $\Def(\mathbb{G})^{\degen}_{\lvl\varpi^1}$ to higher level and higher heights. We carry this out in Theorem \ref{generic fibre decomp}. Our argument for Theorem \ref{generic fibre decomp} will require some use of ``partial Drinfeld level structures,'' which we now define.
\end{example}

\begin{definition}\label{def of classical type and cyclic type}
Let $\mathbb{G}$ be a one-dimensional formal $A$-module.
\begin{itemize}
\item Suppose $m$ is a nonnegative integer.
By a {\em deformation of $\mathbb{G}$ with partial Drinfeld level $\varpi^m$ structure} we mean a deformation $\widetilde{\mathbb{G}}$ of $\mathbb{G}$ together with an $A$-module homomorphism $D \rightarrow \mathbb{G}(\mathfrak{m}_B)$, for some $A$-submodule $D$ of $\left( \varpi^{-m}A/A\right)^n$, such that
\begin{align*}
\prod_{a\in D} \left( X - \phi(a)\right)\mbox{\ divides\ } [\varpi]_{\mathbb{G}}(X).
\end{align*}
We refer to $D$ as the {\em domain} of the partial level structure.
\item 
It is not difficult to see that, for a fixed choice of $A$-submodule $D$ of $\left( \varpi^{-m}A/A\right)^n$, there exists a formal scheme $\Def(\mathbb{G})^{\partial D}_{\lvl \varpi^m}$ of deformations of $\mathbb{G}$ equipped with partial Drinfeld level $\varpi^m$ structure. By a {\em partial Lubin-Tate space} we will mean a formal scheme of the form $\Def(\mathbb{G})^{\partial D}_{\lvl \varpi^m}$ for some nonnegative integer $m$, some formal $A$-module $\mathbb{G}$ of some finite height $n$, and some $A$-submodule $D$ of $\left( \varpi^{-m}A/A\right)^n$.
\item By contrast, we refer to a formal scheme of the form $\Def(\mathbb{G})_{\lvl \varpi^m}$, for any nonnegative integer $m$ and formal $A$-module of finite height $n$, as a {\em Lubin-Tate space of classical type}.
\end{itemize}
\end{definition}

\begin{lemma}\label{partial lubin-tate decomp lemma}
Every partial Lubin-Tate space of level $\varpi^1$ is finite and flat over a Lubin-Tate space of classical type. Furthermore, after inverting $\varpi$, every partial Lubin-Tate space of level $\varpi^1$ is \'{e}tale over a Lubin-Tate space of classical type.
\end{lemma}
\begin{proof}
This is really a result of Drinfeld, expressed in only slightly different language. 
Every $A$-submodule of $(\varpi^{-1}A/A)^n$ is isomorphic to a direct sum of copies of $A/\varpi$. For concreteness, suppose $D = (\varpi^{-1}A/A)^d \subseteq (\varpi^{-1}A/A)^n$, with $d\leq n$. Then $\Def(\mathbb{G})^{\partial D}_{\lvl \varpi^1}$ is described by a result of Drinfeld, from Proposition 4.3 of \cite{MR0384707}:
\begin{align*}
 \Def(\mathbb{G})^{\partial D}_{\lvl \varpi^1}
  &\cong \Spf L_d,\end{align*}
where $L_d$ is constructed inductively: $L_0 := \Gamma(\Def(\mathbb{G})$, while \begin{align*} L_j &:= L_{j-1}[[\theta]]/\left( [\varpi]_{\widetilde{\mathbb{G}}}(\theta)/\prod_{a\in (\varpi^{-1}A/A)^{j-1}}(\theta - \phi_{j-1}(a))\right)\end{align*}
for $j\geq 1$.
Here $\phi_{j-1}$ is the universal partial Drinfeld level $\varpi^1$-structure with domain $(\varpi^{-1}A/A)^{j-1}$, classified by $\Spf L_{j-1}$. Finiteness and flatness of each of the maps $L_0 \rightarrow L_1 \rightarrow\dots \rightarrow L_d$ is proven by Drinfeld in Proposition 4.3 of \cite{MR0384707}. It is straightforward to see that the maps are consequently \'{e}tale after inverting $\varpi$.
\end{proof}

\begin{prop}\label{generic fibre decomp}
Let $m$ be a positive integer. 
Let $\mathbb{G}$ be a one-dimensional formal $A$-module, of finite $A$-height, over some field extension $k$ of the residue field $A/\varpi$ of $A$. Then the generic fibre of the degenerating Lubin-Tate space $\Def(\mathbb{G})^{\degen}_{\lvl \varpi^m}$ decomposes as a disjoint union of finitely many regular quasi-affine\footnote{``Quasi-affine'' is defined in Definition 5.5 of \cite{MR1395723}: a rigid analytic space $X$ is {\em quasi-affine} if each compact subset of $X$ is contained in an affinoid domain which is isomorphic to a rational domain in the analytification of an affine scheme of finite type over an appropriate base. 

For the purposes of this paper, we suggest taking the notion of ``quasi-affineness'' as a black box. Quasi-affineness is used in this paper in only two places: in Theorem \ref{generic fibre decomp}, we show that degenerating Lubin-Tate spaces are quasi-affine; and in Theorem \ref{main thm}, we apply the quasi-affineness of degenerating Lubin-Tate spaces in order to identify the abutment of the modified vanishing cycles stalk spectral sequence. For the former argument, we use Berkovich's Corollary 5.6 of \cite{MR1395723}, a tool for showing that certain rigid analytic spaces are quasi-affine, which we can deploy in a black box fashion.} rigid analytic spaces.

Furthermore, if $m=1$, then the generic fibre of $\Def(\mathbb{G})^{\degen}_{\lvl \varpi^m}$ decomposes as a disjoint union of the generic fibres of partial Lubin-Tate spaces.
\end{prop}
\begin{proof}
We begin by proving the second claim, before moving on to the first claim.
Suppose that $m=1$.  
Given an integer $d \in \{ 1, \dots ,n\}$, choose a $d$-dimensional $A/\varpi$-linear subspace $V$ of $(\varpi^{-1}A/A)^n$. Consider the $(d-1)$-dimensional $A/\varpi$-linear subspaces $W$ of $V$. The complement of $\Def(\mathbb{G})^{\degen \cancel{W}}_{\lvl \varpi^1}$ in $\Def(\mathbb{G})^{\degen \cancel{V}}_{\lvl \varpi^1}$ classifies deformations of $\mathbb{G}$ with level structure which is nondegenerate outside of $V$, and degenerate somewhere in $W$. Taking the union over $W$: the complement of \[ \bigcup_{\rank_{A/\varpi}W=d-1}\Def(\mathbb{G})^{\degen \cancel{W}}_{\lvl \varpi^1}\] in $\Def(\mathbb{G})^{\degen \cancel{V}}_{\lvl \varpi^1}$ classifies deformations of $\mathbb{G}$ with level structure which is nondegenerate outside of $V$, and degenerate {\em everywhere} in $V$. Such a level structure is uniquely determined by giving a partial level structure with domain of $A/\varpi$-linear dimension $n-d$, i.e.,
\begin{align*}
 \Def(\mathbb{G})^{\degen \cancel{V}}_{\lvl \varpi^1} \backslash \left( \bigcup_{\rank_{A/\varpi}W=d-1}\Def(\mathbb{G})^{\degen \cancel{W}}_{\lvl \varpi^1}\right) 
 &\cong \Def(\mathbb{G})^{\partial \left( (\varpi^{-1}A/A)^{n-d}\right)}_{\lvl \varpi^1}.
\end{align*}
Consequently, after inverting $\varpi$, by Proposition \ref{open and closed subset} $\Def(\mathbb{G})^{\degen}_{\lvl \varpi^1}$ decomposes as follows:
\begin{align}
\label{level 1 grassmannian decomp} \Def(\mathbb{G})^{\degen}_{\lvl \varpi^1}
  &\cong \ \coprod_{d=0, \dots ,n}\ \left( \coprod_{V\in \Gr_d\left( (\varpi^{-1}A/A)^n\right)} \Def(\mathbb{G})^{\partial (\varpi^{-1}A/A)^{n-d}}_{\lvl \varpi^1}\right),
\end{align}
where $\Gr_d\left( (\varpi^{-1}A/A)^n\right)$ is the Grassmannian, i.e., the set of $d$-dimensional $A/\varpi$-linear subspaces of $(\varpi^{-1}A/A)^n$.

The right-hand side of \eqref{level 1 grassmannian decomp} is a disjoint union of partial Lubin-Tate spaces. This proves the second claim.
By Lemma \ref{partial lubin-tate decomp lemma}, the generic fibre of each of those partial Lubin-Tate spaces is finite, \'{e}tale, and regular over the generic fibre of some Lubin-Tate of classical type. Lubin-Tate spaces of classical type are known to be quasi-affine (e.g. see the proof of Lemma 2.5.1 of \cite{MR2383890}). By Corollary 5.6 of \cite{MR1395723}, quasi-affineness is inherited by finite \'{e}tale covers of a quasi-affine rigid analytic space, so partial Lubin-Tate spaces are quasi-affine, so the generic fibre of $\Def(\mathbb{G})^{\degen}_{\lvl \varpi^1}$ is quasi-affine. This proves the first claim in the case $m=1$.

The first claim is proven by induction on $m$. In case $m>1$, the inductive step is proven by much the same argument as that of part (b) of the lemma following Proposition 4.3 of \cite{MR0384707}: $\Def(\mathbb{G})^{\degen}_{\lvl \varpi^m}$ coincides with the formal spectrum of 
\begin{eqnarray*} \Gamma(\Def(\mathbb{G})^{\degen}_{\lvl \varpi^{m-1}})[[\theta_1, \dots \theta_n]]/\left( [\varpi]_{\widetilde{\mathbb{G}}_{m-1}}(\theta_1)-\eta_{m-1}(e_1),  \right.\ \ \ \ \ \ \  \\ \left.[\varpi]_{\widetilde{\mathbb{G}}_{m-1}}(\theta_2)-\eta_{m-1}(e_2),\dots ,[\varpi]_{\widetilde{\mathbb{G}}_{m-1}}(\theta_n)-\eta_{m-1}(e_n) \right),\end{eqnarray*}
where $e_1, \dots ,e_n$ is a minimal set of generators for the $A$-module $(\varpi^{-(m-1)}A/A)^n$, and where $\eta_{m-1}: (\varpi^{-(m-1)}A/A)^n \rightarrow \mathbb{G}_{m-1}(\mathfrak{m}_{\mathbb{G}_{m-1}})$ is the universal level $\varpi^{m-1}$ structure defined on $\Gamma\left(\Def(\mathbb{G})^{\degen}_{\lvl \varpi^{m-1}}\right)$. If we have already decomposed the generic fibre of $\Def(\mathbb{G})^{\degen}_{\lvl \varpi^m}$ as a disjoint union of the generic fibre of finitely many affine formal schemes $X$, each of which has generic fibre which is finite, \'{e}tale, and regular over that of $\Def(\mathbb{G})$, then \begin{eqnarray*}\Spf \Gamma(X)[[\theta_1, \dots \theta_n]]/\left( [\varpi]_{\widetilde{\mathbb{G}}_X}(\theta_1)-\eta_{X}(e_1), \right. \ \ \ \ \ \ \ \ \\ \left. [\varpi]_{\widetilde{\mathbb{G}}_X}(\theta_2)-\eta_{X}(e_2), \dots ,[\varpi]_{\widetilde{\mathbb{G}}_X}(\theta_n)-\eta_{X}(e_n) \right)\end{eqnarray*} also has generic fibre which is finite, \'{e}tale, and regular over that of $\Def(\mathbb{G})$, completing the inductive step.
\end{proof}



\def\cprime{$'$} \def\cprime{$'$} \def\cprime{$'$} \def\cprime{$'$}

\end{document}